 \RequirePackage{fix-cm}
    \documentclass[smallextended,draft]{svjour3}       % onecolumn (second format)
    \smartqed  % flush right qed marks, e.g. at end of proof
    \usepackage{graphicx}
    \usepackage{mathptmx}      % use Times fonts if available on your TeX system
    \usepackage{bbm}
    \usepackage{bm}
    \usepackage{amsmath}
    \usepackage{enumerate}
    \usepackage{cancel}
    \usepackage{color}

\newcommand{\DG}{\mathrm{DG}}
\newcommand{\mesh}{\mathcal{E}}
\newcommand{\aelip}{a_{\mathrm{ellip}}}

\newcommand{\err}{\bm{\tilde{e}}_h^n}
\newcommand{\errn}{\bm{\tilde{e}}_h^{n+1}}
\newcommand{\Bk}{\color{black}}

\newcommand{\bfX}{\bm X}
 
\newcommand{\vertiii}[1]{{\left\vert\kern-0.25ex\left\vert\kern-0.25ex\left\vert #1 
    \right\vert\kern-0.25ex\right\vert\kern-0.25ex\right\vert}}%

\begin{document}
    
    \title{Improved a priori error estimates for a discontinuous Galerkin pressure correction scheme for the Navier-Stokes equations} %\thanks{Grants or other notes
    
    
    
    \author{Rami Masri  \and
            Chen Liu \and 
            Beatrice Riviere}

    \institute{ R. Masri \at 
                Department of Computational and Applied Mathematics, Rice University \\
                %  Tel.: +123-45-678910\\
                \email{rami.masri@rice.edu} \\ \\
                C. Liu \at 
                Department of Mathematics, Purdue University \\
                \email{liu3373@purdue.edu} \\ \\
                B. Riviere \at 
                Department of Computational and Applied Mathematics, Rice University \\
                \email{riviere@rice.edu} \\ \\
              }
    
    \date{\today}
    \titlerunning{DG pressure correction scheme}
    \authorrunning{R.Masri, C.Liu, B.Riviere}
    \maketitle

    \begin{abstract}
The pressure correction scheme is combined with interior penalty discontinuous Galerkin method to solve the time-dependent
Navier-Stokes equations.  Optimal error estimates are derived for the velocity in the L$^2$ norm in time and in space. 
Error bounds for the discrete time derivative of the velocity and for the pressure are also established. 
    \end{abstract}
    
    \section{Introduction}\label{sec:intro}
    This paper contains the derivation of a priori error estimates for the velocity and the pressure of a splitting scheme for solving the
incompressible Navier-Stokes equations. This work  follows up \cite{inspaper1} where the scheme is introduced and stability is proved.
In \cite{inspaper1}, a priori error estimates for the velocity in L$^2$ in time and in the broken H$^1$ space are obtained;
these error bounds are optimal in space but suboptimal in time.  The main goal of the current paper is to obtain optimal error
estimates for the velocity in the L$^2$ norm in time and space. After deriving error bounds for the discrete derivative of the velocity, 
we also show error estimates for the pressure that are optimal in space and suboptimal in time.

The type of pressure-correction splitting scheme used in this work, belongs to the class of operator splitting schemes that
decouple the nonlinearity in the momentum equation from the incompressibility constraint.
An overview of such splitting schemes can be found in the work by Guermond, Minev, and Shen \cite{guermond2006overview} (see also the comprehensive text \cite{Glowinski2003}). Splitting schemes are known to be computationally efficient for large scale problems.
%These schemes present a computationally efficient algorithm to simulate flow in large scale problems. 
Error analysis of pressure correction schemes combined with continuous finite element method has been carried out in several papers \cite{guermond1998approximation,guermond1998stability,nochetto2005gauge}. 

While  there are several computational works on discontinuous Galerkin schemes combined with the pressure correction splitting technique 
\cite{krank2017high,PiatkowskiMuthingBastian2018,liu2019interior}, the only theoretical convergence paper in the literature is \cite{inspaper1}. 
In this second paper, we improve upon the error estimate obtained for the velocity by employing a duality argument
and assuming convexity of the domain.  The discrete velocities belong to the space of discontinuous polynomials of degree $k$ and the
discrete pressure and potential belong to the space of discontinuous polynomials of degree $k-1$.  The theoretical rates
of convergence for the velocity are shown to be order one in time and order $k+1$ in space.  For the pressure error bounds, the
rate is $k$ in space and $1/2$ in time, which is consistent with the results obtained for the continuous finite element method
\cite{nochetto2005gauge}. 
%This improvement requires additional assumptions such as the convexity of the domain and the uniformity of the mesh. Stability and error estimates for the discrete time derivative of the velocity are also shown. We utilize these estimate to show convergence of the discrete pressure. 

This paper is organized as follows. Section \ref{sec:model_pb} introduces the model problem and the fully discrete scheme. The improved error estimate for the velocity is obtained in Section \ref{sec:improved_estimate}. We show error estimates for the discrete time derivative of the velocity in Section \ref{sec:time_derr_velocity} and for the pressure in Section \ref{sec:pressure_estimate}. %Conclusions follow. 

    \section{Model problem and fully discrete scheme}\label{sec:model_pb}
    %\subsection{Model problem}
Let $\Omega$ be an open bounded and convex polyhedral domain in $\mathbbm{R}^d$ where $d = 2$ or $3$.  %Let $T>0$ be the end time. 
Consider the incompressible Navier-Stokes equations in $\Omega$ over the time interval $[0,T]$. 
%with homogenous Dirichlet boundary condition for the velocity and a zero average constraint for the pressure. 
\begin{alignat}{2}
\partial_t \bm{u}  - \mu \Delta \bm{u}+ \bm{u}\cdot \nabla \bm{u}+ \nabla p &= \bm{f},  && \quad\mathrm{ in} \,\, \Omega \times (0,T], \label{eq:first_eq_NS}\\ 
\nabla \cdot \bm{u} & = 0, &&\quad \mathrm{in} \,\, \Omega \times (0,T], \\  
\bm{u} & = \bm{u}^0,  &&\quad \mathrm{in} \,\, \Omega \times \{0\}, \\ 
\bm{u} & = \bm{0},&& \quad \mathrm{on} \,\, \partial \Omega \times (0,T], 
\label{eq:dirichlet_bc}   \\ 
\int_{\Omega} p(t) & =0, && \quad \forall t \in (0,T]. \label{eq:zero_avg}
\end{alignat}
In the above system, $\bm{u}$ is the fluid velocity, $p$ is the pressure, $\bm{f}$ is the external force, and $\mu>0$ is the viscosity. 
For a given non-negative integer $m$ and real number $r \geq 1$, the Sobolev space  $W^{m,r}(\mathcal{O})$ on a domain $\mathcal{O} \subset \mathbbm{R}^d$ is equipped with the usual Sobolev norms and semi-norms  $ \| \cdot \|_{W^{m,r}(\mathcal{O})} $ and $| \cdot |_{W^{m,r}(\mathcal{O})}$ respectively. If $r=2$, we denote $H^m(\mathcal{O}) = W^{m,r}(\mathcal{O})$,  $\| \cdot \|_{H^m(\mathcal{O})} = \|\cdot\|_{W^{m,2}(\mathcal{O})}$ and $|\cdot |_{H^m(\mathcal{O})} = |\cdot |_{W^{m,2}(\mathcal{O})}$. 
The $L^2$ inner-product over $\Omega$ is denoted by $(\cdot,\cdot)$ and the resulting $L^2$ norm by $\Vert \cdot\Vert$.
We assume that $\bm{u}^0 \in H^1_0(\Omega)^d$ and $\nabla \cdot \bm{u}^0 = 0$. 
%\subsection{Discontinuous Galerkin forms}

Let $\mesh_h = \{E_k\}$ denote a family of regular and uniform partitions of the domain  $\Omega$ \cite{ciarlet2002finite}. 
We define the following broken Sobolev spaces.
\begin{alignat}{3}
    \bm{X} &= \{\bm{v}  \in L^2(\Omega)^d: && \quad \forall E \in \mesh_h,&& \quad \bm{v} \vert_E \in W^{2,4/3}(E)^d  \}, \\ 
    M & = \{ q \in L^2(\Omega): && \quad \forall E \in \mesh_h, && \quad q\vert_E \in W^{1,4/3} (E)\}.
\end{alignat}  
Let $h = \max_{E \in \mesh_h} h_E$ where $h_E   = \mathrm{diam}(E)$ . Denote by $\Gamma_h$ the set of all interior faces of the subdivision $\mesh_h$.  For an interior edge $e \in \Gamma_h$, we associate a normal $\bm{n}_e$ and we denote 
by $E_e^1$ and $E_e^2$ the two elements that share $e$, such that $\bm{n}_e$ points from $E_e^1$ to $E_e^2$. 
Define the average and jump for a function $\bm{\theta} \in \bfX$ as such, 
\begin{align}
\{ \bm{\theta} \}  = \frac12 (\bm{\theta}|_{E_e^1} + \bm{\theta}|_{E_e^2}), \quad 
[\bm{\theta}]   = \bm{\theta}|_{E_e^1} - \bm{\theta}|_{E_e^2}, \quad \forall e = \partial E_e^1 \cap \partial E_e^2. 
\end{align}
For a boundary face, $e \in \partial \Omega$, the vector $ \bm{n}_e$ is chosen as the unit outward vector to $\partial \Omega$. The definition of the average and jump in this case are extended as such, 
\begin{equation} 
\{ \bm{\theta} \} = [\bm{\theta}] = \bm{\theta}|_{E_e}, \quad \forall e = \partial E_e \cap \partial \Omega. 
\end{equation}
Similar definitions are used for scalar valued functions, $q \in M$. For the convection term, we use the same discretization form, $a_\mathcal{C}$, as in \cite{girault2005discontinuous}.  Denote by $\bm{n}_E$ the outward normal to $E$, and denote by $\bm{v}^\mathrm{int}$ and $\bm{v}^\mathrm{ext}$ the trace of a function $\bm{v}$
on the boundary of $E$ coming from the interior (resp. exterior) of $E$. By convention, $\bm{v}^\mathrm{ext}|_e = {\bf 0}$ if $e$ is a boundary face $(e\subset\partial\Omega)$. 
We also introduce the notation for the inflow boundary of $E$ with respect to a function $\bm{z}$: 
\begin{equation}
    \partial E_{-}^{\bm{z}} = \{ \bm{x} \in \partial E: \{\bm{z}  (\bm{x}) \}\cdot \bm{n}_E < 0 \}, \quad \bm{z} \in \bfX.
    \label{eq:inflow_boundary}
\end{equation}
With this notation, we have for $\bm{z}, \bm{w}, \bm{v}, \bm{\theta} \in \bfX$, 
\begin{multline}
a_\mathcal{C}(\bm{z};\bm{w}, \bm{v}, \bm{\theta}) = \sum_{E \in \mesh_h} \left( \int_E (\bm{w} \cdot \nabla \bm{v})\cdot \bm{\theta}  + \frac12 \int_E (\nabla \cdot \bm{w}) \, \bm{v} \cdot \bm{\theta} \right) \\ 
- \frac{1}{2} \sum_{e\in\Gamma_h \cup \partial \Omega} \int_{e} [\bm{w}] \cdot \bm{n}_e \{ \bm{v} \cdot \bm{\theta} \} + \sum_{E \in \mesh_h} \int_{\partial E_{-}^{\bm{z}}} | \{\bm{w}\} \cdot \bm{n}_E| (\bm{v}^{\mathrm{int}} - \bm{v}^{\mathrm{ext}}) \cdot  \bm{\theta}^{\mathrm{int}}.
\end{multline}
The form $a_\mathcal{C}$ admits an ``integration by parts" formula \cite{girault2005discontinuous,girault2005splitting}. To this end, define for $\bm{z}, \bm{w}, \bm{v}, \bm{\theta} \in \bfX$, the form $\bar{a}_{\mathcal{C}}$ as: 
\begin{multline}
\bar{a}_\mathcal{C}(\bm{z}; \bm{w}, \bm{\theta}, \bm{v}) = \sum_{E \in \mesh_h}   \left( \int_E (\bm{w} \cdot \nabla \bm{\theta} ) \cdot \bm{v} + \frac{1}{2}\int_E (\nabla \cdot \bm{w} ) \bm{\theta} \cdot \bm{v} \right)\\ - \frac{1}{2} \sum_{e \in \Gamma_h \cup \partial \Omega} \int_e [\bm{w}] \cdot \bm{n}_e \{\bm{\theta} \cdot \bm{v} \} 
+\sum_{E \in \mesh_h } \int_{(\partial E_{-}^{\bm{z}}) \backslash (\partial \Omega)} |\{ \bm{w}\} \cdot \bm{n}_E  | (\bm{\theta}^{\mathrm{int}} - \bm{\theta}^{\mathrm{ext}}) \cdot \bm{v}^{\mathrm{ext}}\\ - \frac{1}{2} \sum_{e\in\partial \Omega }\int_e (|\bm{w} \cdot \bm{n}_e| - \bm{w}\cdot \bm
{n}_e)  \bm{\theta} \cdot \bm{v}. 
  \end{multline}
The following holds (see (1.16) and (1.17) in \cite{girault2005splitting} and Lemma 6.1 in \cite{girault2005discontinuous}):
\begin{equation} 
\label{eq:integration_by_parts_c} 
a_{\mathcal{C}}(\bm{w}; \bm{w}, \bm{v},\bm{\theta}) = -\bar{a}_{\mathcal{C}}(\bm{w}; \bm{w}, \bm{\theta} , \bm{v}), \quad \forall \bm{w}, \bm{v}, \bm{\theta} \in \bm{X}.
\end{equation} 
We recall the positivity property satisfied by $a_\mathcal{C}$ (see (1.18) in \cite{girault2005discontinuous}):
\begin{equation}\label{eq:cpositivity}
a_\mathcal{C}(\bm{w};\bm{w};\bm{v},\bm{v}) \geq 0, \quad \forall \bm{w}, \bm{v} \in \bfX.
\end{equation}
In the analysis below, we will use properties of the form $a_\mathcal{C}$ and it helps to define the following forms.
For $\bm{z}, \bm{w}, \bm{v}, \bm{\theta} \in \bm{X}$, we write:
\begin{align}
\mathcal{C}(\bm{w},  \bm{v}, \bm{\theta}) = &\sum_{E \in \mesh_h}   \left( \int_E (\bm{w} \cdot \nabla \bm{v} ) \cdot \bm{\theta} + \frac{1}{2}\int_E (\nabla \cdot \bm{w} ) \bm{v} \cdot \bm{\theta} \right)  \nonumber\\ 
& - \frac{1}{2} \sum_{e \in \Gamma_h \cup \partial \Omega} \int_e [\bm{w}] \cdot \bm{n}_e \{\bm{v} \cdot \bm{\theta} \},
\nonumber\\
    \mathcal{U}(\bm{z}; \bm{w}, \bm{v}, \bm{\theta}) = &\sum_{E \in \mesh_h } \int_{\partial E_{-}^{\bm{z}}} \{ \bm{w}\} \cdot \bm{n}_E(\bm{v}^{\mathrm{int}} - \bm{v}^{\mathrm{ext}})\cdot \bm{\theta}^{\mathrm{int}}. \label{eq:split_ac1}
\end{align}
Therefore, we have
\begin{equation}
a_{\mathcal{C}}(\bm{w}; \bm{w}, \bm{v}, \bm{\theta})  = \mathcal{C}(\bm{w}, \bm{v}, \bm{\theta})-\mathcal{U}(\bm{w}; \bm{w}, \bm{v}, \bm{\theta}). \label{eq:split_ac2}
\end{equation}
With this notation, for any $\bm{u},\bm{w},\bm{v}_1, \bm{v}_2, \bm{\theta} \in \bm{X}$, we have: %(for our reference, see Lemma \ref{lemma:splitting_technique} in appendix)
\begin{align}
a_\mathcal{C}(\bm{u};\bm{u}, \bm{v}_1, \bm{\theta})&- a_\mathcal{C}(\bm{w};\bm{w}, \bm{v}_2, \bm{\theta})=a_\mathcal{C}(\bm{w};\bm{w}, \bm{v}_1-\bm{v}_2, \bm{\theta}) + \mathcal{C}(\bm{u}-\bm{w},\bm{v}_1, \bm{\theta})\nonumber \\ & \quad  - \mathcal{U}(\bm{w}; \bm{u}-\bm{w},\bm{v}_1, \bm{\theta})   -(\mathcal{U}(\bm{u};\bm{u},\bm{v}_1, \bm{\theta}) - \mathcal{U}(\bm{w}; \bm{u},\bm{v}_1,\bm{\theta})). \label{eq:splitting_technique}
\end{align}
We use the symmetric interior penalty dG for the elliptic operator $-\Delta \bm{v}$ \cite{riviere2008discontinuous}. For $\bm{v}, \bm{\theta} \in \bfX$, 
\begin{multline}
a_{\mathcal{D}}(\bm{v}, \bm{\theta}) = \sum_{E \in \mesh_h} \int_E \nabla \bm{v} : \nabla \bm{\theta} 
- \sum_{e\in\Gamma_h \cup \partial \Omega} \int_e \{\nabla_h \bm{v}\}\bm{n}_e \cdot [\bm{\theta}]  \\ 
- \sum_{e \in \Gamma_h \cup \partial \Omega} \int_e \{\nabla_h \bm{\theta} \}\bm{n}_e\cdot [\bm{v}] 
+  \sum_{e\in \Gamma_h \cup \partial \Omega} \frac{\sigma}{h_e} \int_{e} [\bm{v}]\cdot [\bm{\theta}].
\label{eq:eliptic_form_v}
\end{multline} 
In the above form, $h_e =  |e|^{1/(d-1)}$, $\sigma > 0$ is a user specified penalty parameter, and $\nabla_h$ denotes the broken gradient operator . The discretization for the term $ - \nabla p$ is given as follows. For $\bm{\theta} \in \bfX$ and $q \in M$, define  
\begin{align}
b(\bm{\theta}, q) &= \sum_{E \in \mesh_h }\int_E (\nabla \cdot \bm{\theta}) \, q  - \sum_{e \in \Gamma_h \cup \partial \Omega} \int_e \{q\}[\bm{\theta}]\cdot \bm{n}_e  \label{eq:form_b}  \\
& = -\sum_{E \in \mesh_h} \int_E \bm{\theta} \cdot \nabla q  + \sum_{e\in \Gamma_h} \int_e \{ \bm{\theta} \}\cdot \bm{n}_e [q].  \label{eq:equiv_form_b}
\end{align}
To approximate $\bm{u}$ and $p$, we introduce discrete function spaces $\bm{X}_h \subset \bm{X}$ and  $M_{h0}\subset M_h \subset M $. For any integer  $k \geq 1$:
\begin{alignat}{2}
\bfX_h &= \{ \bm{v}_h \in (L^2(\Omega))^d: &&\quad \forall E \in \mesh_h,\quad \bm{v}_h \vert_{E} \in (\mathbbm{P}_{k}(E))^d\}, \\
M_h &= \{ q_h \in L^2(\Omega):&& \quad  \forall E \in \mesh_h, 
\quad q_h \vert_E \in \mathbbm{P}_{k-1}(E)\}, \\
M_{h0}&=\{ q_h \in M_h:  && \quad \int_\Omega q_h = 0\}. 
\end{alignat}
In the above, for $n \in \mathbbm{N}$, $\mathbbm{P}_n(E)$ denotes the space of polynomials of degree at most $n$. 
To discretize the elliptic operator $-\Delta \phi$, we define for $\phi_h, q_h\in M_h$,
\begin{multline}
    \aelip(\phi_h,q_h) = \sum_{E \in \mesh_h} \int_E \nabla \phi_h \cdot \nabla q_h - \sum_{e \in \Gamma_h} \int_e \{\nabla_h \phi_h \}\cdot \bm{n}_e [q_h] \\ 
   -  \sum_{e\in \Gamma_h} \int_e \{ \nabla_h q_h\}\cdot \bm{n}_e [\phi_h] +  \sum_{e\in \Gamma_h} \int_e \frac{\tilde{\sigma}}{h_e}[\phi_h] [q_h]. 
\end{multline}
Here, $\tilde{\sigma}>0$ is a penalty parameter that will be specified later. \\
We now present the fully discrete scheme by
partitioning the time interval $(0,T]$ into $N_T$ subintervals with equal size $\tau$. Throughout the paper, we use the notation $\bm{g}^n = \bm{g}(t^n)$ and $q^n = q(t^n)$ for  given functions $\bm{g}$ and $q$. We start by setting $p_h^0 = \phi_h^0 = 0$ and letting $\bm{u}_h^0$ be the  $L^2$ projection of $\bm{u}^0$ onto $\bm{X}_h$.  
For $n =1, \ldots, N_T$, given $(\bm{u}_h^{n-1}, p_h^{n-1}) \in \bfX_h \times  M_h$, compute $\bm{v}_h^n \in \bfX_h$ such that for all $\bm{\theta}_h \in \bfX_h$, 
\begin{multline}
    (\bm{v}_h^n, \bm{\theta}_h) + \tau a_{\mathcal{C}}(\bm{u}_h^{n-1};\bm{u}_h^{n-1}, \bm{v}_h^n, \bm{\theta}_h) + \tau \mu a_\mathcal{D}(\bm{v}_h^n, \bm{\theta}_h ) = (\bm{u}_h^{n-1}, \bm{\theta}_h) \\ + \tau b(\bm{\theta}_h,p_h^{n-1}) + \tau (\bm{f}^{n}, \bm{\theta}_h). \label{eq:intermidiate_velocity}
\end{multline}
Next, compute $\phi_h^n \in  M_{h0} $ such that for all $q_h \in M_{h0}$,
\begin{equation}
a_{\mathrm{ellip}}(\phi_h^n, q_h) = - \frac{1}{\tau} b(\bm{v}_h^n, q_h). \label{eq:pressure_correction}
\end{equation} 
Finally, compute $p_h^n \in M_h$ and  $\bm{u}_h^n \in \bfX_h$ such that for all $q_h \in M_h$ and for all  $\bm{\theta}_h \in \bfX_h$, 
\begin{align}
    (p_h^n, q_h) & = (p_h^{n-1}, q_h) + (\phi^n_h, q_h) - \delta \mu b(\bm{v}_h^n, q_h), \label{eq:update_pressure_1} \\ 
    (\bm{u}_h^n, \bm{\theta}_h) & = (\bm{v}_h^n, \bm{\theta}_h) + \tau b(\bm{\theta}_h, \phi_h^n).\label{eq:update_velocity_1}
    \end{align}
Here $\delta$ is a positive constant to be determined in the subsequent sections.
 The unique solvability of above algorithm is proved in \cite{inspaper1}.

For $\bm{\theta}\in \bfX$, define the energy norm as follows: 
\begin{equation}
    \| \bm{\theta}\|^2_{\DG} = \sum_{E \in \mesh_h}  \|\nabla \bm{\theta}  \|^2_{L^2(E)} + \sum_{e \in \Gamma_h \cup \partial \Omega } \frac{\sigma}{h_e} \|[\bm{\theta}]\|^2_{L^2(e)}.
  \end{equation}
  For $q \in M$, the energy semi-norm is defined as such: 
  \begin{equation}
    \vert q_h \vert^2_{\DG} = \sum_{E \in \mesh_h}  \|\nabla q \|^2_{L^2(E)} 
+ \sum_{e \in \Gamma_h } \frac{\tilde{\sigma}}{h_e} \|[q]\|^2_{L^2(e)}.
  \end{equation}
 Clearly, $\vert \cdot \vert_{\DG}$ is a norm for the space $M_{h0}$. We recall the following coercivity properties which hold if $\sigma$ and $\tilde{\sigma}$ are large enough \cite{riviere2008discontinuous}. For all $\bm{\theta}_h \in \bm{X}_h$ for all $q_h \in M_h$:
\begin{alignat}{2}
    a_\mathcal{D} (\bm{\theta}_h, \bm{\theta}_h) \geq \frac{1}{2} \|\bm{\theta}_h \|_{\DG}^2, \quad 
    \aelip(q_h, q_h ) \geq \frac12 | q_h |^2_{\DG}.  \label{eq:coercivity_a_ellip} 
\end{alignat}
In addition, we have the following continuity bound: 
\begin{equation}
| a_{\mathcal{D}}(\bm{\theta}_h , \bm{v}_h) | \leq C \|\bm{\theta}_h \|_{\DG} \|\bm{v}_h\|_{\DG}, \quad \forall \bm{\theta}_h, \bm{v}_h \in \bm{X}_h. \label{eq:continuity_a}
\end{equation}
% We recall that \eqref{eq:coercivity_a_epsilon} holds with $\kappa = 1$ if $\epsilon = 1$ and with $\kappa= 1/2$ if $\epsilon \in \{ -1, 0 \}$ and $\sigma$ is large enough. Property \eqref{eq:coercivity_a_ellip} holds for $\tilde{\sigma}$ large enough. Further, we recall that $a_{\epsilon}$ and $\aelip$ are continuous. There exists a constant $C$ independent of $h$, $\bm{\theta}_h$, $\bm{z}_h$, $q_h$, and $\xi_h$ such that   
% \begin{alignat}{2}
% a_\epsilon(\bm{\theta}_h, \bm{z}_h) &\leq C \| \bm{\theta}_h \|_{\DG} \|\bm{z}_h\|_{\DG}, && \quad \forall (\bm{\theta}_h , \bm{z}_h) \in \bm{X}_h \times \bm{X}_h, \label{eq:continuity_a_epsilon}\\ 
% \aelip(q_h, \xi_h) & \leq C | q_h |_{\DG} |\xi_h |_{\DG},  && \quad \forall (q_h,\xi_h) \in M_h \times M_h \label{eq:continuity_aellip}. 
% \end{alignat}
% We now recall the equivalent expression for the form $b$.   
% \begin{lemma}
% We have the following equivalent form for $b(\bm{\theta}, q)$, 
% \begin{equation}
%     b(\bm{\theta}, q ) = -\sum_{E \in \mesh_h} \int_E \bm{\theta} \cdot \nabla q  + \sum_{e\in \Gamma_h} \int_e \{ \bm{\theta} \}\cdot \bm{n}_e [q], \quad \forall (\bm{\theta},q) \in \bm{X} \times M.  \label{eq:equiv_form_b}
% \end{equation} 
% \end{lemma}
We also recall the following lift operators \cite{inspaper1},
$R_h: \bfX_h\rightarrow M_h$ and $\bm{G}_h: M_h\rightarrow \bfX_h$:
%. Given $e \in \Gamma_h \cup \partial \Omega$, define the lift operator $r_e: (L^2(e))^d \rightarrow M_h$ as follows. 
%\begin{equation}
%\int_{\Omega} r_e(\bm{\zeta}) q_h  = \int_e \{q_h\} \bm{\zeta} \cdot \bm{n}_e, \quad \forall q_h \in M_h.   
%\end{equation}
%Given an interior face $e\in \Gamma_h$, define the lift operator $\bm{g}_e: L^2(e) \rightarrow \bfX_h$,  
%\begin{equation}
%\int_{\Omega} \bm{g}_e(\zeta) \cdot \bm{\theta}_h = \int_e \{ \bm{\theta}_h \} \cdot \bm{n}_e \zeta, \quad \forall \bm{\theta}_h \in \bfX_h. \label{eq:def_ge}
%\end{equation}
%With these definitions, we  construct two operators, $R_h: \bfX_h\rightarrow M_h$ and $\bm{G}_h: M_h\rightarrow \bfX_h$ 
%\begin{alignat}{2}
    %R_h([\bm{\theta}_h]) &= \sum_{e\in \Gamma_h \cup \partial \Omega} r_e([\bm{\theta}_h]),&& \quad \bm{\theta}_h \in \bfX_h,   \label{eq:lift_op_1}\\ 
    %\bm{G}_h([\beta_h]) &= \sum_{e\in \Gamma_h} \bm{g}_e([\beta_h]),&& \quad \beta_h \in M_h. \label{eq:lift_op_2} 
%\end{alignat}
%
\begin{alignat}{2}
    (R_h([\bm{\theta}_h]),q_h) &= \sum_{e\in \Gamma_h \cup \partial \Omega} \int_e \{ q_h\} [\bm{\theta}_h]\cdot\bm{n}_e,&& \quad \bm{\theta}_h \in \bfX_h, \quad q_h\in M_h,  \label{eq:lift_op_1}\\ 
    (\bm{G}_h([q_h]),\bm{\theta}_h) &= \sum_{e\in \Gamma_h} \int_e \{\bm{\theta}_h\} \cdot\bm{n}_e [q_h],&& \quad \bm{\theta}_h \in \bfX_h, \quad q_h \in M_h. \label{eq:lift_op_2} 
\end{alignat}
%
% It is easy to check that $R_h$ and $\bm{G}_h$ are linear operators.
%\begin{lemma}\label{lemma:properties_lift_operators}
There exist constants $M_{k-1}, \tilde{M}_{k} > 0$ independent of $h$ but depending on the polynomial degree $k$, such that 
   for $q_h \in M_h$ and $\bm{\theta}_h \in \bfX_h$, the following bounds hold. 
       \begin{align}
       \|R_h([\bm{\theta}_h])\|^2 &\leq M_{k-1}^2  \sum_{e \in \Gamma_h \cup \partial \Omega} h_e^{-1}\|[\bm{\theta}_h]\|_{L^2(e)}^2, \label{eq:lift_prop_r}\\  
       \|\bm{G}_h([q_h])\|^2 &\leq \tilde{M}_{k}^2 \sum_{e \in \Gamma_h} h_e^{-1}\|[q_h]\|_{L^2(e)}^2. \label{eq:lift_prop_g}
       \end{align}
%       \begin{align}
%       \|R_h([\bm{\theta}_h])\|&\leq M_{k-1} \left(\sum_{e \in \Gamma_h \cup \partial \Omega} h_e^{-1}\|[\bm{\theta}_h]\|_{L^2(e)}^2\right)^{1/2}\label{eq:lift_prop_r}, \\ 
%       \|\bm{G}_h([q_h])\| &\leq \tilde{M}_{k} \left(\sum_{e \in \Gamma_h} h_e^{-1}\|[q_h]\|_{L^2(e)}^2\right)^{1/2}. \label{eq:lift_prop_g}
%       \end{align}
   %\end{lemma}
Considering the definitions of the lift operators \eqref{eq:lift_op_1}-\eqref{eq:lift_op_2} and \eqref{eq:form_b}-\eqref{eq:equiv_form_b}, we have
 % the following equivalent forms to \eqref{eq:form_b} and \eqref{eq:equiv_form_b} respectively.
\begin{align}
    b(\bm{\theta}_h, q_h) & = ( \nabla_h   \cdot \bm{\theta}_h, q_h ) - (R_h([\bm{\theta}_h]), q_h), 
\quad \forall \bm{\theta}_h\in\bfX_h, \, \forall q_h \in M_h, \label{eq:def_b_lift} \\ 
    b(\bm{\theta}_h, q_h) & = -(\nabla_h q_h, \bm{\theta}_h) + (\bm{G}_h([q_h]), \bm{\theta}_h),
\quad \forall \bm{\theta}_h\in\bfX_h, \, \forall q_h \in M_h. \label{eq:def_b_lift_2}
 \end{align}
Consequently, using Cauchy-Schwarz's inequality and \eqref{eq:lift_prop_g}, it is easy to show that
\begin{equation}\label{eq:boundformb}
b(\bm{\theta}_h, q_h) \leq C \Vert \bm{\theta}_h \Vert \, \vert q_h \vert_{\mathrm{DG}}, \quad
 \forall \bm{\theta}_h\in\bfX_h, \, \forall q_h \in M_h. 
\end{equation}
%\subsection{Fully discrete scheme}
% \begin{align}
% (p_h^n, q_h) & = (p_h^{n-1}, q_h) + (\phi^n_h, q_h) - \delta \mu (\nabla_h \cdot \bm{v}_h^n - R_h([\bm{v}_h^n]), q_h), \label{eq:update_pressure} \\ 
% (\bm{u}_h^n, \bm{\theta}_h) & = (\bm{v}_h^n, \bm{\theta}_h) - \tau(\nabla_h \phi_h^n- \bm{G}_h([\phi^n_h]), \bm{\theta}_h).\label{eq:update_velocity}
% \end{align}
 
    \section{Improved error estimate for the velocity}\label{sec:improved_estimate}
    %We assume that $\bm{u}^0
%\in H^1_0(\Omega)^d$ and $\nabla \cdot \bm{u}^0 = 0$.  
Throughout the paper, we denote by  $C$ a generic constant independent of $\mu, h ,$ and $\tau$ and by $C_\mu$ a generic constant independent of $h$ and $\tau$ but depending on $e^{1/\mu}$. These constants may take different values when used in different places.
% and may depend on certain norms of the true solution $(\bm{u}, p)$. %These norms are associated with the spaces specified by the regularity assumptions of each Theorem. 
For a function $\bm{v} \in (W^{1,3}(\Omega) \cap L^{\infty}(\Omega))^d$, define the norm 
    $\vertiii{\bm{v}} = \|\bm{v}\|_{L^{\infty}(\Omega)} + \vert \bm{v}\vert_{W^{1,3}(\Omega)}$.
We begin by recalling the interpolation operators, error equations, and the error and stability estimates derived in \cite{inspaper1}. 
We will make use of the operator $\Pi_h: H^1_0(\Omega)^d \rightarrow \bfX_h$ that satisfies the following \cite{riviere2008discontinuous,chaabane2017convergence}.
\begin{equation}
b(\Pi_h \bm{u} (t), q_h) = b(\bm{u} (t) ,q_h)=0 , \quad \forall q_h \in M_h,~\forall 0\leq t \leq T. \label{eq:def_pi_h}
\end{equation}
This operator satisfies the following stability and approximation properties \cite{chaabane2017convergence}. For $ 0\leq t\leq T$, if $\bm{u}(t) \in (W^{s,r}(\Omega) \cap H_0^1(\Omega) )^d$, we have the global estimates: 
\begin{alignat}{2}
    \|\Pi_h \bm{u}(t)\|_{L^{\infty}(\Omega)} & \leq C \vertiii{\bm{u}(t)}, && \quad r=3, s = 1,  \,\,   \bm{u}(t) \in L^{\infty}(\Omega)^d , \label{eq:stability_linf_pih} \\ 
    \| \Pi_h \bm{u}(t) - \bm{u}(t) \|_{L^r(\Omega)} &\leq Ch^s \vert \bm{u}(t) \vert_{W^{s,r}(\Omega)}, && \quad 1\leq r \leq \infty, \,\,  1 \leq s \leq k+1, \label{eq:approximation_prop_1} \\ 
     \| \Pi_h \bm{u}(t) - \bm{u}(t) \|_{\mathrm{DG}}& \leq C h^{s-1}  \vert \bm{u}(t) \vert_{H^{s}(\Omega)},  && \quad  1\leq r \leq \infty, \,\,  1 \leq s \leq k+1 \label{eq:approximation_prop_2}.
    \end{alignat}
For $E\in\mathcal{E}_h$ and $\bm{u}(t)\in \bfX\cap (W^{s,r}(E)\cap H_0^1(\Omega))^d$
\begin{equation}
\Vert \nabla (\Pi_h \bm{u}(t)-\bm{u}(t))\Vert_{L^r(E)} \leq C h_E^{s-1} \vert \bm{u}(t)\vert_{W^{s,r}(\Delta_E)}
\label{eq:stabgrad}
\end{equation}
where $\Delta_E$ is a macro element that contains $E$. As a corollary, we have
%\Rd Do we also use the following (and is it a consequence of the above?) \blue It is a consequence of the above. I am not sure if it is used specifically. \Bk
\begin{equation}
\vertiii{\bm{u}-\Pi_h\bm{u}}\leq C \vertiii{\bm{u}}. \label{eq:vertiii}
\end{equation}
\Bk
Let $\pi_h: L^2(\Omega) \rightarrow M_h$ denote the $L^2$ projection. %For $ 0\leq t\leq T$ and $p(t) \in L^2(\Omega)$: 
%\begin{equation}
 %   \int_{E} (\pi_h p(t) - p(t))q_h = 0, \quad  \forall q_h \in \mathbbm{P}_{k-1}(E), \quad \forall E \in \mesh_h.
%\end{equation}
For $0\leq t \leq T$ and $p(t) \in H^s(\Omega)$, the following estimate holds. 
\begin{align} \label{eq:l2_proj_approximation}
    \|\pi_h p(t) - p(t) \|_{L^2(E)} + h_E\|\nabla_h(\pi_h p(t) - p(t) )\|_{L^2(E)} \leq Ch_E^{\min(k,s)}\vert p(t) \vert_{H^s(E)}. 
\end{align}
We also recall that the $L^2$ projection is stable in the dG norm \cite{GiraultLiRiviere2016}. For $p(t) \in H^1(\Omega)$, 
\begin{equation}
    | \pi_h p(t)|_{\DG} \leq C | p(t) |_{H^1(\Omega)}.\label{eq:bd_phN}
\end{equation} 
Further, we recall the following Poincare inequality \cite{girault2005discontinuous,lasis2003poincare}. %There exists a constant $C_P$ independent of $h$ and $\tau$ but depending on $q$ such that 
\begin{equation}
    \| \bm{\theta} \|_{L^{q}(\Omega)} \leq C_P \| \bm{\theta} \|_{\DG}, \quad \forall \bm{\theta} \in \bm{X}, \label{eq:discreter_poincare}
\end{equation}
where $2\leq q<\infty$ in 2D $(d=2)$ and $2\leq q\leq 6$ in 3D $(d=3)$.  
We proceed by stating the error equations. To this end, define the following discretization errors, $\bm{\tilde{e}}_h^{n}, \, \bm{e}_h^{n}\in\bm{X}_h$:
\begin{align}
\bm{\tilde{e}}_h^{n}= \bm{v}_h^n - \Pi_h \bm{u}^n, \quad \bm{e}_h^{n} = \bm{u}_h^n - \Pi_h \bm{u}^{n},  \quad \forall n\geq 0.  \label{eq:def_error_functions}
\end{align}
We define $\bm{v}_h^0 = \Pi_h \bm{u}^0$, therefore $\tilde{\bm{e}}_h^0 = \bm{0}$. 
To simplify the writeup, we also define  
\begin{align}
R_{\mathcal{C}} (\bm{\theta}_h) & = a_\mathcal{C}(\bm{u}_h^{n-1}; \bm{u}_h^{n-1}, \Pi_h \bm{u}^n, \bm{\theta}_h) -  a_\mathcal{C}(\bm{u}^{n}; \bm{u}^{n}, \bm{u}^n, \bm{\theta}_h), \\ 
R_{t}(\bm{\theta}_h) & = \tau((\partial_t \bm{u})^n, \bm{\theta}_h) + (\Pi_h \bm{u}^{n-1} - \Pi_h \bm{u}^{n}, \bm{\theta}_h).
\label{eq:defRt}
\end{align}
The error equations are, for $n\geq 1$ \cite{inspaper1}.
\begin{align}
    &(\err, \bm{\theta}_h) + \tau a_\mathcal{C}(\bm{u}_h^{n-1};\bm{u}_h^{n-1}, \err, \bm{\theta}_h) +  \tau R_{\mathcal{C}} (\bm{\theta}_h) + \tau \mu a_{\mathcal{D}}(\err, \bm{\theta}_h) = (\bm{e}_h^{n-1}, \bm{\theta}_h) \nonumber \\ & \quad - \tau \mu a_{\mathcal{D}}(\Pi_h \bm{u}^n - \bm{u}^n, \bm{\theta}_h )
 + \tau  b(\bm{\theta}_h, p_h^{n-1} - p^n) + R_t(\bm{\theta}_h),\, \forall \bm{\theta}_h \in \bm{X}_h, \label{eq:first_err_eq} \\ 
 &(\bm{e}_h^n, \bm{\theta}_h)  = (\bm{\tilde{e}}_h^n, \bm{\theta}_h) +  \tau b(\bm{\theta}_h, \phi_h^n ), \, \forall \bm{\theta}_h \in \bm{X}_h, \label{eq:sec_err_eq}\\
%\end{align} 
%For all $q_h \in M_h$ and $n\geq 1$: 
%\begin{align}
 &   \aelip(\phi_h^n, q_h)  = -\frac{1}{\tau} b(\bm{\tilde{e}}_h^n, q_h),  \forall q_h \in M_{h0}, \label{eq:error_pressure_correction}
    \\
 &   (p_h^n,q_h) = (p_h^{n-1},q_h) + (\phi_h^n,q_h) -\delta \mu (\nabla_h \cdot \err - R_h ([\err]), q_h ), \quad
\forall q_h\in M_h.\label{eq:error_update_pressure}
\end{align}
% For all $q_h \in M_h$ and $n \geq 1$,
% \begin{equation}
% (p_h^n,q_h) = (p_h^{n-1},q_h) + (\phi_h^n,q_h) -\delta \mu (\nabla_h \cdot \err - R_h ([\err]), q_h ), \quad \forall q_h \in M_h. \label{eq:error_update_pressure}
% \end{equation}
% For all $\bm{\theta}_h$ and $n \geq 1$, 
% \begin{equation}
%     (\bm{e}_h^n, \bm{\theta}_h) = (\bm{\tilde{e}}_h^n, \bm{\theta}_h) +  \tau b(\bm{\theta}_h, \phi_h^n ) \quad \forall  \bm{\theta}_h \in \bm{X}_h. \label{eq:sec_err_eq}
% \end{equation}
We now recall and establish useful properties for the error functions and for the forms $a_\mathcal{C}$ and $b$. The error function $\bm{e}_h^n$ satisfies the following important property \cite{inspaper1}. 
\begin{lemma}\label{prep:weak_div_error}  
    For all $q_h \in M_h$ and $n\geq 1$,  the following holds. 
    \begin{align}
        b(\bm{e}_h^n, q_h)& = b(\err, q_h) + \tau\aelip(\phi_h^n,q_h) -\tau \sum_{e \in \Gamma_h} \frac{\tilde{\sigma}}{h_e} \int_e [\phi_h^n][q_h] \nonumber \\ & \quad \quad  + \tau(\bm{G}_h([\phi_h^n]), \bm{G}_h([q_h])), \label{eq:div_error} \\ 
        b(\bm{e}_h^n, q_h) & = -\tau \sum_{e \in \Gamma_h} \frac{\tilde{\sigma}}{h_e} \int_e [\phi_h^n][q_h] + \tau(\bm{G}_h([\phi_h^n]), \bm{G}_h([q_h])).  \label{eq:div_error_2}
    \end{align} 
\end{lemma}
The forms $a_\mathcal{C}$ and $b$ admit the following bounds.
% \begin{lemma}\label{prep:bounds_nonlinear_terms}
%  Fix $\phi_h\in M_h$.  Assume that there is $\bm{w}_h \in \bm{X}_h$ such that
% \begin{equation}
%     b(\bm{w}_h, q_h ) =  -\tau \sum_{e \in \Gamma_h} \frac{\tilde{\sigma}}{h_e} \int_e [\phi_h ][q_h] + \tau(\bm{G}_h([\phi_h ]), \bm{G}_h([q_h])), \quad \forall q_h \in M_h.  \label{eq:assum_on_w}
% \end{equation}
% Then, there  exists a constant $C$, independent of $h$ and $\tau$ such that 
% \begin{enumerate}[(i)]
%     \item 
%     If $\bm{v} \in (W^{1,3}(\Omega)\cap H_0^1(\Omega) \cap L^{\infty}(\Omega))^d$, then for any $\bm{z} \in \bm{X}$ and $\bm{\theta}_h \in \bm{X}_h$:
%         \begin{align}
%      |a_\mathcal{C}(\bm{z}; \bm{w}_h, \bm{v}, &\bm{\theta}_h)| + |b(\bm{w}_h, \bm{v} \cdot \bm{\theta}_h)| \nonumber \\ & \leq 
%     C \left(\|\bm{w}_h \| + \tau \left( \sum_{e \in \Gamma_h} \frac{\tilde{\sigma}}{h_e} \| [ \phi_h ] \|_{L^2(e)}^2 \right)^{1/2}\right) 
%     \vertiii{\bm{v}} \| \bm{\theta}_h\|_{\DG}, \label{eq:first_estimate_nonlinear_form}
%     \end{align}
%     \item If $\bm{v} \in (H^{k+1}(\Omega) \cap  H^1_0(\Omega))^d$, then for any $\bm{z} \in \bm{X}$ and $\bm{\theta}_h \in \bm{X}_h$:
%     \begin{alignat}{2}
%     |\mathcal{C}(\bm{w}_h, \bm{v} - \Pi_h \bm{v},& \bm{\theta}_h)|&& \leq C \left(\|\bm{w}_h\| + \tau \left( \sum_{e \in \Gamma_h} \frac{\tilde{\sigma}}{h_e} \| [\phi_h  ] \|_{L^2(e)}^2 \right)^{1/2} \right)\vertiii{\bm{v}}\| \bm{\theta}_h\|_{\DG},\label{eq:second_estimate_nonlinear_term_C1}
%     \end{alignat}
%     \end{enumerate}
% \end{lemma}
\begin{lemma} \label{prep:bounds_nonlinear_terms}
    There exists a constant $C$, independent of $h$, $\tau$, $\bm{w}_h, \bm{z}, \bm{v}$ and $\bm{\theta}_h$, such that the following estimates hold. 
    \begin{enumerate}[(i)]
    \item If $\bm{v} \in (W^{1,3}(\Omega)\cap H_0^1(\Omega) \cap L^{\infty}(\Omega))^d$, then for any $\bm{z} \in \bm{X}$ and $\bm{\theta}_h, \bm{w}_h \in \bm{X}_h$:
        \begin{equation}
        |a_\mathcal{C}(\bm{z}; \bm{w}_h, \bm{v}, \bm{\theta}_h)| + |b(\bm{w}_h, \bm{v} \cdot \bm{\theta}_h))| \leq 
    C \|\bm{w}_h \| \vertiii{\bm{v}} \| \bm{\theta}_h\|_{\DG}.
    \label{eq:first_estimate_nonlinear_form}
    \end{equation}
    \item If $\bm{v} \in (H^{k+1}(\Omega) \cap  H^1_0(\Omega))^d$, then for any $\bm{z} \in \bm{X}$ and $\bm{\theta}_h , \bm{w}_h \in \bm{X}_h$:
    \begin{multline}
    |\mathcal{C}(\bm{w}_h, \bm{v} - \Pi_h \bm{v}, \bm{\theta}_h)| + |b(\bm{w}_h, (\bm{v} - \Pi_h \bm{v}) \cdot \bm{\theta}_h)| \\  \leq C \|\bm{w}_h\|  |\bm{v}|_{H^{k+1}(\Omega)}\| \bm{\theta}_h\|_{\DG}.\label{eq:second_estimate_nonlinear_term_C1}
    \end{multline}
    \end{enumerate}
    \end{lemma}
    \begin{proof}
    (i) Since $\bm{v}$ belongs to $H^1(\Omega)^d$ and vanishes on the boundary, we have: 
    \begin{equation}
     a_\mathcal{C}(\bm{z}; \bm{w}_h, \bm{v}, \bm{\theta}_h) =  \sum_{E \in \mesh_h} \int_{E} (\bm{w}_h \cdot \nabla \bm{v}) \cdot \bm{\theta}_h  + \frac{1}{2} b(\bm{w}_h, \bm{v} \cdot \bm{\theta}_h). \label{eq:ident_cont_u}
    \end{equation}
    With H\"{o}lder's inequality and \eqref{eq:discreter_poincare}, 
    we have 
    \begin{equation} \label{eq:first_nonlinear_estimate_0}
       \left|  \sum_{E \in \mesh_h} \int_{E} (\bm{w}_h \cdot \nabla \bm{v} ) \cdot \bm{\theta}_h \right| \leq C_P\| \bm{w}_h \| \vert \bm{v} \vert_{W^{1,3}(\Omega)} \|\bm{\theta}_h\|_{\DG}. 
    \end{equation}
    To bound the second term in \eqref{eq:ident_cont_u}, we use \eqref{eq:equiv_form_b}. We have 
    \begin{align*}
    \frac{1}{2} b(\bm{w}_h, \bm{v} \cdot \bm{\theta}_h) &= -\frac12 \sum_{E \in \mesh_h} \int_{E} \bm{w}_h \cdot \nabla(\bm{v} \cdot \bm{\theta}_h) +\frac12 \sum_{e\in \Gamma_h} \int_e \{\bm{w}_h\} \cdot \bm{n}_e (\bm{v} \cdot [\bm{\theta}_h]) = A_1 + A_2. 
    \end{align*}
    The term $A_1$ is bounded as follows. With H\"{o}lder's inequality and \eqref{eq:discreter_poincare},
    \begin{align*}
    |A_1| \leq  \|\bm{w}_h\| \left(\|\nabla \bm{v}\|_{L^3(\Omega)}\|\bm{\theta}_h\|_{L^6(\Omega)} + \|\bm{v}\|_{L^{\infty}(\Omega)} \|\nabla_h \bm{\theta}_h\|_{L^2(\Omega)}\right) \leq C \|\bm{w}_h\| \vertiii{\bm{v}} \|\bm{\theta}_h\|_{\DG}. 
    \end{align*}
    For $A_2$, consider a face $e \in \Gamma_h$ and let $E_e^1$ and $E_e^2$ denote the elements sharing $e$. 
    \begin{align*}
    |\int_e \bm{w}_h|_{E_e^1} \cdot \bm{n}_e (\bm{v} \cdot [\bm{\theta}_h]) |\leq C \|\bm{v}\|_{L^{\infty}(\Omega)} |e|^{1/2}|E_e^1|^{-1/2}\|\bm{w}_h|_{E_e^1}\|_{L^2(E_e^1)}\|[\bm{\theta}_h]\|_{L^2(e)}
    \end{align*}
    Since $|e|^{1/2}|E_e^1|^{-1/2}h_e^{1/2} \leq C$, we apply H\"{o}lder's inequality and we obtain 
    \begin{align*}
    \sum_{e \in \Gamma_h} |\int_e |\bm{w}_h|_{E_e^1} \cdot \bm{n}_e (\bm{v} \cdot [\bm{\theta}_h]) | \leq C \|\bm{v}\|_{L^{\infty}(\Omega)} \|\bm{w}_h\|_{L^2(\Omega)}\|\bm{\theta}_h\|_{\DG}.
    \end{align*}
    The remaining term in $A_2$ is handled similarly. We conclude that \eqref{eq:first_estimate_nonlinear_form} holds. 

    (ii) We remark that by a Sobolev embedding, $\bm{v} \in (W^{1,3}(\Omega) \cap L^{\infty}(\Omega))^d$ since $k\geq 1$. We have  
    \begin{multline}
        \mathcal{C}(\bm{w}_h, \bm{v}- \Pi_h \bm{v}, \bm{\theta}_h) =  \sum_{E \in \mesh_h} \int_{E} (\bm{w}_h \cdot \nabla (\bm{v} - \Pi_h\bm{v} ) )\cdot \bm{\theta}_h    + \frac{1}{2} b(\bm{w}_h, (\bm{v} -\Pi_h \bm{v}) \cdot \bm{\theta}_h). \nonumber % \label{eq:expanding_C} 
       \end{multline}
       The first term is bounded exactly like in the proof of (i) (see the derivation of bound \eqref{eq:first_nonlinear_estimate_0}) with $\bm{v}$ replaced by $\bm{v}- \Pi_h\bm{v}$. By the stability of the interpolant \eqref{eq:stabgrad}, we have
       \begin{equation}
            \sum_{E \in \mesh_h} \left| \int_{E} (\bm{w}_h \cdot \nabla (\bm{v}- \Pi_h\bm{v} ) )\cdot \bm{\theta}_h \right|  
        \leq C \|\bm{w}_h\| |\bm{v}|_{W^{1,3}(\Omega)} \| \bm{\theta}_h \|_{\DG}. \label{eq:third_nonlinear_term_0}
       \end{equation}
       For the second term, we use \eqref{eq:equiv_form_b} and the following identity.
       $$[(\bm{v}-\Pi_h \bm{v})\cdot \bm{\theta}_h]=\{\bm{v}-\Pi_h \bm{v}\}\cdot [\bm{\theta}_h] + [\bm{v}-\Pi_h \bm{v}]\cdot\{\bm{\theta}_h\}, \quad \forall e \in \Gamma_h.$$
       We obtain: 
       \begin{align*}
    & b(\bm{w}_h, (\bm{v}-\Pi_h \bm{v}) \cdot \bm{\theta}_h) = - \sum_{E \in \mesh_h} \int_{E} \bm{w}_h \cdot \nabla((\bm{v} - \Pi_h \bm{v}) \cdot \bm{\theta}_h)\\ & + \sum_{e\in \Gamma_h} \int_e \{\bm{w}_h\} \cdot \bm{n}_e (\{\bm{v} -  \Pi_h \bm{v}\} \cdot [\bm{\theta}_h]) + \sum_{e\in \Gamma_h} \int_e \{\bm{w}_h\} \cdot \bm{n}_e ([\bm{v} -  \Pi_h \bm{v}] \cdot \{\bm{\theta}_h\}) = \sum_{i=1}^3 \tilde{A}_i. % + \tilde{A}_2 + \tilde{A}_3. 
        \end{align*}
    The terms $\tilde{A}_1$ and $\tilde{A}_2$ are handled in a similar way to $A_1$ and $A_2$ with $\bm{v}$ being replaced with $(\bm{v} - \Pi_h \bm{v})$. With the stability property \eqref{eq:stability_linf_pih}  and \eqref{eq:vertiii}, we have
    \begin{align*}
    |\tilde{A}_1| +|\tilde{A}_2| &\leq C \|\bm{w}_h\|\vertiii{\bm{v} - \Pi_h \bm{v}}\|\bm{\theta}_h\|_{\DG}\leq C\|\bm{w}_h\|\vertiii{\bm{v}}\|\bm{\theta}_h\|_{\DG}. 
    %|\tilde{A}_2| &\leq C \|\bm{w}_h\|\|\bm{v} - \Pi_h \bm{v}\|_{L^{\infty}(\Omega)}\|\bm{\theta}_h\|_{\DG} \leq C\|\bm{w}_h\|\vertiii{\bm{v}}\|\bm{\theta}_h\|_{\DG} .
    \end{align*}
        To handle $\tilde{A}_3$, we refer to the proof of bound (6.60) in Lemma 6.5 in \cite{inspaper1}. We have
     \begin{equation*}
     |\tilde{A}_3| \leq C \|\bm{w}_h\| |\bm{v}|_{H^{k+1}(\Omega)}\|\bm{\theta}_h\|_{\DG}. 
     \end{equation*} 
     We conclude that \eqref{eq:second_estimate_nonlinear_term_C1} holds.
     \end{proof}
In addition, we have the following bounds on $a_\mathcal{C}$, $b$, $\mathcal{U}$, and  $\mathcal{C}$ \cite{inspaper1}.
\begin{lemma}\label{lemma:bounds_nonlinear_terms_2}
  There exists a constant $C$, independent of $h, \tau, \bm{w}, \bm{z}, \bm{v},$ and $\bm{\theta}_h$ such that the following bounds hold. 
 \begin{enumerate}[(i)]
      \item If $\bm{w} \in (H^{k+1}(\Omega)\cap H^1_0(\Omega))^d$ satisfies $b(\bm{w}, q_h) = 0,\,\, \forall q_h \in M_h$ and if $\bm{v}$ belongs to $(W^{1,3}(\Omega)\cap H^1_0(\Omega)\cap L^{\infty}(\Omega))^d$, then for any $\bm{z} \in \bm{X}$ and $ \bm{\theta}_h \in \bm{X}_h$:
     \begin{equation}
      |a_\mathcal{
          C}(\bm{z}; \Pi_h\bm{w}- \bm{w}, \bm{v}, \bm{\theta}_h )| + |b(\Pi_h\bm{w}- \bm{w}, \bm{v}\cdot\bm{\theta}_h )|  \leq C  h^{k+1} \vert \bm{w} \vert_{H^{k+1}(\Omega)} \vertiii{\bm{v}} \|\bm{\theta}_h\|_{\DG}. \label{eq:third_estimate_nonlinear_term}
     \end{equation}
     \item If $\bm{w} \in (H^{k+1}(\Omega) \cap  H^1_0(\Omega))^d$ and $\bm{v} \in (W^{1,3}(\Omega) \cap H^1_0(\Omega) \cap L^{\infty}(\Omega))^d$, then for any $\bm{z} \in \bm{X}$ and $\bm{\theta}_h \in \bm{X}_h$:
     \begin{equation}
     \hspace*{-0.5cm}|\mathcal{C}(\Pi_h \bm{v},\Pi_h \bm{w} - \bm{w},  \bm{\theta}_h) | + |\mathcal{U}(\bm{z};\Pi_h \bm{v},\Pi_h \bm{w} - \bm{w}
     ,  \bm{\theta}_h) |  \leq C h^{k}|\bm{w}|_{H^{k+1}(\Omega)}\vertiii{\bm{v}}\| \bm{\theta}_h\|_{\DG}. \label{eq:second_estimate_nonlinear_term_U2}
     \end{equation} 
     \item If $\bm{v} \in (H^{k+1}(\Omega) \cap  H^1_0(\Omega))^d$, then for any $\bm{z} \in \bm{X}$ and $\bm{\theta}_h \in \bm{X}_h$:\begin{equation}
        |\mathcal{U}(\bm{z}; \bm{w}_h, \bm{v}- \Pi_h \bm{v},  \bm{\theta}_h)|  \leq C h^{k-d/4}\| \bm{w}_h \||\bm{v}|_{H^{k+1}(\Omega)} \|\bm{\theta}_h \|_{\DG}. \label{eq:second_estimate_nonlinear_term}
     \end{equation}
 \end{enumerate}
\end{lemma}
The following error and stability estimates were shown in \cite{inspaper1}.
\begin{lemma}\label{lemma:first_err_estimate}
    Assume that $\sigma \geq M_{k-1}^2/d$,  $\tilde{\sigma} \geq 4 \tilde{M}_k^2$, and $\delta \leq 1/(4d)$. We have
    \begin{equation}
        \| \bm{u}_h^m\|^2 +  \frac{\mu }{4} \tau \sum_{n=1}^{m} \| \bm{v}_h^n \|_{\DG}^2  
     \leq \|\bm{u}^0 \|^2 +  \frac{4C_P^2}{\mu} \tau \sum_{n=1}^{m} \|\bm{f}^{n}\|^2. 
    \end{equation}  
    Assume that $\bm{u} \in L^{\infty}(0,T; H^{k+1}(\Omega)^d)$, $\partial_t \bm{u} \in L^2(0,T ; H^{k}(\Omega)^d)$, $\partial_{tt} \bm{u} \in L^2((0,T)\times \Omega)$, and $p \in L^{\infty}(0,T; H^k(\Omega))$. Then, there exists a constant $\gamma$ independent of $\tau, h$ and $\mu$ such that if $\tau \leq \gamma \mu$, 
    \begin{multline} 
      \|\bm{e}_h^m\|^2 +  \sum_{n=1}^m \| \bm{e}_h^{n} - \bm{e}_h^{n-1} \|^2 + \frac{1}{16}\tau^2  \sum_{n=1}^m \vert \phi_h^n \vert_{\DG}^2 + \frac{\mu}{8} \tau \sum_{n=1}^m  \|\err \|_{\DG}^2 \\ \leq C_\mu \left(1 + \mu+\frac{1}{\mu}\right)( \tau + h^{2k}).  \label{eq:first_error_estimate}
    \end{multline}
\end{lemma} 
 We will make use of the following inverse estimates, see Theorem 4.5.11 in \cite{brenner2007mathematical} for \eqref{eq:inverse_estimate}-\eqref{eq:inverse_estimate2} and see Lemma 3.8 in \cite{riviere2008discontinuous} for \eqref{eq:inverse_estimate_dG}. 
\begin{alignat}{2}\label{eq:inverse_estimate}
    \|\bm{\theta}_h \|_{L^3(\Omega)} &\leq C h^{-d/6} \| \bm{\theta}_h \|, \quad \forall \bm{\theta}_h \in \bm{X}_h , \\ 
    \|\bm{\theta}_h \|_{L^{\infty}(\Omega)} & \leq C h^{-d/2} \| \bm{\theta}_h \|, \quad \forall \bm{\theta}_h \in \bm{X}_h ,\label{eq:inverse_estimate2}\\ 
    \|\bm{\theta}_h \|_{\DG} &\leq Ch^{-1}\|\bm{\theta}_h\|, \forall \bm{\theta}_h \in \bm{X}_h, \,\,\,\,   \mathrm{ and } \,\,\,\,  |q_h|_{\DG}  \leq Ch^{-1}\|q_h\|, \forall  q_h \in M_h. \label{eq:inverse_estimate_dG}
\end{alignat}
We will also use the following trace inequalities, see Section 2.1.3 in \cite{riviere2008discontinuous} and Lemma 1.5.2 in \cite{di2011mathematical}. Let $E \in \mesh_h$. For all $e \in \partial E$,  
\begin{alignat}{2} 
\|\bm{v}\|_{L^2(e)} &\leq C h^{-1/2} (\|\bm{v}\|_{L^2(E)} + h\|\nabla \bm{v}\|_{L^2(E)}), && \quad \forall \bm{v}\in H^1(E)^d \label{eq:trace_ineq_continuous}\\ 
\|\bm{\theta}_h\|_{L^r(e)} &\leq Ch^{-1/r}\|\bm{\theta}_h\|_{L^r(E)},&&  \quad \forall \bm{\theta}_h \in \bm{X}_h, \,\,  r \geq 1. \label{eq:trace_ineq_discrete}
\end{alignat}
One important tool in deriving improved error estimates is to carefully construct a dual problem and its discretization. To this end, define the error functions
\[
\bm{\chi}(t) = \bm{u}_h^n - \bm{u}(t), \quad \forall t^{n-1} < t\leq t^n, \quad \forall n\geq 1, \quad
\bm{\chi}(0) = \bm{u}_h^0 - \bm{u}^0. 
\] 
For all $t\geq 0$, let $(\bm{U}(t),P(t)) \in  H^1_0(\Omega)^d \times L^2_0(\Omega)$ be the solution of the following dual Stokes problem. 
    \begin{alignat}{2}
        -\Delta \bm{U}(t)  + \nabla P(t) &= \bm{\chi}(t), && \quad \mathrm{in}\,\, \Omega,  \label{eq:aux_pb_1}\\ 
        \nabla \cdot \bm{U}(t) & = 0, && \quad \mathrm{in } \,\, \Omega, \label{eq:aux_pb_2} \\
        \bm{U}(t) & = 0, && \quad \mathrm{on }\,\, \partial \Omega. \label{eq:aux_pb_3}
    \end{alignat} 
    Since $\bm{\chi}(t)$ belongs to $L^2(\Omega)^d$, we can assume that the solution to the above problem satisfies the following,  if the domain is convex. 
    \begin{equation}
        \|\bm{U}(t)\|_{H^2(\Omega)} + \| P(t)\|_{H^1(\Omega)} \leq C \|\bm{\chi}(t) \| .  \label{eq:regularity_assumption}
    \end{equation}
    Further, for $t\geq 0$, define $(\bm{U}_h(t), P_h(t)) \in \bm{X}_h \times M_{h0}$ as the dG solution to the auxiliary problem \eqref{eq:aux_pb_1}-\eqref{eq:aux_pb_3}. 
    \begin{alignat}{2}
        a_{\mathcal{D}}(\bm{U}_h(t), \bm{\theta}_h) - b(\bm{\theta}_h, P_h(t)) & =(\bm{\chi}(t), \bm{\theta}_h), \quad \forall \bm{\theta}_h \in \bm{X}_h, \label{eq:aux_1_dg} \\ 
        b(\bm{U}_h(t), q_h) & = 0, \quad \forall q_h \in M_{h0}. \label{eq:aux_2_dg}
    \end{alignat}
    The proofs for existence and uniqueness of $(\bm{U}_h(t), P_h(t))$ are found in \cite{riviere2008discontinuous}. We now state stability and error bounds for the dG solution of the dual problem. %The proof is skipped for brevity.
    \begin{lemma}\label{lemma:regularity_bounds_dual_problem}
%    Assume $\Omega$ is a convex domain with Lipschitz boundary. 
For $t \in [0,T]$, there exists a constant $C$ independent of $h$ such that
    \begin{align}
    \vertiii{\bm{U}(t)}&\leq C \| \bm{\chi}(t)\|,\label{eq:regularity_infinity_norm}\\ 
    \|\bm{U}(t)-\bm{U}_h(t)\| +  h\|\bm{U}(t) -\bm{U}_h(t) \|_{\DG} + h\|P(t) - P_h(t)\| &\leq Ch^2\|\bm{\chi}(t)\|, \label{eq:error_dg_aux} \\  
    \|\bm{U}_h(t)\|_{L^{\infty}(\Omega)} &\leq C \| \bm{\chi}(t) \|. \label{eq:Linf_bd_Uh} 
    \end{align}
    \end{lemma}
    \begin{proof}
    Fix $t \in [0,T]$. Note %that by assuming $\Omega$ has a  Lipschitz boundary, 
a Sobolev embedding result, see Theorem 1.4.6 in \cite{brenner2007mathematical}, and Theorem 2 in Section 5.6.1 in  \cite{evans2010partial},
 and \eqref{eq:regularity_assumption} %and Section 5.6.3 in \cite{evans2010partial} and assumption \eqref{eq:regularity_assumption}
     yield
    \begin{equation}
        \|\bm{U}(t) \|_{L^{\infty}(\Omega)} + |\bm{U}(t)|_{W^{1,3}(\Omega)} \leq C \| \bm{U}(t)\|_{H^2(\Omega)} \leq C \|\bm{\chi}(t)\|. 
    \end{equation}
    % Note that $D^{\bm{\alpha}}\bm{U}(t) \in H^{1}(\Omega)^d$ for any multi-index $\bm{\alpha}$ with $|\bm{\alpha}| = 1$. For $d = 3$, we apply Theorem 2 in Section 5.6.1 in  \cite{evans2010partial} and \eqref{eq:regularity_assumption} to conclude that for $|\bm{\alpha}| = 1$,  
    % \begin{equation}
    % \|D^{\bm{\alpha}}\bm{U}(t) \|_{L^{3}(\Omega)} \leq C\|D^{\bm{\alpha}} \bm{U}(t)\|_{W^{1,2}(\Omega)} \leq C \| \bm{U}(t)\|_{H^2(\Omega)} \leq C \|\bm{\chi}(t) \|. \label{eq:case_d_3}
    % \end{equation} 
    % For $d=2$, the above also holds as a corollary to the same theorem. Hence, \eqref{eq:regularity_infinity_norm} is obtained. 
    For $k \geq 1$, since the domain is convex, we have an error estimate for the dG error in the $L^2$ and energy norms \cite{riviere2008discontinuous,girault2005discontinuous}. Using \eqref{eq:regularity_assumption}, we obtain  \begin{align}
        \|\bm{U}(t)-\bm{U}_h(t)\|& + h\|\bm{U}(t) -\bm{U}_h(t) \|_{\DG} + h\|P(t) - P_h(t)\| \nonumber \\ & \leq Ch^2(\|\bm{U}(t)\|_{H^2(\Omega)} + \|P(t)\|_{H^1(\Omega)} ) \leq Ch^2 \|\bm{\chi}(t)\|.
    \end{align} 
    To show \eqref{eq:Linf_bd_Uh}, let $\tilde{\bm{U}}_h(t) \in X_h$ be the Lagrange  interpolant of $\bm{U}(t)$. As a result of \eqref{eq:regularity_infinity_norm}, \eqref{eq:error_dg_aux},  inverse estimate \eqref{eq:inverse_estimate2}, and approximation properties,  we obtain 
    \begin{align}
   & \| \bm{U}_h(t)\|_{L^{\infty}(\Omega)} \leq \| \bm{U}_h(t) - \tilde{\bm{U}}_h(t) \|_{L^{\infty}(\Omega)}+ \| \tilde{\bm{U}}_h(t) \|_{L^{\infty}(\Omega)} \nonumber \\ &\leq  C h^{-d/2}(\| \bm{U}_h(t) - \bm{U}(t) \| + \| \bm{U}(t) - \tilde{\bm{U}}_h(t)\|)\nonumber + C\| \bm{U}(t) \|_{L^{\infty}(\Omega)}
    %& \leq C h^{-d/2 + 2} (\| \bm{U}(t)\|_{H^2(\Omega)}+ \| P(t) \|_{H^1(\Omega)}) + C \| \bm{U}(t)\|_{L^{\infty}(\Omega)},\nonumber \\
    \leq C \|\bm{\chi}(t)\|. 
    \end{align}
    \end{proof}
    \begin{theorem}\label{theorem:improved_estimate_velocity}
    %Assume that $\Omega$ is a convex domain with Lipschitz boundary and $\mesh_h$ is uniform. %Assume that $\sigma \geq M_{k-1}^2/d$,  $\tilde{\sigma} \geq 4 \tilde{M}_k^2$, and $\delta \leq 1/(4d)$.
    Assume that if $k \geq 2$, the condition  $h^2 \leq \tau$ holds. Under the same assumptions as Lemma \ref{lemma:first_err_estimate}, 
with $\partial_t \bm{u} \in L^2(0,T ; H^{k+1}(\Omega)^d)$, we have, with $K_\mu = \sum_{i=-4}^2 \mu^i$:  
    %and $C$ independent of $h, \tau$ but depending on $\mu^{-1}$, $C = \mathcal{O}(e^{\mu^{-1}})$, 
     %such that if $\tau \leq \gamma \mu $, the following estimate holds.  
    \begin{equation}
         \mu \tau \sum_{n=1}^m  \| \bm{u}_h^n - \bm{u}^n \|^2  +  \mu \tau  \sum_{n=1}^m  \| \bm{v}_h^n - \bm{u}^n \|^2  \leq C_\mu K_\mu (\tau^2 + h^{2k+2}), \quad \forall 0\leq m \leq N_T. 
    \end{equation} 
    %Here,  $K_\mu = \sum_{i=-4}^2 \mu^i$. %This estimate holds under the additional regularity assumption: $\partial_t \bm{u} \in L^2(0,T ; H^{k+1}(\Omega)^d)$.
    \end{theorem}
   % This estimate holds under regularity assumptions: $\bm{u} \in L^{\infty}(0,T; H^{k+1}(\Omega)^d)$, $\partial_t \bm{u} \in L^2(0,T ; H^{k+1}(\Omega)^d)$, $\partial_{tt} \bm{u} \in L^2(0,T; L^2(\Omega)^d)$, and $p \in L^{\infty}(0,T; H^k(\Omega))$.
    \begin{proof}
        By consistency of the dG discretization, we have for $\bm{\theta}_h \in \bm{X}_h$ and $n \geq 1$:  
        \begin{equation}
     ((\partial_t \bm{u})^n, \bm{\theta}_h ) + a_\mathcal{C}(\bm{u}^n; \bm{u}^n, \bm{u}^n, \bm{\theta}_h) + \mu a_\mathcal{D}(\bm{u}^n, \bm{\theta}_h) =  b(\bm{\theta}_h, p^n) + (\bm{f}^n , \bm{\theta}_h ).  \label{eq:consistency_true_sol}
        \end{equation}
        Multiplying \eqref{eq:consistency_true_sol} by $\tau$, subtracting it from \eqref{eq:intermidiate_velocity}, and choosing $\bm{\theta}_h = \bm{U}_h^n = \bm{U}_h(t^n)$ yields 
        \begin{multline}
        (\bm{v}_h^n - \bm{u}^n - \bm{\chi}^{n-1},\bm{U}_h^n) + \tau \tilde{R}_{\mathcal{C}} (\bm{U}_h^n) + \tau \mu a_\mathcal{D} (\bm{v}_h^n - \bm{u}^n , \bm{U}_h^n)\\ = \tau b(\bm{U}_h^n, p_h^{n-1}- p^n) +  (\tau (\partial_t \bm{u})^n - (\bm{u}^n - \bm{u}^{n-1}), \bm{U}_h^n), \label{eq:first_error_eq_dual} \end{multline}
        where $$ \tilde{R}_{\mathcal{C}} (\bm{U}_h^n) = a_\mathcal{C}(\bm{u}_h^{n-1}; \bm{u}_h^{n-1}, \bm{v}_h^n, \bm{U}_h^{n}) -  a_\mathcal{C}(\bm{u}^{n}; \bm{u}^{n}, \bm{u}^n, \bm{U}_h^{n}).$$
        Throughout the proof, $\epsilon$ denotes a small positive constant to be determined later. Let us begin with the last two terms in \eqref{eq:first_error_eq_dual}.
        We use the fact that $\bm{U}_h^n$ satisfies \eqref{eq:aux_2_dg} and the definition of $\pi_h p^n$ to obtain the following.  
        \begin{align}
        b(\bm{U}_h^n, p_h^{n-1} - p^n) = -b(\bm{U}_h^n, p^n - \pi_h p^n) = \sum_{e \in \Gamma_h \cup \partial \Omega }\int_e \{p^n - \pi_h p^n \}[\bm{U}_h^n]\cdot \bm{n}_e. \nonumber  
        \end{align}
        Let $\Delta_e$ denote the union of the two elements sharing a face $e$. By a trace inequality, approximation property 
        \eqref{eq:l2_proj_approximation}, \eqref{eq:error_dg_aux}, and the fact that $[\bm{U}^n] =\bm{0}$ a.e on any $e \in \Gamma_h \cup \partial \Omega$ since $\bm{U}^n \in H_0^2(\Omega)^d$, we obtain 
         %%%&|b(\bm{U}_h^n, p_h^{n-1} - p^n)| \leq C \left(\sum_{e \in \Gamma_h \cup \partial \Omega} (\| p^n - \pi_h p^n\|^2_{L^2(\Delta_e)} + h^2\| \nabla(p^n - \pi_h p^n)\|^2_{L^2(\Delta_e)} ) \right)^{1/2} \nonumber \\ & \quad \left( \sum_{e \in \Gamma_h \cup \partial \Omega} \frac{\sigma}{h_e}\| [\bm{U}_h^n -\bm{U}^n ] \|_{L^2(e)}^2 \right)^{1/2} \leq Ch^k| p^n|_{H^k(\Omega)} \| \bm{U}_h^n - \bm
        %%%%{U}^n \|_{\DG}\nonumber \\ 
        %%%%& \leq C  h^{k+1} \vert p^n \vert_ {H^k(\Omega)} \| \bm{\chi}^n \| \leq   
 %%%%\epsilon \mu \| \bm{e}_h^n \|^2 + C h^{2k+2} \left(1+\frac{1}{\epsilon\mu}\right) . 
%\epsilon \mu (\| \bm{e}_h^n \|^2 + h^{2k+2} |\bm{u}^n|_{H^{k+1}(\Omega)}^2)+ \frac{C}{\epsilon\mu}
 %        h^{2k+2}.
\begin{equation}
|b(\bm{U}_h^n, p_h^{n-1} - p^n)| \leq C  h^{k+1} \vert p^n \vert_ {H^k(\Omega)} \| \bm{\chi}^n \| \leq
 \epsilon \mu \| \bm{e}_h^n \|^2 + C h^{2k+2} \left(1+\frac{1}{\epsilon\mu}\right).
   \label{eq:bd_b_dual}
        \end{equation}
        In the above, we used that 
\begin{equation}\label{eq:boundchin}
\|\bm{\chi}^n\| \leq C(h^{k+1}| \bm{u}^n |_{H^{k+1}(\Omega)} + \|\bm{e}_h^n\|),
\end{equation}
 which is  obtained by applying the triangle inequality and approximation property \eqref{eq:approximation_prop_1}. This bound will be used repeatedly in this proof. For the last term in \eqref{eq:first_error_eq_dual}, we simply have: 
        \begin{equation} |(\tau (\partial_t \bm{u})^n - (\bm{u}^n - \bm{u}^{n-1}),  \bm{U}_h^n)| \leq C\tau^2 \int_{t^{n-1}}^{t^n}  \| \partial_{tt}  \bm{u}\|^2 +  \tau \|\bm{U}_h^n\|_{\DG}^2. \label{eq:bounding_time_err_dual} \end{equation} 
We now consider the terms on the left hand side of \eqref{eq:first_error_eq_dual}.  With \eqref{eq:update_velocity_1} and \eqref{eq:aux_2_dg}, we have
\begin{align*}
    (\bm{v}_h^n - \bm{u}^n - \bm{\chi}^{n-1},\bm{U}_h^n) = (\bm{\chi}^{n} - \bm{\chi}^{n-1}, \bm{U}_h^n) -\tau b(\bm{U}_h^n, \phi_h^n ) =(\bm{\chi}^{n} - \bm{\chi}^{n-1}, \bm{U}_h^n).
\end{align*}
Note that from \eqref{eq:aux_1_dg}, \eqref{eq:aux_2_dg}, the above equality, and the symmetry of $a_{\mathcal{D}}(\cdot, \cdot)$,  we have 
\begin{align}
    (\bm{v}_h^n - \bm{u}^n - \bm{\chi}^{n-1},\bm{U}_h^n) & = a_{\mathcal{D}}(\bm{U}_h^n - \bm{U}_h^{n-1}, \bm{U}_h^n)= \frac{1}{2} \left( a_{\mathcal{D}}(\bm{U}_h^n, \bm{U}_h^n)  - a_{\mathcal{D}}( \bm{U}_h^{n-1}, \bm{U}_h^{n-1}) \right) \nonumber \\ &  + \frac{1}{2} a_{\mathcal{D}}(\bm{U}_h^{n} - \bm{U}_h^{n-1}, \bm{U}_h^{n} - \bm{U}_h^{n-1}). \label{eq:expanding_aepsi_dual}
\end{align}
In addition, we write 
\begin{align}
a_\mathcal{D} (\bm{v}_h^n - \bm{u}^n , \bm{U}_h^n) = a_\mathcal{D} (\err - \bm{e}_h^n, \bm{U}_h^n) + a_\mathcal{D} ( \bm{e}_h^n, \bm{U}_h^n) + a_\mathcal{D} (\Pi_h \bm{u}^n - \bm{u}^n, \bm{U}_h^n). \label{eq:breaking_up_aD_first_err_dual}
\end{align}
To handle the last term in the equality above, let $Q_h \bm{u}^n$ be the elliptic projection of $\bm{u}^n$ onto the space $\bm{X}_h$. Since the domain is convex, this projection satisfies  \cite{riviere2008discontinuous}: 
        \begin{equation}
    \forall \bm{\theta}_h \in \bm{X}_h, \,\, a_{\mathcal{D}}(\bm{u}^n -Q_h \bm{u}^n , \bm{\theta}_h) = 0, \quad \mathrm{and} \quad \| \bm{u}^n - Q_h \bm{u}^n \| \leq Ch^{k+1}|\bm{u}^n|_{H^{k+1}(\Omega)}. \label{eq:elliptic_projection}
        \end{equation}
         Let $\bm{\theta}_h = \Pi_h \bm{u}^n - Q_h \bm{u}^n$ in \eqref{eq:aux_1_dg}. We obtain 
        \begin{align}
             a_{\mathcal{D}} (\Pi_h \bm{u}^n - \bm{u}^n, \bm{U}_h^n) & = a_{\mathcal{D}} (\Pi_h \bm{u}^n - Q_h \bm{u}^n, \bm{U}_h^n) \nonumber \\ & =  (\bm{\chi}^n , \Pi_h \bm{u}^n - Q_h\bm{u}^n ) + b(\Pi_h \bm{u}^n - Q_h\bm{u}^n, P_h^n).
        \end{align}
 Using \eqref{eq:boundformb}, we obtain 
        \[
        | b(\Pi_h \bm{u}^n - Q_h\bm{u}^n, P_h^n)| \leq C  \|\Pi_h \bm{u}^n - Q_h \bm{u}^n \| |P_h^n|_{\DG}. 
        \]
        Note that with \eqref{eq:bd_phN}, the inverse estimate \eqref{eq:inverse_estimate_dG}, \eqref{eq:l2_proj_approximation}, \eqref{eq:error_dg_aux}, \eqref{eq:regularity_assumption}, and triangle inequality:  
\begin{align}
           |P_h^n|_{\DG} \leq & |P_h^n - \pi_h P^n|_{\DG} + |\pi_h P^n|_{\DG} \leq Ch^{-1}\|P_h^n - \pi_h P^n\|+ C| P^n|_{H^1(\Omega)} 
\label{eq:int99}\\  
             \leq & C (\|\bm{e}_h^n\| + h^{k+1}|\bm{u}^n|_{H^{k+1}(\Omega)}).\label{eq:dg_bound_ph}
        \end{align}
        % With the triangle inequality, \eqref{eq:approximation_prop_1} and \eqref{eq:elliptic_projection}, we have
        % $$\|\Pi_h \bm{u}^n - Q_h \bm{u}^n \| \leq \| \Pi_h \bm{u}^n - \bm{u}^n \| + \| \bm{u}^n - Q_h \bm{u}^n \| \leq C h^{k+1} |\bm{u}^n|_{H^{k+1}(\Omega)}.$$
        Hence, Cauchy-Schwarz's inequality, the above bounds, the assumption that $\bm{u} \in L^{\infty}(0,T;H^{k+1}(\Omega)^d)$,  and Young's inequality yield: 
        \begin{align}
        & | a_{\mathcal{D}} (\Pi_h \bm{u}^n - \bm{u}^n, \bm{U}_h^n) |  \leq C  \|\Pi_h \bm{u}^n  - Q_h \bm{u}^n \|(\|\bm{e}_h^n\| + h^{k+1}|\bm{u}^n|_{H^{k+1}(\Omega)})  \nonumber  \\ & \leq C h^{k+1 } |\bm{u}^n|_{H^{k+1}(\Omega)}(\|\bm{e}_h^n\| + h^{k+1}|\bm{u}^n|_{H^{k+1}(\Omega)}) \leq  \epsilon  \|\bm{e}_h^n \|^2 + C\left(\frac{1}{\epsilon} + 1  \right)  h^{2k+2}. \label{eq:pih_un_u_dual}
         \end{align} 
    Consider now the second term in \eqref{eq:breaking_up_aD_first_err_dual}.  Letting $\bm{\theta}_h = \bm{e}_h^n$ in \eqref{eq:aux_1_dg},  we obtain
        \begin{equation}
            a_{\mathcal{D}}(\bm{e}_h^n,\bm{U}_h^n) = \| \bm{e}_h^n \|^2 + (\bm{\chi}^n - \bm{e}_h^n , \bm{e}_h^n) +  b(\bm{e}_h^n, P_h^n). \label{eq:rewriting_a_ep_dual}
        \end{equation}
        With  \eqref{eq:bd_b_dual}, \eqref{eq:bounding_time_err_dual},\eqref{eq:expanding_aepsi_dual}, \eqref{eq:breaking_up_aD_first_err_dual}, \eqref{eq:pih_un_u_dual}, and \eqref{eq:rewriting_a_ep_dual}, the equality \eqref{eq:first_error_eq_dual} becomes 
        \begin{multline}\label{eq:second_err_eq_dual}
            \frac{1}{2} \left( a_{\mathcal{D}}(\bm{U}_h^n, \bm{U}_h^n)  - a_{\mathcal{D}}( \bm{U}_h^{n-1}, \bm{U}_h^{n-1}) + a_{\mathcal{D}}(\bm{U}_h^{n} - \bm{U}_h^{n-1}, \bm{U}_h^{n} - \bm{U}_h^{n-1})\right) + \tau \mu \| \bm{e}_h^n \|^2 \\ \leq   2\epsilon \tau \mu \|\bm{e}_h^n\|^2  + C\tau^2 \int_{t^{n-1}}^{t^n}  \| \partial_{tt}  \bm{u} \|^2 + C \tau \|\bm{U}_h^n\|_{\DG}^2   %+ C\left(\mu(1+\epsilon) + \frac{1}{\epsilon \mu}  + \frac{\mu}{\epsilon} \right)\tau h^{2k+2} \\ 
 + C \left(\mu\left(1+\frac{1}{\epsilon}\right) + 1+ \frac{1}{\epsilon\mu}\right) \tau h^{2k+2} 
\\ - \tau \tilde{R}_{\mathcal{C}}(\bm{U}_h^n)
-  \tau \mu a_{\mathcal{D}}(\err - \bm{e}_h^n, \bm{U}_h^n) -\tau \mu (\Pi_h \bm{u}^n - \bm{u}^n, \bm{e}_h^n ) - \tau \mu b(\bm{e}_h^n, P_h^n).
        \end{multline}
We now handle the last three terms in the above bound. 
% let us do in order of the terms
        Let $\bm{\theta}_h = \err - \bm{e}_h^n$ in \eqref{eq:aux_1_dg}. 
        \begin{align} 
          a_{\mathcal{D}}( \err - \bm{e}_h^n,\bm{U}_h^n)
        & = (\bm{e}_h^n, \err - \bm{e}_h^n)  +  (\Pi_h \bm{u}^n - \bm{u}^n, \err - \bm{e}_h^n) +   b(\err - \bm{e}_h^n, P_h^n). 
        \end{align} 
        Recall that by \eqref{eq:sec_err_eq}, \eqref{eq:div_error_2}, \eqref{eq:lift_prop_g}, and the assumption that $\tilde{\sigma} \geq  \tilde{M}_k^2$,  we have 
        \begin{equation}
             (\bm{e}_h^n, \err -\bm{e}_h^n) 
            = -\tau  b(\bm{e}_h^n, \phi_h^n)   = \tau^2 \sum_{e\in \Gamma_h} \frac{\tilde{\sigma}}{h_e} \| [\phi_h^n]\|_{L^2(e)}^2  - \tau^2 \| \bm{G}_h([\phi_h^n])\|^2  \geq 0.
        \end{equation}
         Using  Cauchy-Schwarz's inequality, \eqref{eq:boundformb},  and \eqref{eq:approximation_prop_1}, we have the following bound. 
        \begin{equation}
           | (\Pi_h \bm{u}^n - \bm{u}^n, \err - \bm{e}_h^n) |  +  |b(\err - \bm{e}_h^n , P_h^n) | 
        %\leq | (\Pi_h \bm{u}^n - \bm{u}^n, \err - \bm{e}_h^n)| +  |(\err - \bm{e}_h^n, \nabla_h P_h^n)| \nonumber \\  & + | (\err - \bm{e}_h^n ,  \bm{G}_h([P_h^n]))|  \leq  \| \err - \bm{e}_h^n \| ( h^{k+1}|\bm{u}^n|_{H^{k+1} (\Omega)} + \| \nabla_h P_h^n \|  + \|\bm{G}_h([P_h^n])\| ) \nonumber \\ & 
 \leq C \|\err - \bm{e}_h^n \| ( h^{k+1}|\bm{u}^n|_{H^{k+1} (\Omega)} + \vert P_h^n\vert_{\DG}).   \label{eq:bound105}
        \end{equation}
        With \eqref{eq:sec_err_eq}  and \eqref{eq:boundformb}, we immediately obtain:
\begin{equation}\label{eq:intres10}
\Vert  \err - \bm{e}_h^n\Vert \leq C \tau \vert \phi_h^n \vert_{\mathrm{DG}}.
\end{equation}
Therefore with  \eqref{eq:dg_bound_ph}, the assumption that $\bm{u} \in L^{\infty}(0,T; H^{k+1}(\Omega)^d)$,   and Young's inequality, the  bound \eqref{eq:bound105} becomes 
        \begin{align}
         | (\Pi_h \bm{u}^n - \bm{u}^n, \err - \bm{e}_h^n) |  & +   |b(\err -\bm{e}_h^n, P_h^n)|\leq  C \tau \vert \phi_h^n\vert_{\DG}( h^{k+1}|\bm{u}^n|_{H^{k+1} (\Omega)} + \| \bm{e}_h^n\|) \nonumber  \\ & \leq  \epsilon \|\bm{e}_h^n\|^2+ C \left(1+ \frac{1}{\epsilon}\right) \tau^2 |\phi_h^n|^2_{\DG}  + Ch^{2k+2}.
        \end{align}
With \eqref{eq:approximation_prop_1}, we have 
        \[ |(\Pi_h \bm{u}^n - \bm{u}^n, \bm{e}_h^n )| \leq  \epsilon \|\bm{e}_h^n\|^2 + \frac{C}{\epsilon} h^{2k+2} |\bm{u}^n|^2_{H^{k+1}(\Omega)}.\]
        To handle the last term in \eqref{eq:second_err_eq_dual}, we use \eqref{eq:div_error_2}, \eqref{eq:lift_prop_g}, and  \eqref{eq:dg_bound_ph}. 
        \begin{align}
        & | b(\bm{e}_h^n , P_h^n )|  = \left| -\tau  \sum_{e \in \Gamma_h} \frac{\tilde{\sigma}}{h_e} \int_e [\phi_h^n][P_h^n] +  \tau  (\bm{G}_h([\phi_h^n]),\bm{G}_h([P_h^n]))\right| \nonumber \\
            & \leq C\tau   \vert \phi_h^n \vert_{\DG} \vert P_h^n \vert_{\DG} \leq  \epsilon \|\bm{e}_h^n\|^2  + C \left(1+ \frac{1}{\epsilon}\right) \tau^2 |\phi_h^n|^2_{\DG}+ Ch^{2k+2}.
        \end{align} 
        With the above bounds combined, \eqref{eq:second_err_eq_dual} becomes  
        \begin{multline}
            \frac{1}{2} \left( a_{\mathcal{D}}(\bm{U}_h^n, \bm{U}_h^n)  - a_{\mathcal{D}}( \bm{U}_h^{n-1}, \bm{U}_h^{n-1}) + a_{\mathcal{D}}(\bm{U}_h^{n} - \bm{U}_h^{n-1}, \bm{U}_h^{n} - \bm{U}_h^{n-1})\right) + \tau \mu \| \bm{e}_h^n \|^2  \\ \leq 5\epsilon \tau \mu \|\bm{e}_h^n\|^2 + C  \mu \left(1+ \frac{1}{\epsilon}\right)\tau^3 \vert \phi_h^n \vert_{\DG}^2  - \tau \tilde{R}_{\mathcal{C}}(\bm{U}_h^n)  + C\tau^2 \int_{t^{n-1}}^{t^n}  \| \partial_{tt} \bm{u}^n\|^2 \\ 
+ C \tau \|\bm{U}_h^n\|_{\DG}^2   + C\left(\mu(1+\frac{1}{\epsilon}) + 1+ \frac{1}{\epsilon \mu} \right)\tau h^{2k+2}. \label{eq:improved_estimate_semi_final}
        \end{multline}
        We now want to handle the nonlinear term $\tilde{R}_{\mathcal{C}}(\bm{U}_h^n)$. We write
        \begin{align}
        \tilde{R}_{\mathcal{C}}(\bm{U}_h^n) % &= a_\mathcal{C}(\bm{u}_h^{n-1}; \bm{u}_h^{n-1}, \bm{v}_h^n, \bm{U}_h^{n}) -  a_\mathcal{C}(\bm{u}^{n}; \bm{u}^{n}, \bm{u}^n, \bm{U}_h^{n}), \nonumber \\ 
        %& = a_\mathcal{C}(\bm{u}_h^{n-1}; \bm{u}_h^{n-1}, \bm{v}_h^n, \bm{U}_h^{n}) -  a_\mathcal{C}(\bm{u}^{n}; \bm{u}^{n} - \bm{u}_h^{n-1}, \bm{u}^n, \bm{U}_h^{n}) - a_\mathcal{C}(\bm{u}^{n};\bm{u}_h^{n-1}, \bm{u}^n, \bm{U}_h^{n}), \nonumber  \\ 
        & =  a_\mathcal{C}(\bm{u}_h^{n-1}; \bm{u}_h^{n-1}, \bm{v}_h^n - \bm{u}^n, \bm{U}_h^{n}) - a_\mathcal{C}(\bm{u}^{n}; \bm{u}^{n} - \bm{u}_h^{n-1}, \bm{u}^n, \bm{U}_h^{n})  = T^n_1 + T^n_2.  \nonumber
        \end{align}
To handle $T^n_1$, we resort to \eqref{eq:integration_by_parts_c}. We have, 
\[
T^n_1 %& = - \bar{a}_\mathcal{C}(\bm{u}_h^{n-1}; \bm{u}_h^{n-1}, \bm{U}_h^n,  \bm{v}_h^n - \bm{u}^n) \\ 
 = -\bar{a}_\mathcal{C}(\bm{u}_h^{n-1}; \bm{u}_h^{n-1}, \bm{U}_h^n -\bm{U}^n, \bm{v}^n_h - \bm{u}^n) - \bar{a}_\mathcal{C}(\bm{u}_h^{n-1}; \bm{u}_h^{n-1},\bm{U}^n, \bm{v}^n_h - \bm{u}^n).
\]
%%%%%%%Since
%%%%%\begin{equation} \label{eq:observation_on_U}
 %%%%%\bm{U}^n|_{\partial \Omega}  = \bm{0}, \quad a.e.  \quad\mbox{and}\quad
%%%%%((\bm{U}^n)^{\mathrm{int}} -  (\bm{U}^n)^{\mathrm{ext}})|_{\partial E} = \bm{0},\quad a.e.  \quad \forall E\in\mesh_h,
%%%%%\end{equation}
The upwind term in the second term in the expression for $T^n_1$ vanishes and it becomes linear in the second argument. Thus, we can further split it in the expression for $T^n_1$: 
\begin{align}
& T^n_1  =-\bar{a}_\mathcal{C}(\bm{u}_h^{n-1}; \bm{u}_h^{n-1}, \bm{U}_h^n -\bm{U}^n, \bm{v}^n_h - \bm{u}^n) - \bar{a}_\mathcal{C}(\bm{u}_h^{n-1}; \bm{e}_h^{n-1},\bm{U}^n, \bm{v}^n_h - \bm{u}^n) \nonumber\\
&  -  \bar{a}_\mathcal{C}(\bm{u}_h^{n-1}; \Pi_h \bm{u}^{n-1} - \bm{u}^{n-1},\bm{U}^n, \bm{v}^n_h - \bm{u}^n) - \bar{a}_\mathcal{C}(\bm{u}_h^{n-1}; \bm{u}^{n-1},\bm{U}^n, \bm{v}^n_h - \bm{u}^n) = \sum_{i=1}^4 A^n_i.  \nonumber
 \end{align}
We postpone handling $A^n_1$ till the end as it is more intricate than the other terms. We begin with $A^n_2$. With \eqref{eq:integration_by_parts_c}, we write 
\begin{align*}
A^n_2 = - \bar{a}_\mathcal{C}(&\bm{e}_h^{n-1}; \bm{e}_h^{n-1},\bm{U}^n, \bm{v}^n_h - \bm{u}^n)  =  a_\mathcal{C}(\bm{e}_h^{n-1}; \bm{e}_h^{n-1}, \bm{v}_h^n - \bm{u}^n , \bm{U}^n)\\
& = a_\mathcal{C}(\bm{e}_h^{n-1}; \bm{e}_h^{n-1}, \err , \bm{U}^n) + a_\mathcal{C}(\bm{e}_h^{n-1}; \bm{e}_h^{n-1},\Pi_h \bm{u}^n  - \bm{u}^n, \bm{U}^n) = A^n_{2,1} + A^n_{2,2}. 
\end{align*}
The term $A^n_{2,1}$ reads
\begin{multline}
 A^n_{2,1} =  \sum_{E \in \mesh_h} \int_E (\bm{e}_h^{n-1} \cdot \nabla \err) \cdot \bm{U}^n  + \frac{1}{2} b(\bm{e}_h^{n-1}, \err \cdot \bm{U}^n) \\ +   \sum_{E \in \mesh_h} \int_{\partial E_{-}^{\bm{e}^{n-1}_h}} | \{\bm{e}^{n-1}_h\} \cdot \bm{n}_E | ((\err)^{\mathrm{int}} - (\err)^{\mathrm{ext}}) \cdot (\bm{U}^n)^{\mathrm{int}}. \nonumber 
\end{multline}
 With H\"{o}lder's inequality and \eqref{eq:regularity_infinity_norm}, we bound the first term: 
     \begin{align*}
     \sum_{E \in \mesh_h} \left|\int_E (\bm{e}_h^{n-1} \cdot \nabla \err) \cdot \bm{U}^n \right|  &\leq \|\bm{e}_h^{n-1} \| \|\nabla_h \err \|\|\bm{U}^n \|_{L^{\infty}(\Omega)} \leq C \|\bm{e}_h^{n-1} \| \| \err \|_{\DG} \|\bm{\chi}^n\|.
     \end{align*}  
We use \eqref{eq:first_estimate_nonlinear_form} to bound the second term in $A^n_{2,1}$. With \eqref{eq:regularity_infinity_norm}, we have 
\begin{equation}
    \frac{1}{2} |b(\bm{e}_h^{n-1}, \err\cdot \bm{U}^n)| \leq C\|\bm{e}_h^{n-1}\|\vertiii{\bm{U}^n}\|\err\|_{\DG}  \leq C\|\bm{e}_h^{n-1}\| \|\bm{\chi}^n\|\|\err\|_{\DG}.\nonumber
\end{equation}
 With trace estimate \eqref{eq:trace_ineq_discrete} and \eqref{eq:regularity_infinity_norm}, the upwind term in $ A^n_{2,1}$ is bounded by: 
    \begin{align*}
     C\|\bm{e}_h^{n-1} \| \|\err\|_{\DG} \|\bm{U}^n \|_{L^{\infty}(\Omega)} \leq C \| \bm{e}_h^{n-1} \| \| \err \|_{\DG} \| \bm{\chi}^n \|.  
    \end{align*}
Combining the above bounds and applying Young's and triangle inequality yields:
\begin{equation}
|A^n_{2,1} | \leq  \epsilon \mu \|\bm{e}_h^n\|^2 + C \left(1+ \frac{1}{\epsilon\mu}\right)\|\bm{e}_h^{n-1}\|^2 \|\err\|_{\DG}^2 + Ch^{2k+2}. \label{eq:bd_a_21}
\end{equation}
We now handle $A^n_{2,2}$, it reads 
\begin{multline}
 A^n_{2,2} = \sum_{E \in \mesh_h} \int_{E} (\bm{e}^{n-1}_h \cdot \nabla (\Pi_h\bm{u}^n - \bm{u}^n)) \cdot \bm{U}^n  + \frac{1}{2} b(\bm{e}^{n-1}_h, (\Pi_h \bm{u}^n - \bm{u}^n) \cdot \bm{U}^n)\\  +  \sum_{E \in \mesh_h} \int_{\partial E_{-}^{\bm{e}^{n-1}_h}} | \{\bm{e}^{n-1}_h\} \cdot \bm{n}_E | ((\Pi_h \bm{u}^n - \bm{u}^n)^{\mathrm{int}} - (\Pi_h \bm{u}^n - \bm{u}^n)^{\mathrm{ext}} )\cdot (\bm{U}^n)^{\mathrm{int}}. \nonumber
\end{multline}
With Holder's inequality, approximation property \eqref{eq:approximation_prop_2}, and \eqref{eq:regularity_infinity_norm}: 
\begin{align}
 \sum_{E \in \mesh_h} \int_{E} \left \vert (\bm{e}^{n-1}_h \cdot \nabla (\Pi_h\bm{u}^n - \bm{u}^n)) \cdot \bm{U}^n \right\vert &\leq  \|\bm{e}_h^{n-1} \| \|\nabla_h (\Pi_h \bm{u}^n - \bm{u}^n) \| \| \bm{U}^n \|_{L^{\infty}(\Omega)} \nonumber \\ &\leq  Ch^{k} |\bm{u}^n|_{H^{k+1}(\Omega)}\|\bm{e}_h^{n-1}\|\|\bm{\chi}^n\|. \nonumber
\end{align}
For the second term in $ A^n_{2,2}$, %we have 
%\begin{align}
%\frac{1}{2}b(\bm{e}_h^{n-1}, (\Pi_h \bm{u}^n - \bm{u}^n)\cdot \bm{U}^n )& =\frac{1}{2} \sum_{E \in \mesh_h }\int_{E} \nabla \cdot \bm{e}_h^{n-1}  (\Pi_h \bm{u}^n - \bm{u}^n)\cdot \bm{U}^n  \nonumber \\ & \quad - \frac{1}{2} \sum_{e\in \Gamma_h \cup \partial \Omega} \int_e \{ \Pi_h \bm{u}^n - \bm{u}^n \} \cdot \bm{U}^n [\bm{e}_h^{n-1}]\cdot \bm{n}_e.  \nonumber
%\end{align}
applying Holder's inequality, \eqref{eq:approximation_prop_1}-\eqref{eq:approximation_prop_2}, \eqref{eq:regularity_infinity_norm}, inverse estimate \eqref{eq:inverse_estimate_dG} and trace estimates \eqref{eq:trace_ineq_continuous}-\eqref{eq:trace_ineq_discrete}, we obtain: 
\begin{align}
\frac{1}{2}|b(\bm{e}_h^{n-1}, (\Pi_h \bm{u}^n - \bm{u}^n)\cdot \bm{U}^n)|  \leq C h^{-1} \| \bm{e}_h^{n-1} \| \| \Pi_h \bm{u}^n - \bm{u}^n \| \| \bm{U}^n \|_{L^{\infty}(\Omega)} \nonumber \\ + C \| \bm{U}^n \|_{L^{\infty}(\Omega)}(h^{-1}\| \Pi_h \bm{u}^n - \bm{u}^n  \| +  \| \nabla_h(\Pi_h \bm{u}^n - \bm{u}^n )\| ) \|\bm{e}_h^{n-1}\| \nonumber \\ 
\leq Ch^{k} |\bm{u}^n|_{H^{k+1}(\Omega)}\|\bm{e}_h^{n-1}\|\|\bm{\chi}^n\|. \label{eq:bounding_b_eh_pih_u}
\end{align}
The upwind term in $ A^n_{2,2}$ is bounded similarly by  \begin{align*}
C  \|\bm{e}_h^{n-1}\| \|\bm{U}^n \|_{L^{\infty}(\Omega)} (h^{-1} \| \bm{u}^n - \Pi_h \bm{u}^n \| + \| \nabla_h(\bm{u}^n - \Pi_h \bm{u}^n )\| ) \\ \leq C  h^{k} |\bm{u}^n|_{H^{k+1}(\Omega)}\|\bm{e}_h^{n-1}\|\|\bm{\chi}^n\|.
\end{align*}
Hence, with the assumption that $\bm{u} \in L^{\infty}(0,T;H^{k+1}(\Omega)^d)$ and Young's inequality, we attain the following bound for $A^n_{2,2}$. 
\begin{equation} \label{eq:bd_a_22}
    |A^n_{2,2}| \leq \epsilon \mu \| \bm{e}_h^n \|^2 + C\left( \frac{1}{\epsilon \mu} +  1 \right) h^{2k}\|\bm{e}_h^{n-1}\|^2 + Ch^{2k+2}.
\end{equation}
With \eqref{eq:bd_a_21} and \eqref{eq:bd_a_22}, we obtain a bound on $A_2^n$: 
\begin{equation}
 |A^n_2| \leq 2\epsilon \mu \|\bm{e}_h^n\|^2  + C\left( \frac{1}{\epsilon \mu} +  1 \right) \left(  h^{2k}\|\bm{e}_h^{n-1}\|^2 + \|\bm{e}_h^{n-1}\|^2 \|\err\|_{\DG}^2 \right) + Ch^{2k+2}. \label{eq:bd_A_2}
\end{equation}
We now focus on $ A_3^n$ and write
\begin{align*}
A^n_3 & = - \sum_{E \in \mesh_h} \int_{E} ( (\Pi_h \bm{u}^{n-1} - \bm{u}^{n-1}) \cdot \nabla ((\bm{U}^n - \bm{U}_h^n ) + \bm{U}_h^n ) )\cdot (\bm{v}_h^n - \bm{u}^n ) \\& \quad - \frac{1}{2} \sum_{E \in \mesh_h} \int_E (\nabla \cdot (\Pi_h\bm{u}^{n-1}- \bm{u}^{n-1}))(\bm{U}^n \cdot (\bm{v}_h^n - \bm{u}^n )) \\ &\quad  + \frac{1}{2} \sum_{e\in \Gamma_h \cup \partial \Omega}  \int_e [\Pi_h \bm{u}^{n-1} - \bm{u}^{n-1}] \cdot \bm{n}_e \left(\{ \bm{U}^n \cdot \err \} +  \{ \bm{U}^n \cdot (\Pi_h \bm{u}^n - \bm{u}^n)\}\right)  = \sum_{i=1}^3 A^n_{3,i}. 
\end{align*}
Applying H\"{o}lder's inequality and trace estimates yields: 
\begin{align*}
 |A^n_{3,1}| &\leq C\|\Pi_h \bm{u}^{n-1} - \bm{u}^{n-1} \|_{L^{\infty}(\Omega)} (\|\nabla_h (\bm{U}^n - \bm{U}_h^n) \| + \|\bm{U}_h^n\|_{\DG}) \| \bm{v}_h^n - \bm{u}^n \|, \\ 
  |A^n_{3,2}| & \leq  C \|\nabla_h (\Pi_h \bm{u}^{n-1} - \bm{u}^{n-1}) \| \| \bm{U}^n \|_{L^{\infty}(\Omega)} \|\bm{v}_h^n - \bm{u}^n \|, \\   
  |A^n_{3,3} |& \leq  C( h^{-1}\|\Pi_h \bm{u}^{n-1} - \bm{u}^{n-1} \| + \|\nabla_h (\Pi_h \bm{u}^{n-1} - \bm{u}^{n-1})\| )\| \bm{U}^n \|_{L^{\infty}(\Omega)} \| \err \| \\ 
& \quad + C( h^{-1/2}\|\Pi_h \bm{u}^{n-1} - \bm{u}^{n-1} \| + h^{1/2}\|\nabla_h (\Pi_h \bm{u}^{n-1} - \bm{u}^{n-1})\|) \| \bm{U}^n \|_{L^{\infty}(\Omega)} \\ & \quad \quad ( h^{-1/2}\|\Pi_h \bm{u}^{n} - \bm{u}^{n} \| + h^{1/2}\|\nabla_h (\Pi_h \bm{u}^{n} - \bm{u}^{n})\|) .
\end{align*}
Using \eqref{eq:approximation_prop_1}, \eqref{eq:approximation_prop_2}, \eqref{eq:stability_linf_pih}, \eqref{eq:error_dg_aux},  and \eqref{eq:regularity_infinity_norm}, we obtain: 
\begin{align*}
    |A^n_3| &\leq C (h  \vertiii{\bm{u}^{n-1}} \|\bm{\chi}^n \| +  \vertiii{\bm{u}^{n-1}}  \| \bm{U}_h^n\|_{\DG} + h^{k}|\bm{u}^{n-1}|_{H^{k+1}(\Omega)}\|\bm{\chi}^n\|)\|\bm{v}_h^n - \bm{u}^n\| \nonumber \\ &\quad  + C h^{k} |\bm{u}^{n-1}|_{H^{k+1}(\Omega)}\|\bm{\chi}^n\|\| \err\| + C h^{2k+1}| \bm{u}^n|_{H^{k+1}(\Omega)} | \bm{u}^{n-1}|_{H^{k+1}(\Omega)}\| \bm{\chi}^{n} \|.
\end{align*}
In view of  \eqref{eq:intres10} and  the triangle inequality, we have 
\begin{align}
 \| \err \| \leq \| \err - \bm{e}_h^n  \| + \|\bm{e}_h^n \| 
\leq C \tau |\phi_h^n|_{\DG} + \| \bm{e}_h^n \|.  \label{eq:bd_l2_ehn}
\end{align}
Similarly,  with \eqref{eq:update_velocity_1}, \eqref{eq:boundformb} and \eqref{eq:approximation_prop_1},  we obtain 
\begin{align}
\|\bm{v}_h^n - \bm{u}^n \| %&\leq \Rd C \tau \vert \phi_h^n \vert_{\mathrm{DG}} \Bk +  \| \bm{e}_h^n \| + \|\Pi_h \bm{u}^n - \bm{u}^n \| \nonumber  \\
& \leq C\tau \vert \phi_h^n \vert_{\DG}+  \|\bm{e}_h^n\| + Ch^{k+1}|\bm{u}^n|_{H^{k+1}(\Omega)}. \label{eq:bd_vn_un}
\end{align}
Using \eqref{eq:bd_l2_ehn}, \eqref{eq:bd_vn_un}, and the assumption that $\bm{u} \in L^{\infty}(0,T;H^{k+1}(\Omega)^d)$, we have:  
\begin{align*}
     |A^n_3| 
     \leq  C (h\|\bm{\chi}^n \| + \| \bm{U}_h^n\|_{\DG}  &)( \|\bm{e}_h^n\|  +\tau  \vert \phi_h^n \vert_{\DG} + h^{k+1}) +  C h^{2k+1} \| \bm{\chi}^{n} \|. % \\ & \quad  + C  h^{k} \| \bm{e}_h^n\|(\|\bm{e}_h^n\|+ \tau|\phi_h^{n}|_{\DG}) +  C h^{2k+1} \| \bm{e}_h^{n} \|. 
    \end{align*}
 With \eqref{eq:boundchin}, this yields
\[
|A_3^n| \leq
C (h\|\bm{e}_h^n\| + h^{k+2}  + \|\bm{U}_h^n \|_{\DG}) (\|\bm{e}_h^n\| + h^{k+1} + \tau |\phi_h^n |_{\DG})
+ C h^{3k+2} + C h^{2k+1} \Vert \bm{e}_h^n\Vert.  
\]
To handle $A^n_4$, we note that it reduces to
\begin{align*}
      A^n_4 = -  \sum_{E \in \mesh_h }\int_E (\bm{u}^{n-1} \cdot \nabla \bm{U}^n) \cdot (\bm{v}_h^n - \bm{u}^n ).  %\\
    %& = -  \sum_{E \in \mesh_h }\int_E (\bm{u}^{n-1} \cdot \nabla (\bm{U}^n - \bm{U}_h^n) )\cdot (\bm{v}_h^n - \bm{u}^n ) -   \sum_{E \in \mesh_h }\int_E (\bm{u}^{n-1} \cdot \nabla \bm{U}_h^n ) \cdot (\bm{v}_h^n - \bm{u}^n ).
 \end{align*}
 Thus, with H\"{o}lder's inequality, \eqref{eq:error_dg_aux},  \eqref{eq:boundchin}, \eqref{eq:bd_vn_un}, and $\bm{u} \in L^{\infty}(0,T; H^{k+1}(\Omega)^d)$: 
 \begin{align*}
 |A^n_4|  & \leq  \| \bm{u}^{n-1}\|_{L^{\infty}(\Omega)} ( \|\nabla_h (\bm{U}^n - \bm{U}_h^n ) \|+ \|\nabla_h \bm{U}_h^n\| ) \|\bm{v}_h^n - \bm{u}^n  \| \\ 
& \quad \leq C (h\|\bm{e}_h^n\| + h^{k+2}  + \|\bm{U}_h^n \|_{\DG}) (\|\bm{e}_h^n\| + h^{k+1} + \tau |\phi_h^n |_{\DG}).  
 \end{align*}
 %\begin{align*}
     %|A^n_3| + |A^n_4| & \leq  \| \bm{u}^{n-1}\|_{L^{\infty}(\Omega)} ( \|\nabla_h (\bm{U}^n - \bm{U}_h^n ) \|+ \|\nabla_h \bm{U}_h^n\| ) \|\bm{v}_h^n - \bm{u}^n  \| \\ & \quad \leq C (h\|\bm{e}_h^n\| + h^{k+2}  + \|\bm{U}_h^n \|_{\DG}) (\|\bm{e}_h^n\| + h^{k+1} + \tau |\phi_h^n |_{\DG}).  
 %\end{align*}
Applying Young's inequality, we obtain the following bound for $|A_3^n| + |A_4^n|$.  
\begin{multline} \label{eq:bd_A_3}
 |A^n_3| + |A^n_4| \leq \epsilon  \mu \| \bm{e}_h^n\|^2+ C\left(\frac{1}{\epsilon\mu} +1 \right) \|\bm{U}_h^n\|_{\DG}^2+  C \left(\frac{h^2}{\epsilon \mu} + 1\right) h^{2k+2} \\ + \frac{C}{\epsilon \mu} h^{2} \|\bm{e}_h^n \|^2 + C\tau^2\left(\frac{h^2}{\epsilon \mu} + 1\right) |\phi_h^n |^2_{\DG}.
\end{multline}
% With Young's inequality, we attain 
% \begin{multline}
%  |A^n_4| \leq \epsilon \mu \|\bm{e}_h^n\|^2 + \frac{C}{\epsilon \mu} h^2 \| \bm{e}_h^n \|^2 + C\left(\frac{h^2}{\epsilon \mu}+1\right) h^{2k+2} \\ + C\tau^2\left(\frac{h^2}{\epsilon \mu}+1\right) |\phi_h^n|_{\DG}^2 + C \left(\frac{1}{\epsilon\mu}+1\right) \| \bm{U}_h^n \|_{\DG}^2. \label{eq:bd_A_4} 
% \end{multline}
We now consider $A^n_1$. It reads:
\begin{align*}
 A^n_1 &= - \sum_{E \in \mesh_h} \int_E (\bm{u}_h^{n-1} \cdot \nabla(\bm{U}_h^n - \bm{U}^n ))\cdot(\bm{v}_h^n - \bm{u}^n ) - \frac{1}{2}b(\bm{u}_h^{n-1}, (\bm{U}_h^n - \bm{U}^n)\cdot (\bm{v}_h^n - \bm{u}^n)) \\& \quad -   \sum_{E\in\mesh_h} \int_{\partial E_{-}^{\bm{u}_h^{n-1}} \backslash \partial \Omega} |\{\bm{u}_h^{n-1}\}\cdot \bm{n}_E| ((\bm{U}_h^n - \bm{U}^n)^{\mathrm{int} } - (\bm{U}_h^n - \bm{U}^n)^{\mathrm{ext}})\cdot (\bm{v}_h^n - \bm{u}^n)^{\mathrm{ext}} \\& \quad + \frac{1}{2}  \sum_{e\in \partial \Omega} \int_e (|\bm{u}_h^{n-1} \cdot \bm{n}_e|- \bm{u}_h^{n-1} \cdot \bm{n}_e ) (\bm{U}_h^n -\bm{U}^{n})\cdot (\bm{v}_h^n - \bm{u}^n)  =  \sum_{i=1}^4 A^n_{1,i}.
\end{align*}
% We split $A^n_{1,1}$ as follows. 
% \begin{equation*}
%  A^n_{1,1} = -  \sum_{E \in \mesh_h}  \int_E ( (\bm{e}_h^{n-1} +\Pi_h \bm{u}^{n-1}) \cdot \nabla(\bm{U}_h^n - \bm{U}^n ) ) \cdot(\bm{v}_h^n - \bm{u}^n ).
% \end{equation*}
Writing $\bm{u}_h^{n-1} = \bm{e}_h^{n-1} + \Pi_h \bm{u}^{n-1}$ and using H\"{o}lder's  inequality, we obtain 
\begin{align*}
 |A^n_{1,1}| & \leq  \|\nabla_h(\bm{U}_h^n - \bm{U}^n) \|( \|\bm{e}_h^{n-1}\|_{L^3(\Omega)} \|\bm{v}_h^n - \bm{u}^n \|_{L^{6}(\Omega)}   + \| \Pi_h \bm{u}^{n-1}\|_{L^{\infty}(\Omega)}\|\bm{v}_h^n - \bm{u}^n \|).
\end{align*} 
%We recall the inverse estimate: 
%\begin{equation}\label{eq:inverse_estimate}
%    \|\bm{\theta}_h \|_{L^3(\Omega)} \leq C h^{-d/6} \| \bm{\theta}_h \|, \quad \forall \bm{\theta}_h \in \bm{X}_h. 
%\end{equation}
By using \eqref{eq:discreter_poincare},\eqref{eq:approximation_prop_1}, and a Sobolev embedding result, %see Theorem 4.12 in \cite{adams2003sobolev}, 
we obtain 
\begin{align}
\| \bm{v}_h^n - \bm{u}^n \|_{L^{6}(\Omega)} & \leq \| \bm{v}_h^n - \Pi_h \bm{u}^n  \|_{L^{6}(\Omega)} + \| \Pi_h \bm{u}^n - \bm{u}^n \|_{L^{6}(\Omega)} \nonumber \\ &
\leq C \| \err \|_{\DG} +  Ch^{k}\vert \bm{u}^n\vert_{W^{k,6}(\Omega)} \leq C\| \err \|_{\DG} +  Ch^{k}\vert \bm{u}^n\vert_{H^{k+1}(\Omega)}  . \label{eq:bd_vh_uh_lr} 
\end{align}
With bounds \eqref{eq:bd_vh_uh_lr}, \eqref{eq:bd_vn_un}, \eqref{eq:inverse_estimate}, \eqref{eq:stability_linf_pih},  and \eqref{eq:error_dg_aux}, we have 
\begin{align*}
 |A^n_{1,1}| & \leq C \|\bm{\chi}^n \|\|\bm{e}_h^{n-1}\|(\|\err\|_{\DG} + h^{k}|\bm{u}^n|_{H^{k+1}(\Omega)})  \\ & \quad + C h \|\bm{\chi}^n \| \vertiii{\bm{u}^{n-1}}(\| \bm{e}_h^n\| + h^{k+1} + \tau \vert \phi_h^n \vert_{\DG}).
\end{align*}
We apply  \eqref{eq:boundchin},  the assumption that $\bm{u} \in L^{\infty}(0,T; H^{k+1}(\Omega)^d)$ and Young's inequality. %We obtain  
\begin{align}
& |A^n_{1,1}|% &\leq C \|\bm{e}_h^n \|\|\bm{e}_h^{n-1}\|\|\err \|_{\DG}+ C h^{k} \|\bm{e}_h^n\|\|\bm{e}_h^{n-1}\| + C h \|\bm{e}_h^n \|^2 \nonumber \\ & \quad + Ch^{k+2} \|\bm{e}_h^n \| + C\tau h \|\bm{e}_h^n \| \vert \phi_h^n \vert_{\DG}, \nonumber \\ 
 \leq \epsilon \mu  \|\bm{e}_h^n\|^2 + C\left(\frac{1}{\epsilon\mu} + 1 \right)h^2 (\| \bm{e}_h^n\|^2 + h^{2k+2} + \tau^2 |\phi_h^n|^2_{\DG})\nonumber \\ & +C \left( \frac{1}{\epsilon\mu} + 1\right) (\|\bm{e}_h^{n-1}\|^2\| \err \|_{\DG}^2 + h^{2k}\| \bm{e}_h^{n-1} \|^2) + Ch^{2k+2}.  %\nonumber \\ & \quad + \|\bm{e}_h^{n} \|^2. % +\frac{C}{\mu} h^{2k+4} + \frac{C}{\mu}\tau^2 h^2 \vert \phi_h^n  \vert^2_{\DG}. 
\label{eq:bd_A_111} 
\end{align}
To handle $A^n_{1,2}$, we split it into:
\begin{align*}
 A^n_{1,2} = -\frac{1}{2} b(\bm{u}_h^{n-1}, (\bm{U}_h^n - \bm{U}^n)\cdot \err ) - \frac{1}{2} b(\bm{u}_h^{n-1},(\bm{U}_h^n - \bm{U}^n)\cdot(\Pi_h \bm{u}^n - \bm{u}^n )).
\end{align*}
% The first term reads: 
% \begin{align*}
%  -\frac{1}{2} b(\bm{u}_h^{n-1}, (\bm{U}_h^n - \bm{U}^n)\cdot \err ) & = -\frac{1}{2}\sum_{E\in \mesh_h } \int_{E } (\nabla \cdot \bm{u}_h^{n-1} ) (\bm{U}_h^n - \bm{U}^n) \cdot \err \nonumber  \\ & \quad  + \frac{1}{2} \sum_{e\in\Gamma_h \cup \partial \Omega} \int_e \{ (\bm{U}_h^n - \bm{U}^n) \cdot \err \} [\bm{u}_h^{n-1}] \cdot \bm{n}_e. 
%  \end{align*}
 With Holder's inequality and trace estimate \eqref{eq:trace_ineq_discrete}, we have the bound: 
 \begin{multline*}
\frac{1}{2}|b(\bm{u}_h^{n-1}, (\bm{U}_h^n - \bm{U}^n)\cdot \err)| \leq C  \| \nabla_h \bm{u}_h^{n-1} \| \|\bm{U}_h^n - \bm{U}^n \|_{L^3(\Omega)} \| \err \|_{L^6(\Omega)} \\   \nonumber \quad + C \| \err \|_{L^6(\Omega)}\| \bm{u}_h^{n-1}\|_{L^3(\Omega)}h^{-1}  \left( \|\bm{U}_h^n - \bm{U}^n  \| + h \| \nabla_h (\bm{U}_h^n - \bm{U}^n )\|\right).   
 \end{multline*}
 %Note that with an inverse estimate, \eqref{eq:approximation_prop_1}, \eqref{eq:error_dg_aux}, \eqref{eq:regularity_assumption} and \eqref{eq:regularity_infinity_norm}, we have 
 %\begin{align}
  %   \| & \bm{U}_h^n - \bm{U}^n \|_{L^3(\Omega)} \leq  \| \bm{U}_h^n - \Pi_h \bm{U}^n  \|_{L^3(\Omega)} + \| \Pi_h \bm{U}^n - \bm{U}^n \|_{L^3(\Omega)} \nonumber \\ 
   %  & \leq C h^{-d/6} (\| \bm{U}_h^n -\bm{U}^n  \| + \|\bm{U}^n- \Pi_h \bm{U}^n  \|) +  C h \Vert \nabla \bm{U}^n \Vert_{L^3(\Omega)} \leq C h \| \bm{e}_h^n \|. \label{eq:bd_diff_L4}
 %\end{align}
 Note that with Poincare's inequality \eqref{eq:discreter_poincare} and \eqref{eq:error_dg_aux}, we have 
 \begin{equation}
     \| \bm{U}_h^n - \bm{U}^n \|_{L^3(\Omega)} \leq C_P \| \bm{U}_h^n - \bm{U}^n\|_{\DG} \leq C h \|\bm{\chi}^n\|.\label{eq:bd_diff_L4}
 \end{equation} 
 With \eqref{eq:bd_diff_L4}, \eqref{eq:error_dg_aux}, \eqref{eq:discreter_poincare}, triangle inequality, \eqref{eq:inverse_estimate_dG},  and \eqref{eq:inverse_estimate}, we obtain: 
 \begin{align}
    \frac{1}{2}|b(&\bm{u}_h^{n-1}, (\bm{U}_h^n - \bm{U}^n)\cdot \err)|  \leq C( h^{-1}\| \bm{e}_h^{n-1}\| + \Vert \nabla_h (\Pi_h \bm{u}^{n-1}) \Vert) h \| \bm{\chi}^n \|\|\err\|_{\DG}\nonumber \\ &  + C \|\err\|_{\DG}(h^{-d/6}\|\bm{e}_h^{n-1}\|+\| \Pi_h \bm{u}^{n-1} \|_{L^3(\Omega)}) h\|\bm{\chi}^{n}\| \nonumber \\ 
     & \leq C( \|\bm{e}_h^n\|  +h^{k+1}|\bm{u}^n|_{H^{k+1}(\Omega)}) \left( \|\bm{e}_h^{n-1} \| +h \|\bm{u}^{n-1}\|_{W^{1,3}(\Omega)}\right)\| \err\|_{\DG}. \label{eq:bd_first_term_A21}
 \end{align}
For the second term in the expression for $ A^n_{1,2}$, we write $\bm{u}_h^{n-1} = \bm{e}_h^{n-1} + \Pi_h \bm{u}^{n-1}$.
%\begin{align*}
% -\frac{1}{2}  b(\bm{u}_h^{n-1}, (\bm{U}_h^n - &\bm{U}^n)\cdot (\Pi_h\bm{u}^n - \bm{u}^n) )  = -\frac{1}{2}\sum_{E\in \mesh_h } \int_{E } \nabla \cdot \bm{u}_h^{n-1}  (\bm{U}_h^n - \bm{U}^n) \cdot (\Pi_h\bm{u}^n - \bm{u}^n)  \nonumber  \\ & \quad + \frac{1}{2} \sum_{e\in\Gamma_h \cup \partial \Omega} \int_e \{ (\bm{U}_h^n - \bm{U}^n) \cdot (\Pi_h\bm{u}^n - \bm{u}^n)\} [\bm{e}_h^{n-1}+\Pi_h \bm{u}^{n-1}] \cdot \bm{n}_e. 
 %\end{align*}
With Holder's inequality, inverse estimate \eqref{eq:inverse_estimate_dG} and trace estimates \eqref{eq:trace_ineq_continuous}-\eqref{eq:trace_ineq_discrete}, \eqref{eq:error_dg_aux}, \eqref{eq:bd_diff_L4} and \eqref{eq:approximation_prop_1} - \eqref{eq:approximation_prop_2}, this term is bounded by  
\begin{align}
 &  C (h^{-1}\|\bm{e}_h^{n-1}\| \|\bm{U}_h^n - \bm{U}^n\|_{L^3(\Omega)} +  \| \nabla_h \cdot \Pi_h \bm{u}^{n-1}\|_{L^3(\Omega)}\| \bm{U}_h^n - \bm{U}^n\|)\|\Pi_h \bm{u}^n - \bm{u}^n\|_{L^6(\Omega)} 
%&C(h^{-1}\|\bm{e}_h^{n-1}\|+\|\nabla \cdot \Pi_h \bm{u}^{n-1}\|) \|\bm{U}_h^n - \bm{U}^n\|\|\Pi_h \bm{u}^n - \bm{u}^n\|_{W^{k,6}(\Omega)} 
\nonumber \\ &  + C \|\Pi_h \bm{u}^n - \bm{u}^n \|_{L^{\infty}(\Omega)}( h^{-1}\|\bm{U}_h^n - \bm{U}^n \| + \|\nabla_h(\bm{U} - \bm{U}_h^n)\|)\| \bm{e}_h^{n-1}\|\nonumber \\ &+ C \|\Pi_h \bm{u}^{n-1} \|_{L^{\infty}(\Omega)}(h^{-1/2}\|\bm{U}_h^n - \bm{U}^n \| +  h^{1/2}\|\nabla_h(\bm{U}^n - \bm{U}_h^n)\|) \nonumber\\  &\quad \quad  (h^{-1/2}\| \Pi_h \bm{u}^n - \bm{u}^n \| + h^{1/2}\| \nabla_h(\Pi_h \bm{u}^n - \bm{u}^n)\|)\nonumber\\ 
& \leq C h^{k}\|\bm{e}_h^{n-1}\|\| \bm{\chi}^n\||\bm{u}^n|_{W^{k,6}(\Omega)} + C h^{k + 2}\|\bm{u}^{n-1}\|_{W^{1,3}(\Omega)} |\bm{u}^n|_{W^{k,6}(\Omega)}\| \bm{\chi}^n \| \nonumber \\ & \quad + Ch\vertiii{\bm{u}^n} \| \bm{\chi}^n\|\|\bm{e}_h^{n-1}\| + C h^{k+2}\vertiii{\bm{u}^{n-1}} \vert\bm{u}^n\vert_{H^{k+1}(\Omega)}\|\bm{\chi}^n\|.
\label{eq:bd_second_term_A21}
\end{align}
With \eqref{eq:bd_first_term_A21},  \eqref{eq:bd_second_term_A21}, Young's inequality, and that 
$\bm{u} \in L^{\infty}(0,T; H^{k+1}(\Omega)^d)$,  we obtain a bound for $ A_{1,2}^n$: 
\begin{align}
|A^n_{1,2}|  \leq \epsilon \mu \|\bm{e}_h^n \|^2 + C\left(\frac{1}{\epsilon \mu} + 1\right) (\|\bm{e}_h^{n-1}\|^2 + h^2 )\|\err\|^2_{\DG} \nonumber   \\  + C\left(\frac{1}{\epsilon \mu}  + 1\right)h^2 \|\bm{e}_h^{n-1}\|^2  + C\left(\frac{h^2}{\epsilon \mu} +1  \right) 
 h^{2k+2}.  % + \frac{C}{\mu} h^{2}\|\bm{e}_h^{n-1}\|^2+ \frac{C}{\mu}(\| \bm{e}_h^{n-1}\|^2 + h^2) \|\err\|_{\DG}^2.
  \label{eq:bd_A_12}
\end{align}
It remains to handle $A^n_{1,3}$ and $A^n_{1,4}$. We will focus on $A^n_{1,3}$. We use the definition of $\partial E_{-}^{\bm{u}_h^{n-1}}$, see \eqref{eq:inflow_boundary}, and wirte $\bm{u}_h^{n-1} = \bm{e}_h^{n-1} + \Pi_h \bm{u}^{n-1}$. 
% We have
% \begin{align*}
% & A^n_{1,3} =  \sum_{E\in\mesh_h} \int_{\partial E_{-}^{\bm{u}_h^{n-1}}\setminus \partial \Omega} \{\bm{e}_h^{n-1}\}\cdot \bm{n}_E ((\bm{U}_h^n - \bm{U}^n)^{\mathrm{int} } - (\bm{U}_h^n - \bm{U}^n)^{\mathrm{ext}})\cdot ( \err+  ( \Pi_h \bm{u}^n- \bm{u}^n) )^{\mathrm{ext}}\\ & +   \sum_{E\in\mesh_h} \int_{\partial E_{-}^{\bm{u}_h^{n-1}}\setminus \partial \Omega}  \{\Pi_h \bm{u}^{n-1}\}\cdot \bm{n}_E ((\bm{U}_h^n - \bm{U}^n)^{\mathrm{int} } - (\bm{U}_h^n - \bm{U}^n)^{\mathrm{ext}})\cdot  ( \err +  ( \Pi_h \bm{u}^n- \bm{u}^n) )^{\mathrm{ext}}.
% \end{align*}
With trace estimates \eqref{eq:trace_ineq_continuous}-\eqref{eq:trace_ineq_discrete}, we obtain 
 \begin{align*}
     |A^n_{1,3}|& \leq C \|\bm{e}_h^{n-1}\|_{L^{3}(\Omega)}(h^{-1}\|\bm{U}_h^n - \bm{U}^n \| + \|\nabla_h (\bm{U}_h^n - \bm{U}^n) \|)\|\err\|_{L^{6}(\Omega)}  \\ & \quad + C \| \Pi_h \bm{u}^n - \bm{u}^n\|_{L^{\infty}(\Omega)}\| \bm{e}_h^{n-1}\| \|\bm{U}_h^n\|_{\DG}+ C\| \Pi_h \bm{u}^{n-1}\|_{L^{\infty}(\Omega)} \|\bm{U}_h^n\|_{\DG} \|\err \| \\ & \quad + C \|\Pi_h \bm{u}^{n-1} \|_{L^{\infty}(\Omega)}\| \bm{U}_h^n\|_{\DG}\left(\|\Pi_h \bm{u}^n - \bm{u}^n\| + h\|\nabla_h (\Pi_h \bm{u}^n - \bm{u}^n )\|\right).
 \end{align*}
A similar technique is used to bound $ A^n_{1,4}$ and  the above bound also holds for $A^n_{1,4}$. Using \eqref{eq:error_dg_aux}, \eqref{eq:inverse_estimate}, \eqref{eq:discreter_poincare}, \eqref{eq:bd_l2_ehn}, and the assumption $\bm{u} \in L^{\infty}(0,T;H^{k+1}(\Omega)^d)$, we obtain 
\begin{align}
 |A^n_{1,3}| +|A^n_{1,4}| &\leq C h^{- d/6} \|\bm{e}_h^{n-1}\|h\|\bm{\chi}^n\|\|\err\|_{\DG} + C \vertiii{\bm{u}^n}  (\| \bm{e}_h^{n-1} - \bm{e}_h^n\| + \|\bm{e}_h^n\|) \|\bm{U}_h^n\|_{\DG} \nonumber \\ &\quad+ C \vertiii{\bm{u}^{n-1}}\| \bm{U}_h^n\|_{\DG}(\|\bm{e}_h^n\|+ \tau |\phi_h^n|_{\DG} + h^{k+1}| \bm{u}^n |_{H^{k+1}(\Omega)}) \nonumber \\
& \leq \epsilon  \mu \|\bm{e}_h^{n}\|^2 + \epsilon \| \bm{e}_h^{n-1} - \bm{e}_h^n\|^2+ C\left(\frac{1}{\epsilon\mu} +1  \right)\| \bm{e}_h^{n-1}\|^2\| \err\|_{\DG}^2  \nonumber \\ & \quad + C\tau^2 \vert \phi_h^n \vert_{\DG}^2 + C h^{2k+2} +C \left(\frac{1}{\epsilon\mu}+1\right)\|\bm{U}_h^n\|_{\DG}^2 . \label{eq:bd_a14}
\end{align}
With bounds \eqref{eq:bd_A_111}, \eqref{eq:bd_A_12},  and \eqref{eq:bd_a14}, we obtain 
\begin{align}
& |A^n_{1}| \leq  3\epsilon \mu\|\bm{e}_h^n\|^2 +\epsilon\|\bm{e}_h^{n-1} - \bm{e}_h^n\|^2  + C\left(\frac{h^2}{\epsilon \mu} + 1 \right)(\tau^2 |\phi_h^n|^2_{\DG} + h^{2k+2})  \nonumber \\ & +   C \left( \frac{1}{\epsilon\mu} + 1\right)( h^{2}(\|\bm{e}_h^n\|^2  + \|\bm{e}_h^{n-1}\|^2)+(\|\bm{e}_h^{n-1}\|^2+h^2)\| \err \|_{\DG}^2 +\|\bm{U}_h^n\|^2_{\DG}).%((\|\bm{e}_h^{n-1}\|^2+h^2)\| \err \|_{\DG}^2 + \|\bm{U}_h^n\|^2_{\DG}). 
 \label{eq:bd_A_1}
\end{align}
To handle $ T_2$, we write 
\begin{multline}
 T^n_2 = -  a_\mathcal{C}(\bm{u}^n; \bm{u}^n - \bm{u}^{n-1}, \bm{u}^n , \bm{U}_h^n ) -  a_\mathcal{C}(\bm{u}^n; \bm{u}^{n-1} - \Pi_h \bm{u}^{n-1}, \bm{u}^n , \bm{U}_h^n) \\  +  a_\mathcal{C}(\bm{u}^n; \bm{e}_h^{n-1}, \bm{u}^n, \bm{U}_h^n)   = T_{2,1}^n+  T_{2,2}^n + T_{2,3}^n. 
\end{multline}
The first term simply reads:
\begin{equation}
 T_{2,1}^n = -  \sum_{E \in \mesh_h} \int_E ((\bm{u}^n - \bm{u}^{n-1})\cdot \nabla \bm{u}^n) \cdot \bm{U}_h^n. 
\end{equation} 
By H\"{o}lder's inequality, \eqref{eq:discreter_poincare}, and a Taylor's expansion, we have: 
\begin{equation}
     |T_{2,1}^n| \leq %C \| \bm{u}^n - \bm{u}^{n-1} \| \vert \bm{u}^n\vert_{W^{1,3}(\Omega)}\|\bm{U}_h^n\|_{L^6(\Omega)}  \\  
 C \tau|\bm{u}^n|^2_{W^{1,3}(\Omega)} \int_{t^{n-1}}^{t^n}\|\partial_t \bm{u}\|^2  + C \| \bm{U}_h^n \|^2_{\DG}. \label{eq:bd_T_21}
\end{equation}
%\begin{multline}
     %|T_{2,1}^n| \leq C \| \bm{u}^n - \bm{u}^{n-1} \| \vert \bm{u}^n\vert_{W^{1,3}(\Omega)}\|\bm{U}_h^n\|_{L^6(\Omega)}  \\  \leq  C \tau|\bm{u}^n|^2_{W^{1,3}(\Omega)} \int_{t^{n-1}}^{t^n}\|\partial_t \bm{u}\|^2  + C \| \bm{U}_h^n \|^2_{\DG}. \label{eq:bd_T_21}
%\end{multline}
The term $ T_{2,2}^n$ is bounded by \eqref{eq:third_estimate_nonlinear_term}  in Lemma \ref{lemma:bounds_nonlinear_terms_2}, and the term $T_{2,3}^n$ is bounded by \eqref{eq:first_estimate_nonlinear_form} in Lemma \ref{prep:bounds_nonlinear_terms}. With the assumption that $\bm{u} \in L^{\infty}(0,T;H^{k+1}(\Omega)^d)$, we have  
\begin{multline}
 |T_{2,2}^n| + | T_{2,3}^n|  \leq C  (h^{k+1}|\bm{u}^{n-1}|_{H^{k+1}(\Omega)} + \|\bm{e}_h^{n-1} \|) \vertiii{\bm{u}^n}\|\bm{U}_h^n\|_{\DG} 
 \\ \leq \epsilon \mu \| \bm{e}_h^n\|^2  +C h^{2k+2}+ \epsilon  \| \bm{e}_h^{n-1} - \bm{e}_h^n\|^2 + C\left(\frac{1}{\epsilon \mu}+1\right)\|\bm{U}_h^n\|^2_{\DG}. \label{eq:bd_T_22}
\end{multline}
% The term $T_{2,3}^n$ is bounded  \eqref{eq:first_estimate_nonlinear_form} in Lemma \ref{prep:bounds_nonlinear_terms}. We have, 
% \begin{align}
%  |T_{2,3}^n | & \leq C\|\bm{e}_h^{n-1}\| \vertiii{\bm{u}^n} \|\bm{U}_h^n\|_{\DG} \nonumber \\ 
%     & \leq \ell \mu \| \bm{e}_h^n\|^2 + \ell  \| \bm{e}_h^{n-1} - \bm{e}_h^n\|^2 + C\left(\frac{1}{\mu}+1\right)\vertiii{\bm{u}^n}^2 \|\bm{U}_h^n\|^2_{\DG}.  \label{eq:bd_T_23}
% \end{align}

        Combining the above bounds, \eqref{eq:bd_A_2}, \eqref{eq:bd_A_3},\eqref{eq:bd_A_1},\eqref{eq:bd_T_21}, and \eqref{eq:bd_T_22}, yield the following bound on $ |\tilde{R}_{\mathcal{C}}(\bm{U}_h^n)|$. 
        \begin{align}
         & |\tilde{R}_{\mathcal{C}}(\bm{U}_h^n)| \leq  7\epsilon \mu\|\bm{e}_h^n\|^2 +2\epsilon\|\bm{e}_h^{n-1} - \bm{e}_h^n\|^2  + C\left(\frac{h^2}{\epsilon \mu} + 1 \right)(\tau^2 |\phi_h^n|^2_{\DG} + h^{2k+2})  \nonumber \\ & +   C \left( \frac{1}{\epsilon\mu} + 1\right)( h^{2}(\|\bm{e}_h^n\|^2  + \|\bm{e}_h^{n-1}\|^2)+(\|\bm{e}_h^{n-1}\|^2+h^2)\| \err \|_{\DG}^2 )\nonumber \\ & + C\tau \int_{t^{n-1}}^{t^n} \|\partial_t \bm{u}\|^2 + C\left( \frac{1}{\epsilon\mu} + 1\right) \|\bm{U}_h^n\|_{\DG}^2.  %((\|\bm{e}_h^{n-1}\|^2+h^2)\| \err \|_{\DG}^2 + \|\bm{U}_h^n\|^2_{\DG})
          \label{eq:bound_R_C_improved}
        \end{align}
%\Rd BR: I removed the bound for $R_t$, not used here\Bk
%A standard argument and approximation properties yield 
%        \begin{equation}\label{eq:bound_Rt_improved}
%           |R_t(\bm{U}_h^n)|  \leq C  \tau^2 \int_{t^{n-1}}^{t^n} \| \partial_{tt} \bm{u}\|^2 + C h^{2k+2} \int_{t^{n-1}}^{t^n} \vert \partial_t  \bm{u}\vert_{H^{k+1}(\Omega)}^2  + C \tau \| \bm{U}_h^n \|^2_{\DG}.
%        \end{equation}
        We use the bounds above   in \eqref{eq:improved_estimate_semi_final}, use the coercivity property \eqref{eq:coercivity_a_ellip}, and choose $\epsilon = 1/24$. We sum the resulting equation,  from $n= 1$ to $n =m$, use the regularity assumptions, and obtain the following. 
        \begin{align}
            &\frac{1}{2} a_{\mathcal{D}}(\bm{U}_h^m, \bm{U}_h^m)  - \frac{1}{2}a_{\mathcal{D}}( \bm{U}_h^{0}, \bm{U}_h^{0}) + \frac14 \sum_{n= 1}^m  \|\bm{U}_h^n - \bm{U}_h^{n-1}\|_{\DG}^2   + \frac{\tau \mu}{2} \sum_{n=1}^m \| \bm{e}_h^n \|^2 \nonumber \\ & \leq C\left(1+ \frac{1}{\mu} + \mu\right) h^{2k+2} + C\tau^2  + C \tau^3  \sum_{n=1}^m \left(\frac{h^2}{\mu} +  1 + \mu \right)\vert \phi_h^n \vert_{\DG}^2 \nonumber \\&  +  \frac{\tau}{12} \sum_{n=1}^m \| \bm{e}_h^n - \bm{e}_h^{n-1} \|^2 + C \tau \left( \frac{1}{\mu} + 1 \right) \sum_{n=1}^m \|\bm{U}_h^n\|_{\DG}^2 \nonumber \\  & +C\left(\frac{1}{\mu}+1\right)\tau \sum_{n=1}^m ( h^{2}(\|\bm{e}_h^n\|^2  + \|\bm{e}_h^{n-1}\|^2)+(\|\bm{e}_h^{n-1}\|^2+h^2)\| \err \|_{\DG}^2 ). \label{eq:error_eq_semi_final_improved_estimate}
        \end{align}
     Note that by \eqref{eq:aux_1_dg}-\eqref{eq:aux_2_dg} and by definition of the $L^2$ projection, we have 
        \begin{equation} a_\mathcal{D}(\bm{U}_h^0, \bm{U}_h^0) = (\bm{\chi}^0, \bm{U}_h^0 ) = 0.  \label{eq:UH0_ZERO}
        \end{equation}
        %To simplify the writeup and to demonstrate the dependence of the constant $C$ on $\mu^{-1}$ as $\mu \rightarrow 0$, we assume without loss of generality that $\mu < 1$. 
        Hence, $\bm{U}_h^0 = \bm{0}$ since  $a_\mathcal{D}$ is coercive. 
        With Lemma \ref{lemma:first_err_estimate}, coercivity of $a_\mathcal{D}$ \eqref{eq:coercivity_a_ellip}, we have
        \begin{align*}
        &\|\bm{U}_h^m\|^2_{\DG}+ \sum_{n= 1}^m  \|\bm{U}_h^n - \bm{U}_h^{n-1}\|_{\DG}^2 + 2 \tau \mu  \sum_{n=1}^m \| \bm{e}_h^n \|^2  \leq 
C\left(1+ \frac{1}{\mu} +\mu\right) h^{2k+2} + C\tau^2 \\ & + C_{\mu} (\tau+h^{2k})(\tau+h^2) \sum_{i=-4}^2 \mu^i
+C \tau \left( \frac{1}{\mu} + 1 \right) \sum_{n=1}^m \|\bm{U}_h^n\|_{\DG}^2. 
        \end{align*}
  Under  the assumption that $C(1/\mu + 1)\tau < 1$, we use Gronwall's inequality and obtain 
        \begin{equation}
        \tau \mu \sum_{n=1}^m \|\bm{e}_h^n\|^2 \leq C_{\mu} (\tau^2 + \tau h^2 + h^{2k+2})\sum_{i=-4}^2 \mu^i . 
        \end{equation}
        Since  $\tau h^2 \leq  (\tau^2 + h^4)/2$, the result is obtained if $k=1$. If $k \geq 2$, we use the assumption $h^{2} \leq \tau$. 
        To obtain a bound on $\err$, we use \eqref{eq:sec_err_eq}, \eqref{eq:lift_prop_g} and \eqref{eq:first_error_estimate}. 
        \begin{align}
            \tau \mu \sum_{n=1}^m  \| \err \|^2 &\leq 2 \tau \mu \sum_{n=1}^m \| \err-\bm{e}_h^n \|^2 + 2\tau  \mu \sum_{n=1}^m \|\bm{e}_h^n \|^2
  \leq C\mu  \tau^3 \sum_{n=1}^m | \phi_h^n |_{\DG}^2 \nonumber\\ 
           &  + 2\tau \mu \sum_{n=1}^m \|\bm{e}_h^n \|^2 \leq  C_{\mu} (\tau^2 + \tau h^2 + h^{2k+2})\sum_{i=-4}^2 \mu^i. 
        \end{align}
     For $k\geq 2$, we use the assumption $h^2 \leq \tau$. The result follows by the triangle inequality. 
        \end{proof}    

     \section{Stability and error estimates for the discrete time derivative of the velocity}\label{sec:time_derr_velocity}
In this section, we establish estimates for the discrete time derivative of the velocity. For a given function $\bm{W} \in \bm{X}$, define the discrete time derivative as 
\begin{equation}
    \delta_\tau \bm{W}^{n+1} = \frac{\bm{W}^{n+1} - \bm{W}^{n}}{\tau}, \quad n \geq 0. 
\end{equation}
The same notation is used for the   discrete time derivative of a scalar function in $\mathcal{C}([0,T],M)$. 
\begin{lemma}\label{lemma:stability_time_derivative} Assume $\partial_t \bm{u} \in L^2(0,T;H^2(\Omega)^d)$ and $ \partial_{t} p \in L^2(0,T; H^1(\Omega ))$. Fix $0 <  \gamma \leq 1$ and assume that there exist positive constants $c_1, c_2$ such that $\tau$ satisfies
    \begin{equation}
        c_1 h^2 \leq \tau \leq c_2 h^{(1+\gamma)d/3}.  \label{eq:cfl_cond_detla} 
    \end{equation}
Under the same assumptions of Theorem \ref{theorem:improved_estimate_velocity},  
we have the following stability bound.  For $ 1\leq m \leq N_T-1$, 
\[ 
    \|\delta_\tau \bm{e}_h^{m+1}\|^2 +\sum_{n=1}^{m} 
     \|\delta_\tau \bm{e}_h^{n+1} - \delta_\tau \bm{e}_h^n \|^2
+ \tau^2 \sum_{n=1}^m \vert \delta_\tau \phi_h^{n+1} \vert_{\DG}^2  +  \mu \tau \sum_{n=1}^{m}  \|\delta_\tau \errn  \|_{\DG}^2 \leq C_{\gamma,\mu}.  
\]
The constant $C_{\gamma,\mu}$ is independent of $h$ and $\tau$ but depends nonlinearly on $\gamma$ and $\mu$: $$C_{\gamma,\mu} =  C_\mu \tilde{K}_\mu + (C_\mu \tilde{K}_\mu)^{\beta + 2}\mu^{-\beta}c_2^{\beta},$$
where  $\beta$ is an integer such that $ \beta \geq 1/\gamma$, $C_\mu$ depends on $e^{\frac{1}{\mu}}$ and $\tilde{K}_\mu =  \sum_{j=-3}^2 \mu^j$.
\end{lemma}
\begin{proof}
 Subtracting \eqref{eq:first_err_eq} at time $ t^n$ from \eqref{eq:first_err_eq} at time $t^{n+1}$, we obtain for all $\bm{\theta}_h \in \bm{X}_h$, 
\begin{align}
(\delta_\tau \bm{\tilde{e}}_h^{n+1}& - \delta_\tau \bm{e}_h^n , \bm{\theta}_h) + \tau \mu a_{\mathcal{D}}(\delta_\tau \bm{\tilde{e}}_h^{n+1}, \bm{\theta}_h )= \tau b(\bm{\theta}_h,\delta_\tau p_h^n - \delta_\tau p^{n+1})\nonumber  \\ &\quad  - \tau \mu a_{ \mathcal{D}}(\delta_\tau \Pi_h \bm{u}^{n+1} - \delta_\tau \bm{u}^{n+1}, \bm{\theta}_h)  + \mathcal{N}_1(\bm{\theta}_h)  + \mathcal{N}_2(\bm{\theta}_h) + \hat{R}_t(\bm{\theta}_h). \label{eq:first_err_eq_dteh} 
\end{align}
Here $\mathcal{N}_1(\bm{\theta}_h), \mathcal{N}_2(\bm{\theta}_h)$ and $\hat{R}_t(\bm{\theta}_h)$ are defined as follows. For $\bm{\theta}_h \in \bm{X}_h$ 
\begin{align*}
    \mathcal{N}_1(\bm{\theta}_h) &=a_\mathcal{C}(\bm{u}^{n+1}; \bm{u}^{n+1}, \bm{u}^{n+1},\bm{\theta}_h)  -  a_\mathcal{C}(\bm{u}_h^n;  \bm{u}_h^n, \bm{v}_h^{n+1},\bm{\theta}_h ), \\ 
    \mathcal{N}_2(\bm{\theta}_h) & = a_\mathcal{C}(\bm{u}_h^{n-1}; \bm{u}_h^{n-1}, \bm{v}_h^{n},\bm{\theta}_h ) -  a_\mathcal{C}(\bm{u}^{n}; \bm{u}^{n}, \bm{u}^{n},\bm{\theta}_h), \nonumber \\   
\hat{R}_t(\bm{\theta}_h) & =  ((\partial_t \bm{u})^{n+1} - (\partial_t \bm{u})^n, \bm{\theta}_h ) - \frac{1}{\tau}(\Pi_h \bm{u}^{n+1} -  2\Pi_h\bm{u}^n + \Pi_h \bm{u}^{n-1}, \bm{\theta}_h ).
 \end{align*}
 Choosing $\bm{\theta}_h = \delta_\tau \bm{\tilde{e}}_h^{n+1}$ in \eqref{eq:first_err_eq_dteh} and using the coercivity property of $a_\mathcal{D}$ \eqref{eq:coercivity_a_ellip} yields 
\begin{multline}
\frac{1}{2} ( \|\delta_\tau \bm{\tilde{e}}_h^{n+1}\|^2 -\| \delta_\tau \bm{e}_h^n \|^2 + \|\delta_\tau \bm{\tilde{e}}_h^{n+1} - \delta_\tau \bm{e}_h^n \|^2 ) + \frac12 \tau \mu  \|\delta_\tau \bm{\tilde{e}}_h^{n+1} \|_{\DG}^2 \\  \leq \tau b(\delta_\tau \bm{\tilde{e}}_h^{n+1}, \delta_\tau p_h^n - \delta_\tau p^{n+1})    -\tau \mu a_{\mathcal{D}}(\delta_\tau \Pi_h \bm{u}^{n+1} - \delta_\tau \bm{u}^{n+1}, \delta_\tau \errn ) \\  + \mathcal{N}_1(\delta_\tau \errn) + \mathcal{N}_2(\delta_\tau \errn) +  \hat{R}_t(\delta_\tau \errn). \label{eq:error_eq_dt}
\end{multline} 
From \eqref{eq:sec_err_eq}, we have for $n \geq 1$: 
\begin{equation}
    (\delta_\tau \bm{e}_h^{n+1} - \delta_\tau \bm{\tilde{e}}^{n+1}_h, \bm{\theta}_h) = \tau b(\bm{\theta}_h , \delta_\tau \phi_h^{n+1}) %- \tau (\nabla_h \delta_\tau \phi_h^{n+1} - \bm{G}_h ([\delta_\tau \phi_h^{n+1}]), \bm{\theta}_h) 
    , \quad \forall \bm{\theta}_h \in \bm{X}_h. 
    \label{eq:update_velocity_delta}
\end{equation}
In addition, from \eqref{eq:div_error}, and \eqref{eq:div_error_2}, we obtain $\forall q_h \in M_h$ and $n \geq 1$ 
\begin{align}
    b(\delta_\tau \bm{e}_h^{n+1}, q_h) &= b(\delta_\tau \errn, q_h) + \tau\aelip(\delta_\tau \phi_h^{n+1},q_h) -\tau \sum_{e \in \Gamma_h} \frac{\tilde{\sigma}}{h_e} \int_e [\delta_\tau \phi_h^{n+1}][q_h] \nonumber \\ &  \quad + \tau(\bm{G}_h([\delta_\tau \phi_h^{n+1}]), \bm{G}_h([q_h])) \nonumber \\ &
     = -\tau \sum_{e \in \Gamma_h} \frac{\tilde{\sigma}}{h_e} \int_e [\delta_\tau \phi_h^{n+1}][q_h] + \tau(\bm{G}_h([\delta_\tau \phi_h^{n+1}]), \bm{G}_h([q_h])). \label{eq:property_err_delta}
\end{align} 
From the above properties and \eqref{eq:update_velocity_delta}, observe that 
\begin{multline}
\|\delta_\tau \bm{e}_h^{n+1} - \delta_\tau \errn\|^2 = 
\tau b(\delta_\tau \bm{e}_h^{n+1} - \delta_\tau \errn,  \delta_\tau \phi_h^{n+1} ) =  \tau^2\aelip(\delta_\tau \phi_h^{n+1}, \delta_\tau \phi_h^{n+1}) \\  -\tau^2 \sum_{e\in\Gamma_h} \frac{\tilde{\sigma}}{h_e} \| [\delta_\tau \phi_h^{n+1}] \|^2_{L^2(e)} + \tau^2  \|\bm{G}_h([\delta_\tau \phi_h^{n+1}])\|^2. 
\end{multline}
Hence,  by letting $\bm{\theta}_h = \delta_\tau \bm{e}_h^{n+1}$ in \eqref{eq:update_velocity_delta} and using \eqref{eq:property_err_delta}, we obtain 
\begin{multline}\label{eq:error_eq_2_dt}
\frac{1}{2}\left( \| \delta_\tau \bm{e}_h^{n+1}\|^2 - \|\delta_\tau \errn\|^2 \right) + \frac{\tau^2}{2}\aelip(\delta_\tau \phi_h^{n+1}, \delta_\tau \phi_h^{n+1}) \\  + \frac{\tau^2}{2} \sum_{e\in\Gamma_h} \frac{\tilde{\sigma}}{h_e} \| [\delta_\tau \phi_h^{n+1}] \|^2_{L^2(e)} = \frac{\tau^2}{2} \|\bm{G}_h([\delta_\tau \phi_h^{n+1}])\|^2.
\end{multline} 
In addition, note that from \eqref{eq:update_velocity_delta} and \eqref{eq:def_b_lift_2}, we have   
\[ 
\delta_\tau \bm{e}_h^{n+1} - \delta_\tau \tilde{\bm{e}}_h^{n+1} = - \tau \nabla_h \delta_\tau
\phi_h^{n+1} + \tau \bm{G}_h[\delta_\tau \phi_h^{n+1}].
\]
This implies that
\[
\Vert \delta_\tau \tilde{\bm{e}}_h^{n+1} - \delta_\tau \bm{e}_h^{n}\Vert^2 =
\Vert  \tau \nabla_h \delta_\tau
\phi_h^{n+1} - \tau \bm{G}_h[\delta_\tau \phi_h^{n+1}]
+ \delta_\tau \bm{e}_h^{n+1} - \delta_\tau \bm{e}_h^n \Vert^2.
\]
Expanding the norm and using \eqref{eq:def_b_lift_2} with $\bm{\theta}_h = \delta_\tau \bm{e}_h^{n+1} - \delta_\tau \bm{e}_h^n$ and $q_h = \delta_\tau \phi_h^{n+1}$ yields
\begin{align}
\|& \delta_\tau \errn -  \delta_\tau \bm{e}_h^n \|^2 =  \|\delta_\tau \bm{e}_h^{n+1} - \delta_\tau \bm{e}_h^n \|^2 + \tau^2 \| \nabla_h \delta_\tau \phi_h^{n+1} \|^2 + \tau^2 \|\bm{G}_h ([\delta_\tau \phi_h^{n+1}]) \|^2 \nonumber \\ & -2\tau b(\delta_\tau \bm{e}_h^{n+1}  -  \delta_\tau \bm{e}_h^n, \delta_\tau \phi_h^{n+1})  -2\tau^2(\nabla_h \delta_\tau \phi_h^{n+1}, \bm{G}_h([\delta_\tau \phi_h^{n+1}])).
\label{eq:diff_tilde}
\end{align}
With \eqref{eq:property_err_delta}  and \eqref{eq:div_error_2}  ,  we find (see Lemma~\ref{lemma:appendix_delta_expression} in Appendix)
\begin{multline}
-2 b(\delta_\tau \bm{e}_h^{n+1}  -  \delta_\tau \bm{e}_h^n, \delta_\tau \phi_h^{n+1})  =\tau( \tilde{A}^{n}_1 - \tilde{A}^{n}_2) + \tau \sum_{e\in\Gamma_h } \frac{\tilde{\sigma}}{h_e} \|[\delta_\tau \phi_h^{n+1} - \delta_\tau \phi_h^n]\|^2_{ L^2(e)} \\  - \tau\|\bm{G}_h([\delta_\tau \phi_h^{n+1} - \delta_\tau \phi_h^n ]) \|^2 - \frac{2}{\tau} \delta_{n,1} b(\bm{e}_h^0, \delta_\tau \phi_h^2), \label{eq:bform_delta_expression} 
\end{multline}
where for $n \geq 1$
\begin{align*}
\tilde{A}^n_1 & = \sum_{e \in \Gamma_h} \frac{\tilde{\sigma}}{h_e}\left( \|[\delta_\tau \phi_h^{n+1}] \|^2_{L^2(e)} - \|[\delta_\tau \phi_h^n]\|^2_{L^2(e)}\right) , \\   \tilde{A}^n_2  &=\| \bm{G}_h([\delta_\tau \phi_h^{n+1}]) \|^2 - \| \bm{G}_h([\delta_\tau \phi_h^{n}]) \|^2. 
\end{align*}
%(\textit{for our reference, see derivation of \eqref{eq:bform_delta_expression} in the Appendix \ref{lemma:appendix_delta_expression}}).
%\Rd Rami: can you also add the proof in the appendix for the general case $n\geq 2$? \Bk \\
Using \eqref{eq:lift_prop_g} and assuming $\tilde{\sigma} \geq 4\tilde{M}_k^2$, we have 
\begin{align*}
 \sum_{e\in\Gamma_h } \frac{\tilde{\sigma}}{h_e} \|[\delta_\tau \phi_h^{n+1} - \delta_\tau \phi_h^n]\|^2_{L^2(e)} - \|\bm{G}_h&([\delta_\tau \phi_h^{n+1} - \delta_\tau \phi_h^n ]) \|^2 \\ & \geq \frac{1}{2} \sum_{e\in\Gamma_h } \frac{\tilde{\sigma}}{h_e} \|[\delta_\tau \phi_h^{n+1} - \delta_\tau \phi_h^n]\|_{L^2(e)}^2 , \\ 
 |(\nabla_h \delta_\tau \phi_h^{n+1}, \bm{G}_h([\delta_\tau \phi_h^{n+1}]))|  %& \leq \frac{1}{4} \| \nabla_h \delta_\tau \phi_h^{n+1}\|^2 + \|\bm{G}_h([\delta_\tau \phi_h^{n+1}])\|^2 
  &  \leq \frac{1}{4} \| \nabla_h \delta_\tau \phi_h^{n+1}\|^2 + \frac{1}{4} \sum_{e\in \Gamma_h} \frac{\tilde{\sigma}}{h_e} \| [\delta_\tau \phi_h^{n+1}] \|_{L^2(e)}^2 .
\end{align*}
With the above expressions and \eqref{eq:error_eq_2_dt},  \eqref{eq:error_eq_dt} becomes: 
\begin{align}
& \frac{1}{2} ( \|\delta_\tau \bm{e}_h^{n+1}\|^2 - \| \delta_\tau \bm{e}_h^n \|^2 + \|\delta_\tau \bm{e}_h^{n+1} - \delta_\tau \bm{e}_h^n \|^2 ) + \frac12  \tau \mu \|\delta_\tau \bm{\tilde{e}}_h^{n+1} \|_{\DG}^2  +\frac{\tau^2}{4} \vert \delta_\tau \phi_h^{n+1} \vert_{\DG}^2 \nonumber\\ & \quad + \frac{\tau^2}{2} \aelip(\delta_\tau \phi_h^{n+1}, \delta_\tau \phi_h^{n+1}) + \frac{\tau^2}{4} \sum_{e\in\Gamma_h } \frac{\tilde{\sigma}}{h} \|[\delta_\tau \phi_h^{n+1} - \delta_\tau \phi_h^n]\|_{L^2(e)}^2 + \frac{\tau^2}{2}(\tilde{A}^n_1 -\tilde{A}^n_2 ) \nonumber \\ &  \leq \tau b(\delta_\tau \bm{\tilde{e}}_h^{n+1}, \delta_\tau p_h^n - \delta_\tau p^{n+1})   -\tau \mu a_{\mathcal{D}}(\delta_\tau \Pi_h \bm{u}^{n+1} - \delta_\tau \bm{u}^{n+1},  \delta_\tau \bm{\tilde{e}}_h^{n+1}) \nonumber \\ &  \quad + \mathcal{N}_1(\delta_\tau \errn) + \mathcal{N}_2(\delta_\tau \errn) +  \hat{R}_t(\delta_\tau \bm{\tilde{e}}_h^{n+1}) +  \delta_{n,1} b(\bm{e}_h^0, \delta_\tau \phi_h^2). \label{eq:error_eq_3_dt}
\end{align}
We begin by handling the first two terms on the right hand side of \eqref{eq:error_eq_3_dt}.  We write: 
\begin{multline*}
     b(\delta_\tau \errn, \delta_\tau p_h^n - \delta_\tau p^{n+1}) =  b(\delta_\tau \errn, \delta_\tau p_h^n ) -   b(\delta_\tau \errn , \pi_h (\delta_\tau p^{n+1})) \\  +  b(\delta_\tau \errn, \pi_h( \delta_\tau p^{n+1}) - \delta_\tau p^{n+1}). 
\end{multline*}
Since $\delta_\tau p^{n+1}$ has a zero average and $\pi_h$ preserves cell averages,
$\pi_h (\delta_\tau p^{n+1})$ belongs to $M_{h0}$. Hence, by \eqref{eq:error_pressure_correction} and the stability of the $L^2$ projection \eqref{eq:bd_phN}, we obtain  
\begin{multline}
    | b(\delta_\tau \errn , \pi_h(\delta_\tau p^{n+1}))| =  \tau| \aelip(\delta_\tau \phi_h^{n+1}, \pi_h (\delta_\tau p^{n+1}))| \\ \leq C \tau \vert\delta_\tau \phi_h^{n+1}\vert_{\DG} \vert \delta_\tau p^{n+1}\vert_{H^1(\Omega)}.  \nonumber 
\end{multline}
Using Young's inequality and Taylor's theorem, we obtain
\begin{align*}
| b(\delta_\tau \errn , \pi_h(\delta_\tau p^{n+1}))|% & \leq \frac{\tau}{16} |\delta_\tau  \phi_h^{n+1} |_{\DG}^2 + C\tau^{-1}| p^{n+1} - p^n |_{H^1(\Omega)}^2 \nonumber \\ 
 &\leq \frac{\tau}{16} |\delta_\tau  \phi_h^{n+1} |_{\DG}^2 + C \int_{t^{n}}^{t^{n+1}} \vert \partial_t p  \vert_{H^1(\Omega)}^2.  
\end{align*}
By \eqref{eq:form_b}, the definition of $ \pi_h(\delta_\tau p^{n+1})$, a trace inequality, and \eqref{eq:l2_proj_approximation}, we have:
\begin{align*}
 &|b(\delta_\tau \errn, \pi_h(\delta_\tau p^{n+1}) - \delta_\tau p^{n+1})| \\ & =  \left\vert\sum_{e\in\Gamma_h \cup \partial \Omega}  \int_e  \{  \pi_h(\delta_\tau p^{n+1})- \delta_\tau p^{n+1}\} [ \delta_\tau \errn] \cdot \bm{n}_e \right\vert \\ 
    & \leq C h\vert \delta_{\tau} p^{n+1}\vert_{H^1(\Omega)} \| \delta_\tau \errn \|_{\DG} \leq \epsilon  \mu \| \delta_\tau \errn\|_{\DG}^2 + \frac{C}{\epsilon
     \mu} h^2 \tau^{-1} \int_{t^{n}}^{t^{n+1}} \vert \partial_t p  \vert^2_{H^1(\Omega)}.
\end{align*}
To handle $b(\delta_\tau \errn, \delta_\tau p_h^n)$, we introduce the auxiliary functions: 
\begin{align*}
    \hat{S}_h^n = \delta \mu \sum_{i=1}^n (\nabla_h \cdot \delta_\tau \bm{\tilde{e}}_h^{i} - R_h ([\delta_\tau \bm{\tilde{e}}_h^i]) ), 
    \quad \hat{\xi}_h^{n}  =  \delta_\tau p_h^n + \hat{S}_h^n,  \quad  n \geq 1. 
\end{align*}
 From \eqref{eq:error_update_pressure}  and the above definitions, it is implied 
 %we have we derive for all $q_h \in M_h$, 
% \begin{align*}
%     (\delta_\tau p_h^{n+1} - \delta_\tau p_h^n, q_h) = (\delta_\tau \phi_h^{n+1}, q_h) - \delta \mu( \nabla_h \cdot \delta_\tau \bm{\tilde{e}}_h^{n+1} - R_h ([\delta_\tau \bm{\tilde{e}}_h^{n+1}]), q_h).
% \end{align*}
that $\hat{\xi}_h^{n+1} - \hat{\xi}_h^{n}
 = \delta_\tau \phi_h^{n+1}$. With this expression, we write: 
 \begin{equation*}
      b(\delta_\tau \errn, \delta_\tau p_h^n) =  b(\delta_\tau \errn, \hat{\xi}_h^n) - b(\delta_\tau \errn,\hat{S}_h^n).
 \end{equation*}
 Observe that $\hat{\xi}_h^n \in M_{h0}$ since $\hat{S}_h^{n} \in M_{h0}$ for $n \geq 1$. Hence, we use \eqref{eq:error_pressure_correction} to deduce that 
 \begin{align*}
      b(\delta_\tau \errn, &\hat{\xi}_h^n) = - \tau\aelip(\delta_\tau \phi_h^{n+1}, \hat{\xi}_h^n) = - \tau \aelip(\hat{\xi}_h^{n+1}- \hat{\xi}_h^{n}, \hat{\xi}_h^n )\\ 
     & = -\frac{\tau}{2}\left( \aelip(\hat{\xi}_h^{n+1} , \hat{\xi}_h^{n+1}) - \aelip(\hat{\xi}_h^{n} , \hat{\xi}_h^{n}) - \aelip(\delta_\tau \phi_h^{n+1}, \delta_\tau \phi_h^{n+1}) \right). 
 \end{align*}
 With \eqref{eq:def_b_lift}, we have 
 \begin{align*}
     b(\delta_\tau \errn, \hat{S}_h^n) & =  \frac{1}{\delta \mu} (\hat{S}_h^{n+1} - \hat{S}_h^n , \hat{S}_h^n )= \frac{1}{2\delta \mu} \left(\| \hat{S}_h^{n+1} \|^2 - \| \hat{S}_h^n \|^2 - \|\hat{S}_h^{n+1} - \hat{S}_h^n \|^2 \right). 
 \end{align*}
 With the assumption that $\delta \leq 1/(4d)$ and $\sigma \geq  M_{k-1}^2/d$, we use \eqref{eq:lift_prop_r} and the bound $\|\nabla_h \cdot \bm{\theta}\| \leq d^{1/2} \|\nabla_h \bm{\theta} \|$ for $\bm
 {\theta} \in \bm{X}$. We obtain 
 \begin{equation}
    \frac{1}{2\delta \mu} \| \hat{S}_h^{n+1} - \hat{S}_h^n \|^2 \leq \frac{ \mu}{4} \|\nabla_h \delta_\tau \errn \|^2 + \frac{\mu}{4} \sum_{e\in\Gamma_h \cup \partial \Omega} \sigma h_e^{-1} \| [\delta_\tau \errn] \|^2_{L^2(e)}. 
 \end{equation}
 Following a similar technique as the one used in \cite{girault2005splitting} and in \cite{inspaper1}, we also have %{\blue (We cannot directly apply the continuity of $a_\mathcal{D}$, since $\delta_\tau \bm{u}^{n+1}$ may not be polynomial. Question: is it better to mention more details as INS1 paper page 27--28, see the bound of $Q_1$ to $Q_4$?)}
 \begin{align*}
&|a_\mathcal{D}(\delta_\tau \Pi_h \bm{u}^{n+1} - \delta_\tau \bm{u}^{n+1} , \delta_\tau \errn)|  =  |a_\mathcal{D}(\Pi_h \delta_\tau  \bm{u}^{n+1} - \delta_\tau \bm{u}^{n+1} , \delta_\tau \errn)|  \\ 
&\leq C  h \vert \delta_\tau \bm{u}^{n+1}\vert_{H^2(\Omega)} \|\delta_\tau \errn \|_{\DG} \leq \epsilon \|\delta_\tau \errn \|_{\DG}^2  + \frac{C}{\epsilon}  h^2 \tau^{-1} \int_{t^{n}}^{t^{n+1}} |\partial_t \bm{u}|^2_{H^2(\Omega)}.
 \end{align*}
 With the above expressions and bounds,  \eqref{eq:error_eq_3_dt} reads: 
\begin{align} \label{eq:error_eq_4_dt}
& \frac{1}{2} ( \|\delta_\tau \bm{e}_h^{n+1}\|^2 - \| \delta_\tau \bm{e}_h^n \|^2 + \|\delta_\tau \bm{e}_h^{n+1} - \delta_\tau \bm{e}_h^n \|^2 ) + \frac{ \tau \mu}{4} \|\delta_\tau \bm{\tilde{e}}_h^{n+1} \|_{\DG}^2 \nonumber\\ &   \quad + \frac{\tau^2}{2} \left( \aelip(\hat{\xi}_h^{n+1}, \hat{\xi}_h^{n+1}) - \aelip(\hat{\xi}_h^{n}, \hat{\xi}_h^{n})\right)   +\frac{3\tau^2}{16} \vert \delta_\tau \phi_h^{n+1} \vert_{\DG}^2 \nonumber\\ &  \quad + \frac{\tau^2}{4} \sum_{e\in\Gamma_h } \frac{\tilde{\sigma}}{h} \|[\delta_\tau \phi_h^{n+1} - \delta_\tau \phi_h^n]\|_{L^2(e)}^2 + \frac{\tau^2}{2}(\tilde{A}^n_1 -\tilde{A}^n_2 ) + \frac{\tau}{2\delta \mu}\left(\|\hat{S}_h^{n+1}\|^2 - \|\hat{S}_h^n\|^2\right) \nonumber \\ &  \leq 2\epsilon \tau \mu \| \delta_\tau \errn \|_{\DG}^2 + C \frac{\mu }{\epsilon} h^2 \int_{t^n}^{t^{n+1}}| \partial_t \bm{u} |^2_{H^2(\Omega)} + C\left(\tau + \frac{h^2}{\epsilon \mu}\right) \int_{t^n}^{t^{n+1}} |\partial_t p|^2_{H^1(\Omega)}\nonumber \\ &  \quad + \mathcal{N}_1(\delta_\tau \errn) + \mathcal{N}_2(\delta_\tau \errn) +  \hat{R}_t(\delta_\tau \bm{\tilde{e}}_h^{n+1}) + \delta_{n,1} b(\bm{e}_h^0, \delta_\tau \phi_h^2 ). 
\end{align}
We proceed by noting the following splitting on the nonlinear terms. 
%$\mathcal{N}_1(\bm{\theta}_h)$ and $\mathcal{N}_2(\bm{\theta}_h)$. 
\begin{align}
    \mathcal{N}_1(\bm{\theta}_h) & = a_{\mathcal{C}}(\bm{u}^{n+1}; \bm{u}^{n+1} - \bm{u}^n , \bm{u}^{n+1}, \bm{\theta}_h)  + a_\mathcal{C}(\bm{u}_h^{n}; \bm{u}^n - \bm{u}_h^{n}, \bm{u}^{n+1}, \bm{\theta}_h) \nonumber \\ & \quad \quad + a_\mathcal{C}(\bm{u}_h^n; \bm{u}_h^n, \bm{u}^{n+1} -  \bm
    {v}_h^{n+1}, \bm{\theta}_h)
     = \xi^n_1(\bm{\theta}_h) + \xi^n_2(\bm{\theta}_h) + \xi^n_3(\bm{\theta}_h),    \label{eq:expanding_N1}\\
    \mathcal{N}_2(\bm{\theta}_h) & = -  a_{\mathcal{C}}(\bm{u}^{n}; \bm{u}^{n} - \bm{u}^{n-1} , \bm{u}^{n}, \bm{\theta}_h)   - a_\mathcal{C}(\bm{u}_h^{n-1}; \bm{u}^{n-1} - \bm{u}_h^{n-1}, \bm{u}^{n}, \bm{\theta}_h)  \nonumber \\ & \quad \quad - a_\mathcal{C}(\bm{u}_h^{n-1}; \bm{u}_h^{n-1}, \bm{u}^{n} -  \bm
    {v}_h^{n}, \bm{\theta}_h )
     = \vartheta^n_1(\bm{\theta}_h) + \vartheta^n_2(\bm{\theta}_h) + \vartheta^n_3(\bm{\theta}_h).  \label{eq:expanding_N2}   
\end{align}
To further simplify the writeup, we write: 
\[
\mathcal{N}_1(\delta_\tau \errn) + \mathcal{N}_2(\delta_\tau \errn) 
= \sum_{i = 1}^3 (\xi_i^n(\delta_\tau \errn) + \vartheta_i^n(\delta_\tau \errn)) = \sum_{i=1}^3 Q_i^n. 
\]
The term $Q^n_1$ can be handled as follows. %Since $\bm{u}(t)$ is divergence free, has zero jumps and vanishes on the boundary, this term reads: 
\begin{align}
Q^n_1 = \sum_{E \in \mesh_h}  \int_{E} ((\bm{u}^{n+1} - \bm{u}^n )\cdot \nabla \bm{u}^{n+1}  - (\bm{u}^{n} - \bm{u}^{n-1} ) )\cdot \nabla \bm{u}^{n}) \cdot \delta_{\tau} \errn . \nonumber
\end{align} 
With H\"{o}lder's inequality, \eqref{eq:discreter_poincare}, and the assumption that $\bm{u} \in L^{\infty}(0,T;H^{k+1}(\Omega)^{d}$: 
\begin{align*}
|Q^n_1| \leq & C (\|\bm{u}^{n+1} - \bm{u}^n \| |\bm{u}^{n+1}  |_{W^{1,3}(\Omega)} 
 +  \|\bm{u}^{n} - \bm{u}^{n-1}\|   |\bm{u}^{n}  |_{W^{1,3}(\Omega)}) \|\delta_\tau \errn \|_{\DG}  \\
 \leq  &  \epsilon \tau \mu \|\delta_\tau \errn \|_{\DG}^2 + \frac{C}{\epsilon \mu} \int_{t^{n-1}}^{t^{n+1}}\| \partial_t \bm{u}\|^2. 
\end{align*}
Since the upwind terms in $Q^n_2$ vanish, we write 
\begin{align}
Q^n_2 %& = \mathcal{C}(\bm{u}^n - \bm{u}_h^n,\bm{u}^{n+1}, \delta_\tau \errn) - \mathcal{C}(\bm{u}^{n-1} - \bm{u}_h^{n-1}, \bm{u}^n, \delta_\tau \errn) , \nonumber\\ 
& = \tau \mathcal{C}(\bm{u}^n - \Pi_h \bm{u}^n, \delta_\tau \bm{u}^{n+1}, \delta_\tau \errn ) - \tau \mathcal{C}(\bm{e}_h^n, \delta_\tau \bm{u}^{n+1}, \delta_\tau \errn ) \nonumber \\
& \quad + \tau \mathcal{C}(\delta_{\tau}(\bm{u}^n - \Pi_h \bm{u}^n),  \bm{u}^{n}, \delta_\tau \errn ) - \tau \mathcal{C}(\delta_\tau \bm{e}_h^n,  \bm{u}^{n}, \delta_\tau \errn). 
\end{align}
The first and third terms are bounded by \eqref{eq:third_estimate_nonlinear_term}. The second  and fourth terms are  bounded by \eqref{eq:first_estimate_nonlinear_form}.  We obtain:  
\begin{multline*}
 |Q^n_2  |   \leq  C\tau( h^{k+1} |\bm{u}^n|_{H^{k+1}(\Omega)} + \|\bm{e}
_h^n \|)\vertiii{\delta_\tau \bm{u}^{n+1}}\|\delta_\tau \errn \|_{\DG} \nonumber \\ 
 %&+ C\tau \left( h^2 | \delta_\tau \bm{u}^n|_{H^2(\Omega)} + \| \delta_\tau \bm{e}_h^n\|  \right)\vertiii{\bm{u}^n}\|\delta_\tau \errn \|_{\DG} \nonumber
 + C\tau \| \delta_\tau \bm{e}_h^n\|  \vertiii{\bm{u}^n}\|\delta_\tau \errn \|_{\DG} 
 +  C h^2 \sqrt{\tau} \left(\int_{t^{n-1}}^{t^n} \vert \partial_t \bm{u}\vert_{H^2(\Omega)} \right)^{1/2}   \vertiii{\bm{u}^n}\|\delta_\tau \errn \|_{\DG}. \nonumber
\end{multline*}
With the assumption that $\partial_t \bm{u} \in L^2(0,T; H^2(\Omega)^d)$, a Sobolev embedding result and a Taylor expansion, we have 
\begin{align}
\vertiii{\delta_\tau \bm{u}^{n}}^2 \leq C \tau^{-2} \Vert \bm{u}^{n} - \bm{u}^{n-1} \Vert_{H^2(\Omega)}^2 \leq C \tau^{-1} \int_{t^{n-1}}^{t^{n}} \Vert \partial_t \bm{u} \Vert_{H^2(\Omega)}^2, \quad n \geq 1. \label{eq:controlling_triplenorm_delta}
\end{align}
With the assumption that  $\bm{u} \in L^\infty(0,T;H^{k+1}(\Omega)^d)$, \eqref{eq:controlling_triplenorm_delta}, \eqref{eq:cfl_cond_detla}, and Young's inequality, we obtain 
\begin{align}
|Q^n_2|\leq & \epsilon \tau \mu \|\delta_\tau \errn\|^2_{\DG}+\frac{C}{\epsilon \mu} (\tau^2 + h^{2k+2} +\|\bm{e}_h^n\|^2 ) \int_{t^{n-1}}^{t^{n+1}}\|\partial_t \bm{u}\|_{H^2(\Omega)}^2  
\nonumber \\ &+ \frac{C}{\epsilon \mu} \tau^{-1} \|\bm{e}_h^n - \bm{e}_h^{n-1}\|^2.  \label{eq:bd_Q2_W_2}
\end{align}
% We further split $Q_2$ into two terms. We use \eqref{eq:third_estimate_nonlinear_term} to bound the first term and \eqref{eq:first_estimate_nonlinear_form} to bound the second term. We have, 
% \begin{align*}
% |Q_2|&=|  a_\mathcal{C}(\bm{u}_h^{n}; \bm{u}^n - \Pi_h \bm{u}^{n}, \bm{u}^{n+1}, \delta_\tau \errn) - a_\mathcal{C}(\bm{u}_h^{n}; \bm{e}_h^n, \bm{u}^{n+1}, \delta_\tau \errn) |,  \\ 
% & \leq C h^{k_1+1} |\bm{u}^n|_{H^{k+1}(\Omega)}\vertiii{\bm{u}^{n+1}} \|\delta_\tau \errn \|_{\DG} + C (\| \bm{e}_h^n\| + \tau |\phi_h^n|_{\DG})\vertiii{\bm{u}^{n+1}}\|\delta_\tau \errn  \|_{\DG} . 
% \end{align*}
% With the assumptions that $c_1 h^2 \leq \tau$ and $\bm{u}\in L^{\infty}(0,T;H^{k+1}(\Omega))$  and with applying Young's inequality, we obtain 
% \begin{equation}
% | Q_2| \leq \frac{C}{\mu}\tau + \ell \tau \mu \| \delta_\tau \errn \|_{\DG}^2 + \frac{C}{\mu} \tau^{-1} \|\bm{e}_h^n \|^2 + \frac{C}{\mu}\tau |\phi_h^n|^2_{\DG}.
% \end{equation}
% Similar arguments yield a bound on $W_2$. We have 
% \begin{equation}
%     |W_2| \leq \frac{C}{\mu}\tau + \ell \tau \mu \| \delta_\tau \errn\|_{\DG}^2 + \frac{C}{\mu}\tau^{-1} \|\bm{e}_h^{n-1} \|^2 + \frac{C}{\mu}\tau \vert \phi_h^{n-1} \vert_{\DG}^2. 
% \end{equation}
We write $Q^n_3$ as follows. 
\begin{align}
& Q^n_3   = a_\mathcal{C}(\bm{u}_h^n; \bm{u}_h^n, \bm{u}^{n+1} -  \Pi_h\bm{u}^{n+1}, \delta_\tau \errn) -a_\mathcal{C}(\bm{u}_h^n; \bm{u}_h^n, \errn , \delta_\tau \errn) \nonumber \\ 
&  -a_\mathcal{C}(\bm{u}_h^{n-1}; \bm{u}_h^{n-1}, \bm{u}^{n} -  \Pi_h \bm
{u}^{n}, \delta_\tau \errn) + a_\mathcal{C}(\bm{u}_h^{n-1}; \bm{u}_h^{n-1}, \err , \delta_\tau \errn) %\nonumber \\ & 
= \sum_{i=1}^4 \eta_i^n. 
\end{align}
 To handle $\eta^n_i$, we use the forms $\mathcal{U}$ and $\mathcal{C}$ given in \eqref{eq:split_ac1}-\eqref{eq:split_ac2}.  
Using \eqref{eq:splitting_technique}, we have: 
\begin{multline}
\eta^n_1 + \eta^n_3  = \tau a_{\mathcal{C}}(\bm{u}_h^{n-1}; \bm{u}_h^{n-1}, \delta_\tau \bm{u}^{n+1} - \delta_\tau \Pi_h \bm{u}^{n+1}, \delta_\tau \errn )    + U^n_1 \\ + \tau \mathcal{C}(\delta_\tau \bm{u}_h^{n}, \bm{u}^{n+1}- \Pi_h \bm{u}^{n+1}, \delta_\tau \errn )    - \tau \mathcal{U}(\bm{u}_h^{n-1}; \delta_\tau \bm{u}_h^n, \bm{u}^{n+1}- \Pi_h \bm{u}^{n+1}, \delta_\tau \errn )  , \nonumber
\end{multline}
where 
\begin{align*} U^n_1 &= 
 \mathcal{U}(\bm{u}_h^{n-1};\bm{u}_h^{n}, \bm{u}^{n+1}-\Pi_h \bm{u}^{n+1}, \delta_\tau \errn)  
-\mathcal{U}(\bm{u}_h^{n};\bm{u}_h^{n}, \bm{u}^{n+1}-\Pi_h \bm{u}^{n+1}, \delta_\tau \errn). 
\end{align*}
% The first term in $\eta_1 + \eta_3$ is: 
% $$\tau a_{\mathcal{C}}(\bm{u}_h^{n-1}; \bm{u}_h^{n-1}, \delta_\tau \bm{u}^{n+1} - \delta_\tau \Pi_h \bm{u}^{n+1}, \delta_\tau \errn )$$
% By eq (18), we have: 
% \begin{align*}
% \tau a_{\mathcal{C}}(\bm{u}_h^{n-1}; \bm{u}_h^{n-1}, \delta_\tau \bm{u}^{n+1} - \delta_\tau \Pi_h \bm{u}^{n+1}, \delta_\tau \errn ) = \tau \mathcal{C} ( \bm{u}_h^{n-1}, \delta_\tau \bm{u}^{n+1} - \delta_\tau \Pi_h \bm{u}^{n+1}, \delta_\tau \errn )\\ 
% - \tau \mathcal{U} (\bm{u}_h^{n-1}; \bm{u}_h^{n-1}, \delta_\tau \bm{u}^{n+1} - \delta_\tau \Pi_h \bm{u}^{n+1}, \delta_\tau \errn )
% \end{align*}
% We apply (63) to bound the first term and (75) to bound the second term: 
% \begin{align*}
%     \tau |a_{\mathcal{C}}(\bm{u}_h^{n-1}; \bm{u}_h^{n-1}, \delta_\tau \bm{u}^{n+1} - \delta_\tau \Pi_h \bm{u}^{n+1}, \delta_\tau \errn )| \\ 
%     \leq C\tau (1+ h^{1-d/4}) \|\bm{u}_h^{n-1}\|| \delta_\tau \bm{u}^{n+1}|_{H^2(\Omega)} \| \delta_\tau \errn \|_{\DG}\\  \leq  C\tau |\bm{u}_h^{n-1}\|| \delta_\tau \bm{u}^{n+1}|_{H^2(\Omega)} \| \delta_\tau \errn \|_{\DG}
%     \end{align*}
The term $U^n_1$ is bounded by Lemma \ref{lemma:bound_I1} in the Appendix, 
by the inverse inequality \eqref{eq:inverse_estimate} and by the approximation property \eqref{eq:approximation_prop_2}.
\begin{align*}
|U_1^n| &\leq C \tau h^{-d/6} \Vert \delta_\tau \bm{u}_h^n\Vert \, \Vert \bm{u}^{n+1}-\Pi_h \bm{u}^{n+1}\Vert_{\mathrm{DG}}
\, \Vert \delta_\tau \tilde{\bm e}_h^{n+1} \Vert_{\mathrm{DG}}\\
&\leq C \tau \Vert \delta_\tau \bm{u}_h^n\Vert \, \vert \bm{u}^{n+1}\vert_{H^2(\Omega)}
\, \Vert \delta_\tau \tilde{\bm e}_h^{n+1} \Vert_{\mathrm{DG}}.
\end{align*}
 We use \eqref{eq:split_ac2}, \eqref{eq:second_estimate_nonlinear_term_C1} and \eqref{eq:second_estimate_nonlinear_term} to bound the remaining terms in $\eta^n_1 + \eta^n_3$.  
%{\blue(just a comment: if directly apply \eqref{eq:second_estimate_nonlinear_term_C1} and \eqref{eq:second_estimate_nonlinear_term} we will obtain $|\delta_\tau \bm{u}^{n+1}|_{H^{k+1}(\Omega)}$ and $|\bm{u}^{n+1}|_{H^{k+1}(\Omega)}$. But from the proofs of these lemma, we know changing $H^{k+1}$ to $H^2$ these inequalities also hold.)} 
%\Rd BR: which bound is used to bound the first term in $\eta_1+\eta_3$? \Bk
%\begin{align}
%|\eta_1^n & +\eta^n_3|  \leq C\tau\|\bm{u}_h^{n-1}\|| \delta_\tau \bm{u}^{n+1}|_{H^2(\Omega)} \| \delta_\tau \errn \|_{\DG} + C \tau \| \delta_\tau  \bm{u}_h^{n}\|  | \bm{u}^{n+1}|_{H^2(\Omega)} \| \delta_\tau \errn\|_{\DG}\nonumber \\ 
%& \quad + C \tau^2(|\phi_h^{n-1}|_{\DG}\vertiii{\delta_\tau \bm{u}^{n+1}} + (\sum_{e\in \Gamma_h} \tilde{\sigma}h_e^{-1} \| [\delta_\tau \phi_h^n] \|_{L^2(e)}^2)^{1/2}\vertiii{\bm{u}^{n+1}}) \| \delta_\tau \errn\|_{\DG} \nonumber \\ 
%&\quad + C\tau \| \delta_\tau \bm{u}_h^n\|_{L^3(\Omega)} \left( \sum_{e\in\Gamma_h \cup \partial \Omega} \sigma h_e^{-1} \|[\Pi_h \bm{u}^{n+1} - \bm{u}^{n+1}] \|_{L^2(e)}^2\right)^{1/2} \| \delta_\tau \errn \|_{\DG}. \nonumber  
%&  \quad + \frac{\delta_{n,1}}{2}\tau |b(\bm{u}_h^0, (\delta_\tau \bm{u}^2 - \delta_\tau \Pi_h \bm{u}^2) \cdot \delta_\tau \tilde{\bm{e}}_h^2)| + \frac{\delta_{n,1}}{2} \tau  |b(\delta_\tau \bm{u}_h^1,(\bm{u}^2 - \Pi_h \bm{u}^2) \cdot \delta_\tau \tilde{\bm{e}}_h^2)|. \nonumber 
%\end{align}
\begin{align}
|\eta^n_1 +\eta^n_3| & \leq C\tau\left( \|\bm{u}_h^{n-1}\| |\delta_\tau \bm{u}^{n+1}|_{H^2(\Omega)} +  \| \delta_\tau \bm{u}_h^n\||\delta_\tau \bm{u}^{n+1}|_{H^2(\Omega)} \right)\| \delta_\tau \errn \|_{\DG}. \nonumber % \\ & \quad + C\tau ( 1 + h^{-d/6 + 1}) \| \delta_\tau \bm{u}_h^n\|\|\delta_\tau \errn\|_{\DG}. \nonumber %\\ &
 %+ C \tau^2(|\phi_h^{n-1}|_{\DG} |\delta_\tau \bm{u}^{n+1}|_{H^2(\Omega)} + (\sum_{e\in \Gamma_h} \tilde{\sigma}h_e^{-1} \| [\delta_\tau \phi_h^n] \|_{L^2(e)}^2)^{1/2}) \| \delta_\tau \errn\|_{\DG} \nonumber \\  &+ \frac{\delta_{n,1}}{2} \tau |b(\bm{u}_h^0, (\delta_\tau \bm{u}^2 - \delta_\tau \Pi_h \bm{u}^2) \cdot \delta_\tau \tilde{\bm{e}}_h^2)| + \frac{\delta_{n,1}}{2} \tau |b(\delta_\tau \bm{u}_h^1,(\bm{u}^2 - \Pi_h \bm{u}^2) \cdot \delta_\tau \tilde{\bm{e}}_h^2)|. \nonumber 
\end{align}
Further, with the triangle inequality, \eqref{eq:approximation_prop_1}, and a Taylor expansion, we have: 
\begin{align}
 \|\delta_{\tau} \bm{u}_h^{n}\|^2 & \leq 2 \|\delta_{\tau} \bm{e}_h^n\|^2 + 4 \|\delta_\tau \Pi_h \bm{u}^n - \delta_\tau \bm{u}^n\|^2 + 4  \| \delta_\tau \bm{u}^n  \|^2 \nonumber \\ 
& \leq 2 \tau^{-2} \|\bm{e}_h^n - \bm{e}_h^{n-1} \|^2 +C \tau^{-1} h^2 \int_{t^{n-1}}^{t^{n}}|\partial_t \bm{u}|_{H^1(\Omega)}^2 + C \tau^{-1} \int_{t^{n-1}}^{t^n} \| \partial_t \bm{u}\|^2 . \label{eq:delta_tau_uh}
\end{align}
 Hence, \eqref{eq:controlling_triplenorm_delta}, \eqref{eq:delta_tau_uh}, Lemma \ref{lemma:first_err_estimate}, Young's inequality, and the assumption that $\bm{u} \in L^{\infty}(0,T; H^{k+1}(\Omega)^d)$ yield the following bound.
\begin{align}
|\eta_1^n + \eta_3^n|  \leq  \epsilon \tau\mu   \|\delta_\tau \errn \|_{\DG}^2 + \frac{C}{\epsilon \mu \tau} \|\bm{e}_h^n - \bm{e}_h^{n-1} \|^2  \nonumber \\  +  \frac{C}{\epsilon \mu}\left(1+\frac{1}{\mu}\right)\int_{t^{n-1}}^{t^{n+1}} \| \partial_t \bm{u}\|_{H^2(\Omega)}^2.\nonumber 
\end{align}
Similarly, we use  \eqref{eq:splitting_technique} and write: 
\begin{equation}
\eta^n_2+ \eta^n_4  = -\tau a_\mathcal{C}(\bm{u}_h^{n-1}; \bm{u}_h^{n-1}, \delta_\tau \errn, \delta_\tau \errn )  
+  U^n_2 ,
\label{eq:eta2eta4}
\end{equation}
with
\begin{align*}
 U^n_2 = &\mathcal{U}(\bm{u}_h^n; \bm{u}^n_h, \errn , \delta_\tau \errn )  -  \mathcal{U}(\bm{u}_h^{n-1} ; \bm{u}^n_h, \errn , \delta_\tau \errn ) \\
&-\mathcal{C}(\bm{u}^n_h - \bm{u}_h^{n-1}, \errn , \delta_\tau \errn ) +   \mathcal{U}(\bm{u}_h^{n-1} ; \bm{u}^n_h - \bm{u}_h^{n-1}, \errn , \delta_\tau \errn ).
\end{align*}
By the positivity property of $a_\mathcal{C}$ \eqref{eq:cpositivity}, the first term in the right-hand side of 
\eqref{eq:eta2eta4} is non-positive.  The first two terms of $U_2^n$ are bounded by Lemma \ref{lemma:bound_I1} 
and the last two terms by Lemma \ref{lemma:bd_eta_1_eta_2_semi}. 
\[
| U_2^n | \leq C \| \bm{u}_h^n - \bm{u}_h^{n-1}\|_{L^3(\Omega)} \| \errn \|_{\DG} \|\delta_\tau \errn \|_{\DG}.
\]
%For a proof for bound \eqref{eq:bd_eta_1_eta_2_semi}, we refer to   Lemma \ref{lemma:bd_eta_1_eta_2_semi} in the Appendix.
Note that, with a triangle inequality, \eqref{eq:discreter_poincare}, and the stability of the interpolant \eqref{eq:approximation_prop_2}, 
\[
\|\bm{u}_h^{n} - \bm{u}_h^{n-1} \|^2_{L^3(\Omega)}
%\leq 2\tau^2 \|\delta_\tau \bm{e}_h^n \|^2_{L^3(\Omega)} + C\| \Pi_h \bm{u}^n  - \Pi_h \bm{u}^{n-1}\|^2_{\DG}  
 \leq 2\tau^2 \|\delta_\tau \bm{e}_h^n \|^2_{L^3(\Omega)} + C \tau \int_{t^{n-1}}^{t^n} \vert \partial_t \bm{u}\vert^2_{H^1(\Omega)}. 
\]
With Young's inequality and the above bounds, we attain 
\begin{multline}
|U^n_2| \leq \epsilon \tau \mu \| \delta_\tau \errn\|_{\DG}^2 + \frac{C}{\epsilon\mu} \tau \| \delta_\tau \bm{e}_h^n \|_{L^3(\Omega)}^2\|\errn\|^2_{\DG} \\ + \frac{C}{\epsilon \mu} \|\errn\|^2_{\DG} \int_{t^{n-1}}^{t^n} |\partial_t \bm{u}|_{H^1(\Omega)}^2. 
\end{multline}
Therefore,  we obtain the following bound on the nonlinear terms: 
\begin{align}
&\mathcal{N}_1(\delta_\tau \errn) + \mathcal{N}_2(\delta_\tau \errn) \leq 4\epsilon\tau \mu  \| \delta_\tau \errn  \|_{\DG}^2 + \frac{C}{\epsilon \mu}\tau \|\delta_\tau \bm{e}_h^n \|_{L^3(\Omega)}^2\|\errn \|^2_{\DG}  \nonumber   \\   & +\frac{C}{\epsilon \mu}\tau^{-1} \| \bm{e}_h^n - \bm{e}_h^{n-1}\|^2 + \frac{C}{\epsilon \mu} \|\errn\|^2_{\DG} \int_{t^{n-1}}^{t^n} |\partial_t \bm{u}|_{H^1(\Omega)}^2 \nonumber \\ %& + \frac{\tau^2}{8} \sum_{e\in\Gamma_h } \tilde{\sigma} h_e^{-1} ( \|[\delta_\tau (\phi_h^{n+1} -\phi_h^n)]\|_{L^2(e)}^2  + \| [\delta_\tau \phi_h^{n+1} ]\|_{L^2(e)}^2)  \\ 
&  + \frac{C}{\epsilon\mu}(1+\mu^{-1} + \tau^2+ h^{2k+2} +   \| \bm{e}_h^n\|^2 )\int_{t^{n-1}}^{t^{n+1}}\|\partial_t \bm{u}\|^2_{H^2(\Omega)} .  %+\frac{\delta_{n,1}}{2}\tau |b(\bm{u}_h^0, (\delta_\tau \bm{u}^2- \delta_\tau \Pi_h \bm{u}^2) \cdot \delta_\tau \tilde{\bm{e}}_h^2)| + \frac{\delta_{n,1}}{2} \tau |b(\delta_\tau \bm{u}_h^1,(\bm{u}^2 - \Pi_h \bm{u}^2) \cdot \delta_\tau \tilde{\bm{e}}_h^2)| \nonumber \\& 
%+ \frac{\delta_{n,1}}{2}\tau|b(\delta_\tau \bm{e}_h^1, \bm{u}^1 \cdot \delta_{\tau} \tilde{\bm{e}}_h^2)|. 
\end{align}
We now handle $\hat{R}_t(\delta_\tau\errn)$. For details, we refer to Lemma \ref{lemma:bound_hat_Rt_stability_proof} in the Appendix. %For $n\geq 0$, define $\bm{E}^n \in \bm{X}$: $\bm{E}^n = \Pi_h \bm{u}^n - \bm{u}^n$. We write: 
%\begin{align*}
%\hat{R}_t(\delta_\tau \errn) & = \frac{1}{\tau}\left(\tau(\partial_t \bm{u})^{n+1} - \tau (\partial_t \bm{u})^n - \bm{u}^{n+1} + 2\bm{u}^n - \bm{u}^{n-1}, \delta_\tau \errn \right) \\ & \quad  -\frac{1}{\tau}\left(\bm{E}^{n+1} - 2\bm{E}^n + \bm{E}^{n-1}, \delta_\tau \errn\right). 
%\end{align*}
With Taylor expansions, approximation properties \eqref{eq:approximation_prop_1} - \eqref{eq:approximation_prop_2}, and \eqref{eq:cfl_cond_detla}, we have for $n\geq1$:  
\begin{align}
    \hat{R}_t(\delta_\tau\errn) & \leq \epsilon \tau \mu \| \delta_\tau \errn\|_{\DG}^2 +  \frac{C}{\epsilon \mu} \int_{t^{n-1}}^{t^{n+1}} \| \partial_{tt} \bm{u }\|^2 + \frac{C}{\epsilon \mu} \int_{t^{n-1}}^{t^{n+1}}|\partial_t \bm{u}|_{H^{2}(\Omega)}^2. \label{eq:bound_hat_Rt_stability_proof}
\end{align}
With the above bounds and with choosing $\epsilon  = 1/56$, \eqref{eq:error_eq_4_dt} reads:
 \begin{align}
    &\frac{1}{2}( \|\delta_\tau \bm{e}_h^{n+1}\|^2 - \| \delta_\tau \bm{e}_h^n \|^2 + \|\delta_\tau \bm{e}_h^{n+1} - \delta_\tau \bm{e}_h^n \|^2 ) + \frac{ \tau \mu }{8}\|\delta_\tau \bm{\tilde{e}}_h^{n+1} \|_{\DG}^2  +\frac{\tau^2}{16}|\delta_\tau \phi_h^{n+1}|^2_{\DG} \nonumber \\ & + \frac{\tau^2}{2} (\aelip(\hat{\xi}_h^{n+1}, \hat{\xi}_h^{n+1}) - \aelip(\hat{\xi}_h^n , \hat{\xi}_h^n))+ \frac{\tau^2}{2}(\tilde{A}^n_1 - \tilde{A}^n_2) + \frac{\tau}{2\delta\mu}( \| \hat{S}_h^{n+1} \|^2 - \| \hat{S}_h^{n} \|^2 )    \nonumber \\ &\leq  \frac{C}{\mu} \int_{t^{n-1}}^{t^{n+1}} \| \partial_{tt} \bm{u}\|^2  + C(\tau +h^{2}\mu^{-1})  \int_{t^n}^{t^{n+1}}  |\partial_t p|^2_{H^1(\Omega)}  + \frac{C}{\mu}\tau \|\delta_\tau \bm{e}_h^n \|_{L^3(\Omega)}^2\|\errn \|^2_{\DG}  \nonumber \\   &  + \frac{C}{\mu}(1+\mu^2 + \mu^{-1} +  \| \bm{e}_h^n\|^2+  \|\errn\|^2_{\DG} )\int_{t^{n-1}}^{t^{n+1}}\|\partial_t  \bm{u}\|^2_{H^2(\Omega)} \nonumber  \\ 
    &+ \frac{C}{\mu}\tau^{-1} \| \bm{e}_h^n - \bm{e}_h^{n-1}\|^2 + \delta_{n,1} |b(\bm{e}_h^0, \delta_\tau \phi_h^2)|.
    \label{eq:error_eq_5_dt}
    \end{align}
   Note the following 
   \begin{multline}
      \frac{\tau^2}{2} \sum_{n=1}^m (\tilde{A}^n_1 - \tilde{A}^n_2 ) = \frac{\tau^2}{2}\left(  \sum_{e\in \Gamma_h}\frac{\tilde{\sigma}}{h_e} \| [\delta_\tau \phi_h^{m+1} ]\|^2 - \| \bm{G}_h([\delta_\tau \phi_h^{m+1}])\|^2 \right)  \\ + \frac{\tau^2}{2}\left(  \| \bm{G}_h([\delta_\tau \phi_h^{1}])\|^2 -\sum_{e\in \Gamma_h}\frac{\tilde{\sigma}}{h_e} \| [\delta_\tau \phi_h^{1} ]\|^2 \right). 
   \end{multline}
Using \eqref{eq:lift_prop_g}, the first term in the above expression is non-negative since $\tilde{\sigma} \geq \tilde{M}_k$ and the second term is bounded in absolute value by $C\tau^2 | \delta_\tau \phi_h^{1}|_{\DG}^2$.   
We sum \eqref{eq:error_eq_5_dt} from $n=1$ to $n=m$ and use the inverse estimate \eqref{eq:inverse_estimate}.   The continuity of
$\aelip$, Lemma \ref{lemma:first_err_estimate}, \eqref{eq:first_error_estimate}, the lower bound on $\tau$ in  \eqref{eq:cfl_cond_detla}, and the regularity assumptions yield
\begin{align}\label{eq:error_eq_6_dt}
& \|\delta_\tau \bm{e}_h^{m+1}\|^2 +  \sum_{n=1}^m \|\delta_\tau \bm{e}_h^{n+1} - \delta_\tau \bm{e}_h^{n}\|^2 + \frac{1}{8} \sum_{n=1}^m \tau^2 \vert \delta_\tau \phi_h^{n+1} \vert_{\DG}^2 + \frac{ \mu}{4} \sum_{n=1}^{m} \tau \|\delta_\tau \errn  \|_{\DG}^2 \nonumber \\ &\leq C_\mu \tilde{K}_\mu   + \| \delta_\tau \bm{e}_h^1 \|^2   + \frac{\tau}{2\delta \mu}\| \hat{S}_h^1\|^2 + C\tau^2 (|\hat{\xi}_h^1|_{\DG}^2 + |\delta_\tau \phi_h^{1}|_{\DG}^2)  \nonumber \\ & \quad  + \frac{C}{\mu}\tau \sum_{n=1}^{m} h^{-d/3}\|\delta_\tau \bm{e}_h^n \|^2\|\errn\|_{\DG}^2+ |b(\bm{e}_h^0, \delta_\tau \phi_h^2)|. 
 \end{align} 
where $\tilde{K}_\mu =  \sum_{i=-3}^{1} \mu^{i}$. To handle the last term, we apply \eqref{eq:boundformb}, the approximation properties, and \eqref{eq:cfl_cond_detla}.
\begin{equation}
   |b(\bm{e}_h^0, \delta_\tau \phi_h^2)| \leq C\|\bm{e}_h^0\| |\delta_\tau \phi_h^2|_{\DG}\leq C\tau|\bm{u}^0|_{H^{2} (\Omega)}|\delta_\tau \phi_h^2|_{\DG} \leq  C + \frac{\tau^{2}}{32} |\delta_\tau \phi_h^2|^2_{\DG}. 
\label{eq:boundbeh0}
\end{equation}
 We now derive a bound for the initial error $\Vert \bm{e}_h^1-\bm{e}_h^0\Vert$ by following an argument in
\cite{inspaper1}. We start with the equation (6.89)  in \cite{inspaper1} and choose $n=1$. 
All the terms in the right-hand side of (6.89) are bounded exactly in \cite{inspaper1} except for the time error term
$R_t(\err)$ defined in (6.15) or \eqref{eq:defRt}.  Under enough regularity for the exact solution, namely
$\partial_t \bm{u}\in L^2(0,T;H^2(\Omega)^d)$, we have 
\begin{align*}  
 |R_t(\err)|  \leq  \frac{C}{\mu}   \tau^2 \int_{t^{n-1}}^{t^n} \| \partial_{tt} \bm{u}\|^2  + \frac{C}{\mu} h^4  \int_{t^{n-1}}^{t^n} |\partial_t \bm{u}|_{H^2(\Omega)}^2  + \frac{ \tau \mu}{32} \| \err \|^2_{\DG}.
\end{align*}
This yields the following estimate for the initial error
\begin{align*} %\label{eq:initial_error_eh} 
&\frac12 \| \bm{e}_h^1 - \bm{e}_h^0 \|^2 
+ \frac{\tau^2}{16} ( |\phi_h^1|_{\DG}^2 +  | p_h^1 + S_h^1 |^2_{\DG} ) 
+ \frac{\tau }{2\delta \mu }  \|S_h^1\|^2 +\frac{\tau \mu}{8} \|\tilde{\bm{e}}_h^1 \|_{\DG}^2 \\ & \leq C \left( \frac{1}{\mu} + \mu + 1\right) \tau^2  + \tau \vert  b(\bm{e}_h^0, \phi_h^1)\vert + \frac{1}{2}\tau |b(\bm{e}_h^0, \Pi_h \bm{u}^1 \cdot \tilde{\bm{e}}_h^1)|,. 
\end{align*}
where $S_h^1 = \delta \mu (\nabla_h \cdot\bm{v}_h^1-R_h([\bm{v}_h^1]))$. Recall that by \eqref{eq:def_pi_h} and \eqref{eq:def_b_lift}, we have for $n \geq 0$ : 
\[b( \Pi_h \bm{u}^n, q_h ) = (\nabla_h \cdot \Pi_h \bm{u}^n - R_h([\Pi_h \bm{u}^n] ) , q_h) = 0,\,\,  \forall  q_h \in M_h \] 
Thus,  $\hat{S}_h^1 = \delta \mu ( \nabla_h \cdot \delta_\tau \bm{v}_h^1 - R_h([\delta_\tau \bm{v}_h^1])$. By recalling that $\bm{v}_h^0 = \Pi_h \bm{u}^0$ and by using the above equality, we obtain 
\[
\hat{S}_h^1 = \frac{1}{\tau}S_h^1, \quad \hat{\xi}_h^1 = \frac{1}{\tau} (p_h^1+S_h^1).
\]
With \eqref{eq:boundformb} and approximation property \eqref{eq:approximation_prop_1}, we have
\begin{align*}
    b(\bm{e}_h^0, \phi_h^1) \leq C\|\bm{e}_h^0\||\phi_h^1|_{\DG} \leq \frac{\tau}{32}|\phi_h^1|^2_{\DG} + C\tau^{-1}h^{2k+2}|\bm{u}^0|^2_{H^{k+1}(\Omega)}. 
\end{align*}
We split the last term, and we use  \eqref{eq:first_estimate_nonlinear_form}, \eqref{eq:second_estimate_nonlinear_term_C1}, \eqref{eq:approximation_prop_1}, and the regularity assumption. 
\begin{align*}
&|b(\bm{e}_h^0, \Pi_h \bm{u}^1 \cdot \tilde{\bm{e}}_h^1)| \leq |b(\bm{e}_h^0, (\Pi_h \bm{u}^1 - \bm{u}^1) \cdot \tilde{\bm{e}}_h^1)| + |b(\bm{e}_h^0, \bm{u}^1 \cdot \tilde{\bm{e}}_h^1)| \\ 
& \leq C \|\bm{e}_h^0\|(|\bm{u}^1|_{H^{k+1}(\Omega)} + \vertiii{\bm{u}^1}) \|\tilde{\bm{e}}_h^1\|_{\DG} \leq \frac{C}{\mu}h^{2k+2} + \frac{\mu}{16}\|\tilde{\bm{e}}_h^1\|_{\DG}^2. 
\end{align*}
With \eqref{eq:cfl_cond_detla}, the above bounds, and with recalling that  $\phi_h^0 =  p_h^0 = 0$,  we obtain  
\begin{multline}
\| \delta_\tau \bm{e}_h^1 \|^2  + \frac{\tau}{\delta \mu}\| \hat{S}_h^1\|^2 + \frac{\tau^2}{16}   (|\hat{\xi}_h^1|_{\DG}^2  + | \delta_\tau \phi_h^{1} |^2_{\DG} ) + \frac{\tau\mu}{16} \|\tilde{\bm{e}}_h^1\|_{\DG}^2%\\  =   \frac{1}{\tau^2} \left( \| \bm{e}_h^1 - \bm{e}_h^0 \|^2 + \frac{\tau}{2\delta\mu}\|S_h^1 \|^2 + C \tau^2 (| p^1_h + S_h^1 |^2_{\DG} + |\phi_h^1|^2_{\DG})\right) \\ 
\\\leq  C \left( \frac{1}{\mu} + \mu + 1\right).  \label{eq:initial_error_eh}
\end{multline}
Thus, with \eqref{eq:boundbeh0} and \eqref{eq:initial_error_eh},  the bound \eqref{eq:error_eq_6_dt} becomes 
\begin{multline}
    \|\delta_\tau \bm{e}_h^{m+1}\|^2 + \sum_{n=1}^m \|\delta_\tau \bm{e}_h^{n+1} - \delta_\tau \bm{e}_h^{n}\|^2 
+ \frac{\tau^2}{32} \sum_{n=1}^m  \vert \delta_\tau \phi_h^{n+1} \vert_{\DG}^2+ \frac{\mu \tau}{4} \sum_{n=1}^{m} \|\delta_\tau \errn  \|_{\DG}^2 \\ \leq C_\mu \tilde{K}_\mu + \frac{C}{\mu} \tau \sum_{n=1}^{m} h^{-d/3}\|\delta_\tau \bm{e}_h^n \|^2\|\errn\|_{\DG}^2. \label{eq:error_eq_semi_final_dt} 
\end{multline}
It remains to handle the last term. We will follow a similar technique as the one used in  \cite{nochetto2005gauge}.  From Lemma \ref{lemma:first_err_estimate} and condition \eqref{eq:cfl_cond_detla}, we have: 
\begin{align*} 
    \|\delta_\tau \bm{e}_h^n \|^2 & = \frac{1}{\tau^2} \|\bm{e}_h^n - \bm{e}_h^{n-1} \|^2  \leq C_\mu \left(1 + 1/\mu + \mu \right)\tau^{-1}, \quad 
     \sum_{n=1}^{m} \|\errn \|_{\DG}^2 \leq C_\mu \tilde{K}_\mu.   
\end{align*}
Using the above estimates in \eqref{eq:error_eq_semi_final_dt} and for $h$ small enough, we obtain:
\begin{equation}
\| \delta_\tau \bm{e}_h^{m+1} \|^2 \leq C_\mu \tilde{K}_\mu + C_\mu^2  \tilde{K}_\mu^2 h^{-d/3} \leq C_\mu^2 \tilde{K}_\mu^2 h^{-d/3}, \quad m \geq 0. 
\end{equation}
We will iteratively apply this bound in  \eqref{eq:error_eq_semi_final_dt}. Applying it once, and using Lemma \ref{lemma:first_err_estimate}: 
\begin{align}
    \|\delta_\tau \bm{e}_h^{m+1}\|^2 \leq  C_\mu \tilde{K}_\mu + C_\mu^2 \mu^{-1} \tilde{K}_\mu^2 \tau h^{-2d/3} \sum_{n=1}^{m} \|\errn\|_{\DG}^2  \leq C_\mu \tilde{K}_\mu + C_\mu^3\mu^{-1}\tilde{K}_\mu^3\tau h^{-2d/3}. \nonumber  % \label{eq:error_eq_semi_final_dt+1} 
\end{align}
%\textit{For our reference, here is an example: If $d = 2$ and $ \gamma = 1$ (we allow $\gamma = 1$ in this case), then condition \eqref{eq:cfl_cond_detla} is: $c_1 h^2 \leq \tau \leq c_2h^{4/3}$. The result is then concluded here. The stability constant is $C_\mu K_\mu + C_\mu^3\mu^{-1}K_\mu^3c_2$. }
Noting that by \eqref{eq:cfl_cond_detla}, we have that $\tau h^{-2d/3}\leq c_2 h^{(\gamma -1 )d/3}$. Since $\gamma < 1$, for $h$ small enough, we have: 
\begin{align*}
    \|\delta_\tau \bm{e}_h^{m+1}\|^2 
    %\leq  C_\mu K_\mu + C_\mu^3\mu^{-1}K_\mu^3\tau h^{-2d/3} & 
    \leq C_\mu \tilde{K}_\mu + C_\mu^3\mu^{-1}\tilde{K}_\mu^3c_2 h^{(\gamma -1 )d/3} 
\leq C_\mu^3 \mu^{-1}\tilde{K}_\mu^3c_2 h^{(\gamma -1 )d/3}.
\end{align*}
Select $\beta$ to be the smallest integer such that $\beta\gamma\geq1$ and  iteratively apply above bound in \eqref{eq:error_eq_semi_final_dt} $\beta$ times. With the condition on $\tau$ \eqref{eq:cfl_cond_detla},we obtain:  
\begin{align}
    \| \delta_\tau \bm{e}_h^{m+1} \|^2   \leq C_\mu \tilde{K}_\mu +  (C_\mu \tilde{K}_\mu)^{\beta + 2}\mu^{-\beta}c_2^{\beta}h^{(\beta \gamma -1)d/3} \leq C_\mu \tilde{K}_\mu +(C_\mu \tilde{K}_\mu)^{\beta + 2}\mu^{-\beta}c_2^{\beta}. \nonumber
\end{align}
The result then follows. 
\end{proof}
We will now use Lemma \ref{lemma:stability_time_derivative} and the dual problem \eqref{eq:aux_pb_1} - \eqref{eq:aux_pb_3} to obtain an $\ell^2$ error estimate for $\delta_\tau \bm{e}_h^{n+1}$. 
\begin{lemma}\label{lemma:error_delta_velocity}
% Assume that $\partial_t \bm{u} \in L^{\infty}(0,T; L^2(\Omega)^d)$. 
 We assume that the hypothesis of Lemma \ref{lemma:stability_time_derivative} hold. If $\tau$ is small enough $(\tau\leq \tau_0)$,  we have the following error estimate.  For $1 \leq m \leq N_T-1$,  
    \begin{equation}
       \frac{\mu}{2} \tau \sum_{n=1}^m \| \delta_\tau \bm{e}_h^{n+1}\|^2 \leq \tilde{C}_{\gamma,\mu}(\tau + h^{2k}). 
    \end{equation}
    Here, $\tilde{C}_{\gamma,\mu} = C_\mu K_\mu + C_{\gamma, \mu} C_\mu (1+K_\mu)$ is independent of $h$ and $\tau$.
% but depends inversely on $\mu$ and $\gamma$:  $\tilde{C}_{\gamma,\mu} = C_\mu K_\mu + C_{\gamma, \mu} C_\mu (1+K_\mu)$ where $C_\mu$ depends on $e^{\frac{1}{\mu}}$.
\end{lemma}
\begin{proof}
From \eqref{eq:aux_1_dg}, we have for $n \geq 0$
\begin{align}
    a_\mathcal{D}(\delta_\tau \bm{U}_h^{n+1}, \bm{\theta}_h) - b(\bm{\theta}_h, \delta_\tau P_h^{n+1}) = (\delta_\tau \bm{\chi}^{n+1}, \bm{\theta}_h), \quad  \forall \bm{\theta}_h \in \bm{X}_h. \label{eq:aux_pb_dt}
\end{align}
Let $\bm{\theta}_h  = \delta_\tau \bm{U}_h^{n+1}$ in \eqref{eq:first_err_eq_dteh} and use \eqref{eq:aux_2_dg}. After some rearranging, we have  
\begin{multline}
    (\delta_\tau (\bm{v}_h^{n+1} - \bm{u}^{n+1}) - \delta_\tau \bm{\chi}^n , \delta_\tau \bm{U}_h^{n+1}) + \tau \mu a_{\mathcal{D}}(\delta_\tau \bm{\tilde{e}}_h^{n+1}, \delta_\tau \bm{U}_h^{n+1})\\ = - \tau b(\delta_\tau \bm{U}_h^{n+1}, \delta_\tau p^{n+1})- \tau \mu a_{\mathcal{D}}(\delta_\tau \Pi_h \bm{u}^{n+1} - \delta_\tau \bm{u}^{n+1}, \delta_\tau \bm{U}_h^{n+1}) \\  + \mathcal{N}_1(\delta_\tau \bm{U}_h^{n+1})   + \mathcal{N}_2(\delta_\tau \bm{U}_h^{n+1})+ \bar{R}_t(\delta_\tau \bm{U}_h^{n+1}).  \label{eq:first_err_eq_delta}
\end{multline}
Here, $$ \bar{R}_t(\delta_\tau \bm{U}_h^{n+1}) = ( (\partial_t \bm{u})^{n+1} - (\partial_t \bm{u})^n -(\delta_\tau \bm{u}^{n+1} - \delta_\tau \bm{u}^n), \delta_\tau \bm{U
}_h^{n+1}).$$
The difficulty is in bounding the last three terms. For completeness, we provide an overview of the way the other terms are handled. We first observe that from \eqref{eq:update_velocity_delta},  \eqref{eq:aux_2_dg}, and \eqref{eq:aux_pb_dt} we have:
\begin{align*}
& (\delta_\tau (\bm{v}_h^{n+1} - \bm{u}^{n+1})  - \delta_\tau \bm{\chi}^n , \delta_\tau \bm{U}_h^{n+1})  = (\delta_\tau \bm{\chi}^{n+1} )- \delta_\tau \bm{\chi}^{n}, \delta_\tau \bm{U}_h^{n+1}) \\ & \quad - \tau b(\delta_\tau \bm{U}_h^{n+1}, \delta_\tau \phi_h^{n+1}) =  a_\mathcal{D}(\delta_\tau \bm{U}_h^{n+1} - \delta_\tau \bm{U}_h^n, \delta_\tau \bm{U}_h^{n+1}) 
\\ & =  \frac{1}{2}\left(a_\mathcal{D}(\delta_\tau \bm{U}_h^{n+1},\delta_\tau \bm{U}_h^{n+1}) - a_\mathcal{D}(\delta_\tau \bm{U}_h^{n},\delta_\tau \bm{U}_h^{n})\right) \\ & \quad +\frac{1}{2}a_\mathcal{D}(\delta_\tau \bm{U}_h^{n+1}- \delta_\tau \bm{U}_h^{n}, \delta_\tau \bm{U}_h^{n+1}- \delta_\tau \bm{U}_h^{n} ).
\end{align*}
In addition, with choosing $\bm{\theta}_h = \delta_\tau \errn$ in \eqref{eq:aux_pb_dt}, we have 
\begin{align*}
&  a_\mathcal{D}(\delta_\tau \errn, \delta_\tau \bm{U}_h^{n+1})=  (\delta_\tau \bm{\chi}^{n+1}, \delta_\tau \errn) +  b(\delta_\tau \errn, \delta_\tau P_h^{n+1})  \\ 
& = (\delta_\tau \bm{\chi}^{n+1},\delta_\tau \errn - \delta_\tau \bm{e}_h^{n+1}) + (\delta_\tau (\Pi_h \bm{u}^{n+1}) - \delta_\tau \bm{u}^{n+1}, \delta_\tau \bm{e}_h^{n+1}) + \|\delta_\tau \bm{e}_h^{n+1}\|^2 \\ & \quad   + b(\delta_\tau \errn - \delta_\tau \bm{e}_h^{n+1}, \delta_\tau P_h^{n+1})  +  b(\delta_\tau \bm{e}_h^{n+1}, \delta_\tau P_h^{n+1}).
\end{align*}
% From \eqref{eq:update_velocity_delta}, \eqref{eq:property_err_delta}, \eqref{eq:lift_prop_g}, and the assumption that $\tilde{\sigma} \geq \tilde{M}_k^2$ , we have: 
% \begin{multline}
%     (\delta_\tau \errn - \delta_\tau \bm{e}_h^{n+1}, \delta_\tau \bm{e}_h^{n+1}) = - \tau  b(\delta_\tau \bm{e}_h^{n+1}, \delta_\tau \phi_h^{n+1}) \\ = \tau^2  \sum_{e\in \Gamma_h} \frac{\tilde{\sigma}}{h_e} \| \delta_\tau \phi_h^{n+1}\|^2_{L^2(e)} - \tau^2  \| \bm{G}_h([\delta_\tau \phi_h^{n+1}]) \|^2 \geq 0. 
% \end{multline}
From \eqref{eq:aux_pb_1}, we have for $n \geq 1$  
\begin{align}
-\Delta (\delta_\tau \bm{U}^{n+1}) + \nabla (\delta_\tau P^{n+1}) = \delta_\tau \bm{\chi}^{n+1}. 
\end{align}
Since the domain is assumed to be  convex and similar to \eqref{eq:regularity_assumption} and \eqref{eq:int99}, we have
\begin{align}
\|\delta_\tau \bm{U}_h^{n+1}\|_{\DG} +  \vert \delta_\tau P_h^{n+1}\vert_{\mathrm{DG}}  \leq C( \| \delta_\tau  \bm{U}^{n+1}\|_{H^2(\Omega)} + \vert \delta_\tau P^{n+1}\vert_{H^1(\Omega)})\nonumber\\ \leq C \|\delta_\tau \bm{\chi}^{n+1}\| \leq C \|\delta_\tau \bm{e}_h^{n+1}\| + C\tau^{-1/2}h^{k+1} \left(\int_{t^{n}}^{t^{n+1}} |\partial_t \bm{u}|^2_{H^{k+1}(\Omega)}\right)^{1/2}.  \label{eq:H2_delta_bound}
\end{align} 
From \eqref{eq:update_velocity_delta}, \eqref{eq:boundformb} and \eqref{eq:H2_delta_bound}, we have
\begin{align*}
|(\delta_\tau \bm{\chi}^{n+1},\delta_\tau \errn - \delta_\tau \bm{e}_h^{n+1})| %+ |b(\delta_\tau \errn - \delta_\tau \bm{e}_h^{n+1}, \delta_\tau P_h^{n+1})|
\leq & \epsilon \| \delta_\tau \bm{e}_h^{n+1}\|^2 +  C \left(\frac{1}{\epsilon} +1\right)  \tau^2\vert \delta_\tau \phi_h^{n+1}\vert_{\DG}^2  \\ &\quad +  C \tau^{-1}h^{2k+2} \int_{t^{n}}^{t^{n+1}} |\partial_t \bm{u}|^2_{H^{k+1}(\Omega)}.
\end{align*}
Similarly with \eqref{eq:boundformb} and \eqref{eq:H2_delta_bound}, we have
\begin{align*}
|b(\delta_\tau \errn - \delta_\tau \bm{e}_h^{n+1}, \delta_\tau P_h^{n+1})|
\leq & \epsilon \| \delta_\tau \bm{e}_h^{n+1}\|^2 +  C \left(\frac{1}{\epsilon} +1\right)  \tau^2\vert \delta_\tau \phi_h^{n+1}\vert_{\DG}^2  \\ &\quad +  C \tau^{-1}h^{2k+2} \int_{t^{n}}^{t^{n+1}} |\partial_t \bm{u}|^2_{H^{k+1}(\Omega)}.
\end{align*}
In addition, from \eqref{eq:property_err_delta}, Cauchy-Schwarz's inequality and \eqref{eq:lift_prop_g}, we obtain: 
\begin{align*}
&|(\delta_\tau (\Pi_h \bm{u}^{n+1}) - \delta_\tau \bm{u}^{n+1}, \delta_\tau \bm{e}_h^{n+1})| + |b(\delta_\tau \bm{e}_h^{n+1}, \delta_\tau P_h^{n+1})| \\ &\leq C h^{k+1} |\delta_\tau \bm{u}^{n+1} |_{H^{k+1}(\Omega)} \|\delta_\tau \bm{e}_h^{n+1} \| + C\tau \vert \delta_\tau \phi_h^{n+1}\vert_{\DG} \vert \delta_\tau P_h^{n+1}\vert_{\DG} \\
& \leq  \epsilon \| \delta_\tau \bm{e}_h^{n+1}\|^2 +  C \left(\frac{1}{\epsilon} +1\right) \tau^2 | \delta_\tau \phi_h^{n+1}|_{\DG}^2 \\    & +C \left(\frac{1}{\epsilon} +1\right)  \tau^{-1}h^{2k+2} \int_{t^{n}}^{t^{n+1}} |\partial_t \bm{u}|^2_{H^{k+1}(\Omega)}  . 
\end{align*}
Further, note that by using \eqref{eq:aux_2_dg}, the definition of $L^2$ projection, and a bound similar to \eqref{eq:error_dg_aux},   we obtain: %{\blue (question: do we need $\|\delta_\tau \bm{U}_h^{n+1}\|_{\DG} \leq Ch \|\delta_\tau \bm{e}_h^{n+1}\|$? how is the factor $h^2$ in first ineq below obtained?)}
\begin{align*}
    & b(\delta_\tau \bm{U}_h^{n+1}, \delta_\tau p^{n+1})  =  b(\delta_\tau \bm{U}_h^{n+1}, \delta_\tau p^{n+1} - \pi_h (\delta_\tau p^{n+1})) \\ & \leq  \sum_{e \in \Gamma_h \cup \partial \Omega } \left|\int_e \{\delta_\tau p^{n+1} - \pi_h(\delta_\tau p^{n+1})\}[\delta_\tau \bm{U}_h^{n+1}]\cdot \bm{n}_e \right| \leq C h^2|\delta_\tau p^{n+1}|_{H^1(\Omega)} \|\delta_\tau \bm{\chi}^{n+1}\| \\ 
    & \leq \epsilon \mu \| \delta_\tau \bm{e}_h^{n+1}\|^2 + C\left(\frac{1}{\epsilon\mu}+1 \right)\tau^{-1} h^4  \int_{t^{n}}^{t^{n+1}} \vert \partial_t p \vert_{H^1(\Omega)}^2 \\ &  \quad  +  C \tau^{-1}h^{2k+2} \int_{t^{n}}^{t^{n+1}} |\partial_t \bm{u}|^2_{H^{k+1}(\Omega)}.
\end{align*}
By taking $\bm{\theta}_h = \Pi_h \delta_\tau \bm{u}^{n+1} - \delta_\tau Q_h \bm{u}^{n+1}$ in \eqref{eq:aux_pb_dt}  (where we recall $Q_h$ is the elliptic projection operator defined in \eqref{eq:elliptic_projection}),  using \eqref{eq:boundformb}  and \eqref{eq:first_err_eq_delta} , we obtain: 
\begin{align}
&|a_\mathcal{D}(\delta_\tau \Pi_h \bm{u}^{n+1} - \delta_\tau \bm{u}^{n+1}, \delta_\tau \bm{U}_h^{n+1})|  \leq |(\delta_\tau \bm{\chi}^{n+1},\delta_\tau  \Pi_h \bm{u}^{n+1} - \delta_\tau Q_h\bm{u}^{n+1} )|\nonumber \\  &   + | b(\delta_\tau \Pi_h  \bm{u}^{n+1} - \delta_\tau Q_h\bm{u}^{n+1}, \delta_\tau P_h^{n+1})|  
\leq C \| \delta_\tau\Pi_h \bm{u}^{n+1} - \delta_\tau Q_h \bm{u}^{n+1} \| \, \|\delta_\tau \bm{\chi}^{n+1}\|\nonumber  \\ 
&  \leq \epsilon   \|\delta_\tau \bm{e}_h^{n+1}\|^2 + C\left(\frac{1}{\epsilon}+1 \right) \tau^{-1} h^{2k +2} \int_{t^{n}}^{t^{n+1}} \vert \partial_t \bm{u}\vert_{H^{k +1}(\Omega)}^2. 
\end{align}
To handle $\bar{R}_t(\delta_\tau \bm{U}_h^{n+1})$, introduce the function $\bm{G}^n = \tau (\partial_t \bm{u})^{n} - (\bm{u}^n - \bm{u}^{n-1})$ for $n \geq 1$. Clearly, $\bm{G}^n$ belongs to $\bm{X}$.%\begin{equation}
%\bm{G}^n = \tau (\partial_t \bm{u})^{n} - (\bm{u}^n - \bm{u}^{n-1}), \quad n \geq 1.  \label{eq:def_E_n}
%\end{equation}
We then write: 
\begin{align*}
&\bar{R}_{t}(\delta_\tau  \bm{U}_h^{n+1}) = \frac{1}{\tau}(\bm{G}^{n+1}, \delta_\tau \bm{U}_h^{n+1})-\frac{1}{\tau} (\bm{G}^n, \delta_\tau \bm{U}_h^{n}) - \frac{1}{\tau} (\bm{G}^{n}, \delta_\tau \bm{U}_h^{n+1} - \delta_\tau \bm{U}_h^n).
\end{align*}
%To simplify the write up, we denote by \begin{align*}
 %   \hat{R}_1^n = \frac{1}{\tau}(\bm{G}^{n}, \delta_\tau \bm{U}_h^n), \quad \hat{R}_2^n = \frac{1}{\tau}(\bm{E}^n - \bm{E}^{n-1}, \delta_\tau \bm{U}_h^n) , \quad n \geq 1.  
%\end{align*} 
We use Cauchy-Schwarz's inequality, \eqref{eq:discreter_poincare}, and Taylor expansions.
\begin{align}
&|(\bm{G}^n,  \delta_\tau \bm{U}_h^{n+1} - \delta_\tau \bm{U}_h^n) | \leq C \| \bm{G}^n\|\|\delta_\tau \bm{U}_h^{n+1} - \delta_\tau \bm{U}_h^n\|_{\DG} \nonumber \\ &  \leq  C \tau^2 \int_{t^{n-1}}^{t^{n}} \|\partial_{tt} \bm{u}\|^2 +  \frac{1}{8}\tau \|  \delta_\tau \bm{U}_h^{n+1} - \delta_\tau \bm{U}_h^n\|_{\DG}^2.\label{eq:bounding_hat_R_1} \end{align}
With the above bounds and the coercivity of $a_\mathcal{D}$ \eqref{eq:coercivity_a_ellip}, \eqref{eq:first_err_eq_delta} reads: 
\begin{align}
&\frac{1}{2}( a_\mathcal{D}(\delta_\tau \bm{U}_h^{n+1},\delta_\tau \bm{U}_h^{n+1})  - a_\mathcal{D}(\delta_\tau \bm{U}_h^{n},\delta_\tau \bm{U}_h^{n}) ) +\frac{1}{8} \| \delta_\tau \bm{U}_h^{n+1}- \delta_\tau \bm{U}_h^{n}\|_{\DG}^2\nonumber \\ &  + (1-5\epsilon)\tau \mu \| \delta_\tau \bm{e}_h^{n+1} \|^2 \leq C \mu \left(\frac{1}{\epsilon} +1\right) \tau^3 | \delta_\tau \phi_h^{n+1}|_{\DG}^2  +  C\tau \int_{t^{n-1}}^{t^{n}} \| \partial_{tt}\bm{u} \|^2 \nonumber \\ & +C \left(\frac{\mu}{\epsilon} +\mu + 1 \right) h^{2k+2} \int_{t^{n}}^{t^{n+1}} |\partial_t \bm{u}|^2_{H^{k+1}(\Omega)} +C\left(\frac{1}{\epsilon\mu}+1 \right)   h^4   \int_{t^{n}}^{t^{n+1}} \vert \partial_t p \vert_{H^1(\Omega)}^2 \nonumber \\ &  +  \mathcal{N}_1(\delta_\tau \bm{U}_h^{n+1}) + \mathcal{N}_2(\delta_\tau \bm{U}_h^{n+1})  + \frac{1}{\tau}(\bm{G}^{n+1}, \delta_\tau \bm{U}_h^{n+1})-\frac{1}{\tau} (\bm{G}^n, \delta_\tau \bm{U}_h^{n}). \label{eq:error_delta_semi_final}
\end{align}
It remains to handle the nonlinear terms. To this end,  we use \eqref{eq:expanding_N1}-\eqref{eq:expanding_N2} and write:
\begin{equation}
    \mathcal{N}_1(\delta_\tau \bm{U}_h^{n+1}) + \mathcal{N}_2(\delta_\tau \bm{U}_h^{n+1}) = \sum_{i=1}^3 (\xi_i^n(\delta_\tau \bm{U}_h^n) + \vartheta_i^n(\delta_\tau \bm{U}_h^n)) = \sum_{i=1}^3 \beta_i. % , \quad  \mathrm{with} \quad \beta_i^n = \xi_i^n(\delta_\tau \bm{U}_h^n) + \vartheta_i^n(\delta_\tau \bm{U}_h^n).  \nonumber 
\end{equation}
% with 
% \begin{align*}
%     \beta_1^n & =a_{\mathcal{C}}(\bm{u}^{n+1}; \bm{u}^{n+1} - \bm{u}^n ,\bm{u}^{n+1}, \delta_\tau \bm{U}_h^{n+1}) -a_{\mathcal{C}}(\bm{u}^{n}; \bm{u}^{n} - \bm{u}^{n-1} ,\bm{u}^{n}, \delta_\tau \bm{U}_h^{n+1}), \\ 
%     \beta_2^n & = a_{\mathcal{C}}(\bm{u}_h^{n}; \bm{u}^{n} - \bm{u}_h^n ,\bm{u}^{n+1}, \delta_\tau \bm{U}_h^{n+1}) - a_{\mathcal{C}}(\bm{u}_h^{n-1}; \bm{u}^{n-1} - \bm{u}_h^{n-1} ,\bm{u}^{n}, \delta_\tau \bm{U}_h^{n+1}), \\ 
%     \beta_3^n & = a_{\mathcal{C}}(\bm{u}_h^{n}; \bm{u}_h^{n},\bm{u}^{n+1} - \bm{v}_h^{n+1}, \delta_\tau \bm{U}_h^{n+1}) -a_{\mathcal{C}}(\bm{u}_h^{n-1}; \bm{u}_h^{n-1},\bm{u}^{n} - \bm{v}_h^{n}, \delta_\tau \bm{U}_h^{n+1}). 
% \end{align*}
We begin with bounding $\beta_1^n$, we have 
\begin{align}
\beta_1^n & = \mathcal{C}(\bm{u}^{n+1} - \bm{u}^n ,\bm{u}^{n+1} - \bm{u}^n , \delta_\tau \bm{U}_h^{n+1}) 
+ \mathcal{C}(\bm{u}^{n+1} -2\bm{u}^n +  \bm{u}^{n-1} ,\bm{u}^{n}, \delta_\tau \bm{U}_h^{n+1}) \nonumber \\ 
& = 
\sum_{E \in \mesh_h} \int_{E } ((\bm{u}^{n+1} - \bm{u}^n) \cdot \nabla(\bm{u}^{n+1} - \bm{u}^n) + (\bm{u}^{n+1} - 2\bm{u}^n + \bm{u}^{n-1}) \cdot \nabla \bm{u}^n ) \cdot \delta_\tau \bm{U}_h^{n+1} \nonumber \\ %\\ &+ \sum_{E \in \mesh_h} \int_{E } ((\bm{u}^{n+1} - 2\bm{u}^n + \bm{u}^{n-1}) \cdot \nabla \bm{u}^n ) \cdot \delta_\tau \bm{U}_h^{n+1}, \nonumber \\ 
& \leq (\|\bm{u}^{n+1} - \bm{u}^n \|_{L^3(\Omega)} \|\nabla(\bm{u}^{n+1} - \bm{u}^{n}) \| \nonumber  \\ & \quad + \| \bm{u}^{n+1} - 2\bm{u}^n + \bm{u}^{n-1}\| |\bm{u}^n|_{W^{1,3}(\Omega)} )\|\delta_\tau \bm{U}_h^{n+1} \|_{L^6(\Omega)}. \nonumber 
\end{align} 
 Since $H^1(\Omega)$ is embedded into $L^3(\Omega)$, 
%Recall that by a Sobolev embedding, we have 
%\begin{equation*}
%    \| \bm{u}^{n+1} - \bm{u}^n  \|_{L^3(\Omega)} \leq C \| \nabla(\bm{u}^{n+1} - \bm{u}^n) \|. 
%\end{equation*}
 with the above bound, \eqref{eq:discreter_poincare}, the assumption that $\bm{u} \in L^{\infty}(0,T; H^{k+1}(\Omega)^d)$, Taylor expansions and Young's inequality, we have 
\begin{align}
|\beta_1^n| &\leq C \tau \| \delta_\tau \bm{U}_h^{n+1} \|^2_{\DG} + C \tau^{-1} \vert \bm{u}^{n+1} - \bm{u}^n \vert_{H^1(\Omega)}^4 + C\tau^{-1 }\| \bm{u}^{n+1} - 2\bm{u}^n + \bm{u}^{n-1}\|^2 \nonumber \\ 
& \leq C \tau \| \delta_\tau \bm{U}_h^{n+1} \|^2_{\DG} + C\tau \int_{t^{n}}^{t^{n+1}} \vert \partial_t \bm{u} \vert_{H^1(\Omega)}^2 \int_0^T \vert \partial_t \bm{u} \vert_{H^1(\Omega)}^2 + C \tau^2 \int_{t^{n-1}}^{t^{n+1}} \| \partial_{tt} \bm{u} \|^2 .\nonumber
\end{align}
Taking the sum of the above inequality from $n =1$ to $n =m $ yields: \begin{equation}
\sum_{n=1}^m|\beta_1^n| \leq C \tau \sum_{n=1}^m \| \delta_\tau \bm{U}^{n+1}\|_{\DG}^2 + C\tau.  \label{eq:bound_beta1}
\end{equation}
We write $\beta_2^n$ as such: 
\begin{align}
    \beta^n_2 &= \mathcal{C}(\bm{u}^{n} - \bm{u}_h^n ,\bm{u}^{n+1} - \bm{u}^n , \delta_\tau \bm{U}_h^{n+1}) + \mathcal{C}(\bm{u}^n - \bm{u}_h^n - (\bm{u}^{n-1} - \bm{u}_h^{n-1}) ,\bm{u}^{n}, \delta_\tau \bm{U}_h^{n+1}) \nonumber \\ 
& = \tau \mathcal{C}(\bm{u}^n-\Pi_h \bm{u}^n, \delta_\tau \bm{u}^{n+1}, \delta_\tau \bm{U}_h^{n+1}) 
-\tau \mathcal{C}(\bm{e}_h^n, \delta_\tau \bm{u}^{n+1} , \delta_\tau \bm{U}_h^{n+1}) \nonumber \\&\quad 
+ \tau \mathcal{C}(\delta_\tau (\bm{u}^n - \Pi_h \bm{u}^n) ,\bm{u}^{n}, \delta_\tau \bm{U}_h^{n+1}) 
- \tau \mathcal{C}(\delta_\tau \bm{e}_h^n,\bm{u}^{n}, \delta_\tau \bm{U}_h^{n+1}).\nonumber
 \end{align}
 We apply \eqref{eq:third_estimate_nonlinear_term} to bound the first and third terms. We use  \eqref{eq:first_estimate_nonlinear_form} to bound the second and fourth terms. 
 \begin{align}
& |\beta^n_2| \leq C\tau(\| \bm{e}_h^n\| + h^{k+1}|\bm{u}^n|_{H^{k+1}(\Omega)}) \vertiii{\delta_\tau \bm{u}^{n+1}} \| \delta_\tau \bm{U}_h^{n+1}\|_{\DG} \nonumber \\ 
  &  +  C\tau(\|\delta_\tau \bm{e}_h^{n+1}\| + \| \delta_\tau \bm{e}_h^n - \delta_\tau \bm{e}_h^{n+1}\| + h^{k+1}|\delta_\tau \bm{u}^n|_{H^{k+1}(\Omega)}) \vertiii{\bm{u}^{n}}\|\delta_{\tau} \bm{U}_h^{n+1}\|_{\DG}. \nonumber
 \end{align}
% With Young's inequality and  the assumption that $\bm{u} \in L^{\infty}(0,T;H^{k+1}(\Omega)^d)$, we have:, we obtain: 
% \begin{align}
% &|\beta^n_2|\leq  C \left(1+\frac{1}{\mu}\right) \tau \|\delta_\tau \bm{U}_h^{n+1} \|_{\DG}^2 + C \tau (h^{2k+2} + \|\bm{e}_h^{n}\|^2 + \tau^3 \vert \phi_h^n \vert_{\DG}^2 )\vertiii{\delta_\tau \bm{u}^{n+1}}^2 \nonumber  \\ & \quad + C \tau h^{2k + 2} \vert \delta_\tau \bm{u}^n \vert^2_{H^{k+1}(\Omega)}  + \ell \tau \|\delta_\tau \bm{e}_h^n - \delta_\tau \bm{e}_h^{n+1}\|^2 + \ell \tau \mu  \|\delta_\tau \bm{e}_h^{n+1}\|^2 + C\tau^3 \vert \delta_\tau \phi_h^{n} \vert_{\DG}^2
% \end{align}
With Young's inequality, \eqref{eq:controlling_triplenorm_delta} and  the assumption that $\bm{u} \in L^{\infty}(0,T;H^{k+1}(\Omega)^d)$:% we have:  \Rd BR: I made some small edits here \Bk
\begin{align}
&|\beta^n_2| \leq \epsilon \tau \mu \|\delta_\tau \bm{e}_h^{n+1}\|^2 + C\left(1+\frac{1}{\epsilon\mu}\right)\tau \|\delta_\tau \bm{U}_h^{n+1} \|_{\DG}^2 
+   \tau \|\delta_\tau \bm{e}_h^n - \delta_\tau \bm{e}_h^{n+1}\|^2 \nonumber  \\ &  + C (h^{2k+2} + \|\bm{e}_h^{n}\|^2)\int_{t^n}^{t^{n+1}} \vert \partial_t \bm{u} \vert_{H^2(\Omega)}^2 
+ C  h^{2k+2} \int_{t^{n-1}}^{t^{n}} \vert \partial_t \bm{u} \vert_{H^{k+1}(\Omega)}^2. 
\end{align}
Taking the sum of $\beta_2^n$ and using the results of Lemma \ref{lemma:stability_time_derivative} and Lemma \ref{lemma:first_err_estimate}, we obtain (recall that the constant $C_\mu$ is a generic constant that depends on $e^\mu$)
\begin{align}
\sum_{n=1}^m |\beta_2^n| & \leq \epsilon\tau \mu \sum_{n=1}^m \| \delta_\tau \bm{e}_h^{n+1}\|^2 + C\left(1+\frac{1}{\epsilon\mu}\right)\tau \sum_{n=1}^m \| \delta_\tau \bm{U}_h^{n+1} \|^2_{\DG} \nonumber  \\ & \quad + C_\mu (1+\mu^{-1} + \mu)(\tau + h^{2k}) + C_{\gamma,\mu} \tau. \label{eq:bound_beta2} 
\end{align}
To handle $\beta^n_3$, 
recall that with \eqref{eq:splitting_technique}, we have: 
% \begin{align*}
% a_\mathcal{C}(\bm{u}_h^n;\bm{u}_h^n, \bm{u}^{n+1}- \bm{v}_h^{n+1}, \delta_\tau \bm{U}_h^{n+1})&= \mathcal{C}(\bm{u}_h^n, \bm{u}^{n+1}- \bm{v}_h^{n+1}, \delta_\tau \bm{U}_h^{n+1}) - \mathcal{U}(\bm{u}_h^n;\bm{u}_h^n, \bm{u}^{n+1}- \bm{v}_h^{n+1}, \delta_\tau \bm{U}_h^{n+1}),\\ 
% & = \mathcal{C}(\bm{u}_h^n - \bm{u}_h^{n-1}, \bm{u}^{n+1}- \bm{v}_h^{n+1}, \delta_\tau \bm{U}_h^{n+1}) + \mathcal{C}(\bm{u}_h^{n-1}, \bm{u}^{n+1}- \bm{v}_h^{n+1}, \delta_\tau \bm{U}_h^{n+1}) \\ &   - ( \mathcal{U}(\bm{u}_h^n;\bm{u}_h^n, \bm{u}^{n+1}- \bm{v}_h^{n+1}, \delta_\tau \bm{U}_h^{n+1})  -  \mathcal{U}(\bm{u}_h^{n-1} ;\bm{u}_h^n, \bm{u}^{n+1}- \bm{v}_h^{n+1}, \delta_\tau \bm{U}_h^{n+1}) ) \\ & - \mathcal{U}(\bm{u}_h^{n-1};\bm{u}_h^n - \bm{u}_h^{n-1}, \bm{u}^{n+1}- \bm{v}_h^{n+1}, \delta_\tau \bm{U}_h^{n+1})\\ & - \mathcal{U}(\bm{u}_h^{n-1};\bm{u}_h^{n-1}, \bm{u}^{n+1}- \bm{v}_h^{n+1}, \delta_\tau \bm{U}_h^{n+1}), \\ 
% & = a_{\mathcal{C}} (\bm{u}_h^{n-1};\bm{u}_h^{n-1}, \bm{u}^{n+1}- \bm{v}_h^{n+1}, \delta_\tau \bm{U}_h^{n+1}) + I_2 \\ & \quad  + \mathcal{C}(\bm{u}_h^n - \bm{u}_h^{n-1}, \bm{u}^{n+1}- \bm{v}_h^{n+1}, \delta_\tau \bm{U}_h^{n+1}) \\ & \quad  -  \mathcal{U}(\bm{u}_h^{n-1}, \bm{u}_h^n - \bm{u}_h^{n-1}, \bm{u}^{n+1}- \bm{v}_h^{n+1}, \delta_\tau \bm{U}_h^{n+1}),  
% \end{align*}
%Hence, $\beta_3$ reads: 
\begin{align*}
    & \beta^n_3  = a_\mathcal{C}(\bm{u}_h^{n-1}; \bm{u}_h^{n-1}, \bm{u}^{n+1} - \bm{v}_h^{n+1} - (\bm{u}^n - \bm{v}_h^n),\delta_\tau \bm{U}_h^{n+1}) + I^n_2 \\ &  + \mathcal{C}(\bm{u}_h^n - \bm{u}_h^{n-1}, \bm{u}^{n+1}- \bm{v}_h^{n+1}, \delta_\tau \bm{U}_h^{n+1}) - \mathcal{U}(\bm{u}_h^{n-1}; \bm{u}_h^n - \bm{u}_h^{n-1}, \bm{u}^{n+1}- \bm{v}_h^{n+1}, \delta_\tau \bm{U}_h^{n+1}) \\ 
& = I^n_1 + I^n_2 +I^n_3 + I^n_4,      
\end{align*}
where 
$$ 
I^n_2 = -\mathcal{U}(\bm{u}_h^n;\bm{u}_h^n, \bm{u}^{n+1}- \bm{v}_h^{n+1}, \delta_\tau \bm{U}_h^{n+1})  +  \mathcal{U}(\bm{u}_h^{n-1} ;\bm{u}_h^n, \bm{u}^{n+1}- \bm{v}_h^{n+1}, \delta_\tau \bm{U}_h^{n+1}). 
$$ 
We have the following bound on $I^n_2$, see Lemma \ref{lemma:bound_I1} in the Appendix.  
\begin{align*}
|I^n_2| &\leq C  \|\bm{u}_h^n - \bm{u}_h^{n-1}\|\| \delta_\tau \bm{U}_h^{n+1} \|_{L^{\infty}(\Omega)} \| \bm{u}^{n+1} - \bm{v}_h^{n+1}\|_{\DG}.
\end{align*}
The terms $I^n_3 + I^n_4$ are handled as follows. Note that $\bm{u}^{n+1} - \bm{v}_h^{n+1} = \bm{u}^{n+1} - \Pi_h \bm{u}^{n+1} - \errn$. Hence, we apply Lemma \ref{lemma:bd_eta_1_eta_2_semi} estimate \eqref{eq:appendix_lemma3_secondestimate} in the appendix: 
\begin{align}
 &  |I^n_3| + |I^n_4| \leq C \|\bm{u}_h^n - \bm{u}_h^{n-1} \| (\| \delta_\tau \bm{U}_h^{n+1}\|_{L^{\infty}(\Omega)} + \| \nabla_h \delta_\tau \bm{U}_h^{n+1} \|_{L^3(\Omega)})\|\bm{u}^{n+1} - \bm{v}_h^{n+1} \|_{\DG} \nonumber \\ &  + C\| \bm{u}^n_h- \bm{u}_h^{n-1}\|_{L^3(\Omega)}\left(\sum_{e\in\Gamma_h} \sigma h_e^{-1} \|[\delta_\tau \bm{U}_h^{n+1} - \delta_\tau \bm{U}^{n+1}] \|_{L^2(e)}^2 \right)^{1/2} \|\errn \|_{\DG}\nonumber \\ 
&  + C \| \bm{u}_h^{n} - \bm{u}_h^{n-1}\|\|\delta_\tau \bm{U}_h^{n+1}\|_{L^{\infty}(\Omega)} (h^{-1} \| \bm{u}^{n+1} - \Pi_h \bm{u}^{n+1}\| + \|\nabla_h(\bm{u}^{n+1} - \Pi_h \bm{u}^{n+1})\|).\nonumber
\end{align} 
Similar to Lemma \ref{lemma:regularity_bounds_dual_problem}, by the linearity of the PDE \eqref{eq:aux_pb_1}-\eqref{eq:aux_pb_3}, we have: 
\begin{align}
   \vertiii{\delta_\tau \bm{U}^{n+1}} +  \| \delta_\tau \bm{U}_h^{n+1}\|_{L^{\infty}(\Omega)}  &\leq C \| \delta_\tau \bm{\chi}^{n+1} \|,\label{eq:regularity_delta} \\ 
\| \delta_\tau \bm{U}_h^{n+1} -\delta_\tau \bm{U}^{n+1} \|_{\DG} &\leq C h \|\delta_\tau \bm{\chi}^{n+1} \| \label{eq:conv_delta_dual}. 
\end{align}
With the above  bounds, triangle inequality, inverse estimate \eqref{eq:inverse_estimate}, and \eqref{eq:approximation_prop_2}, we find 
\begin{align}
\| \nabla_h \delta_\tau&  \bm{U}_h^{n+1}\|_{L^3(\Omega)}  \leq  Ch^{-d/6}\| \nabla_h(\delta_\tau \bm{U}_h^{n+1} - \Pi_h (\delta_\tau \bm{U}^{n+1})) \| \nonumber \\ 
& + \|\nabla_h (\Pi_h (\delta_\tau \bm{U}^{n+1}))\|_{L^3(\Omega)} \leq C h^{1-d/6} \| \delta_\tau \bm{\chi}^{n+1}\| + C \|\nabla \delta_\tau \bm{U}^{n+1}\|_{L^3(\Omega)} \nonumber \\ 
& \quad  \leq C \| \delta_\tau \bm{\chi}^{n+1}\|.\label{eq:l3_bd_nabla_delta_Uh}
\end{align}
Hence, with \eqref{eq:inverse_estimate}, the triangle inequality and  \eqref{eq:approximation_prop_1}-\eqref{eq:approximation_prop_2}, we obtain 
\begin{align}
   \sum_{i=2}^4 |I^n_i| %\leq C \| \bm{u}_h^n - \bm{u}_h^{n-1}\|\| \delta_\tau \bm{\chi}^{n+1}\|( h^{k} |\bm{u}^{n+1}|_{H^{k+1}(\Omega)} + \|\errn\|_{\DG} ) \nonumber  \\ 
   &\leq C\tau ( \|\delta_\tau \bm{e}_h^n\| + \|\delta_\tau \Pi_h \bm{u}^n \|)\| \delta_\tau \bm{\chi}^{n+1}\| ( h^{k} |\bm{u}^{n+1}|_{H^{k+1}(\Omega)} + \|\errn\|_{\DG} ). \nonumber
\end{align}
With a Taylor's expansion,  \eqref{eq:approximation_prop_1},  \eqref{eq:cfl_cond_detla}, and the regularity assumptions we have    
\begin{align*} \| \delta_\tau \Pi_h \bm{u}^n \|^2 \leq  2 \|\delta_\tau (\Pi_h \bm{u}^n - \bm{u}^n)\|^2+ 2 \|\delta_\tau \bm{u}^n\|^2 
    % \leq C\tau^{-1} h^{2k+2} \int_{t^{n-1}}^{t^{n}}\vert\partial_t \bm{u}\vert_{H^{k+1}(\Omega)}^2 +  4\|(\partial_t\bm{u})^{n-1}\|^2 + \frac{4}{3}\tau \int_{t^{n-1}}^{ t^n} \|\partial_{tt} \bm{u}\|^2 \\ 
 \leq  Ch^{2k} \int_{t^{n-1}}^{t^n} \vert\partial_t \bm{u}\vert_{H^{k+1}(\Omega)}^2 \\ + 4\| \partial_t\bm{u} \|^2_{L^{\infty}(0,T;L^2(\Omega))} + \frac{4}{3}\tau \| \partial_{tt} \bm{u}\|^2_{L^2(0,T;L^2(\Omega))}  \leq C. 
\end{align*}
We use the following bound, which is easy to obtain from the definition of $\bm{\chi}$.
\begin{equation}
\Vert \delta_\tau \bm{\chi}^{n+1} \Vert \leq \Vert \delta_\tau\bm{e}_h^{n+1}\Vert + C h^{k+1} \vert \delta_\tau \bm{u}^{n+1}\vert_{H^{k+1}(\Omega)}.
\label{eq:bounddeltatauchi}
\end{equation}
Hence, with Young's inequality and the assumption that $\bm{u} \in L^{\infty}(0,T;H^{k+1}(\Omega)^d)$:
\begin{multline}
\sum_{i=2}^4 |I^n_i|  \leq  \epsilon \tau \mu \|\delta_\tau \bm{e}_h^{n+1} \|^2 + C \left(1 + \frac{1}{\epsilon\mu}\right)\tau( \|\delta_\tau \bm{e}_h^n\|^2 + 1)(\|\errn\|^2_{\DG} + h^{2k} ) \\ + C \tau h^{2k+2} |\delta_\tau \bm{u}^{n+1}|^2_{H^{k+1}(\Omega)}. 
\label{eq:boundI2I3I4}
\end{multline} 
It remains to handle $I^n_1$. It is helpful to define the function $\bm{E}^n \in \bm{X}$, $\bm{E}^n = \Pi_h \bm{u}^n - \bm{u}^n$ for $n \geq 1$.  
We have from \eqref{eq:stability_linf_pih}, \eqref{eq:approximation_prop_1}, \eqref{eq:approximation_prop_2}: 
\begin{align}\label{eq:boundEn}
 & \| \delta_\tau \bm{E}^{n+1} \|^2  + h^2 \|\delta_\tau \bm{E}^{n+1} \|^2_{\DG}  \leq C h^{2k + 2} \vert \delta_\tau \bm{u}^{n+1} \vert^2_{H^{k+1}(\Omega)}, \\% \leq C \tau^{-1} h^{2k+2} \int_{t^n}^{t^{n+1}} \vert \partial_t \bm{u}\vert_{H^{k+1}(\Omega)}^2, \\ 
& \| \delta_\tau \bm{E}^{n+1} \|_{L^\infty(\Omega)}^2  \leq C  \vertiii{ \delta_\tau \bm{u}^{n+1}}^2 \leq C \| \delta_\tau \bm{u}^{n+1} \|_{H^2(\Omega)}^2. % \leq C \tau^{-1} \int_{t^n}^{t^{n+1}} \| \partial_t  \bm{u}\|^2_{H^2(\Omega)}. 
\label{eq:boundEngrad}
\end{align}
With the definition of $\bm{E}^n$ and \eqref{eq:integration_by_parts_c}, 
\begin{align*}
    I^n_1  & = \tau \bar{a}_\mathcal{C}(\bm{u}_h^{n-1}; \bm{u}_h^{n-1}, \delta_\tau \bm{U}_h^{n+1}, \delta_\tau \bm{E}^{n+1} + \delta_\tau \errn)
    \\ &= \tau  \bar{a}_\mathcal{C}(\bm{u}_h^{n-1}; \bm{u}_h^{n-1}, \delta_\tau \bm{U}_h^{n+1} - \delta_\tau \bm{U}^{n+1}, \delta_\tau \bm{E}^{n+1} +  \delta_\tau \errn) \nonumber \\ & \quad  +  \tau \bar{a}_\mathcal{C}(\bm{u}_h^{n-1}; \bm{u}_h^{n-1} - \bm{u}^{n-1}, \delta_\tau \bm{U}^{n+1}, \delta_\tau \bm{E}^{n+1} + \delta_\tau \errn)\nonumber \\ & \quad   +  \tau \bar{a}_\mathcal{C}(\bm{u}_h^{n-1}; \bm{u}^{n-1}, \delta_\tau \bm{U}^{n+1}, \delta_\tau \bm{E}^{n+1} +  \delta_\tau \errn) = \tau I^n_{1,1} +\tau I^n_{1,2} + \tau  I^n_{1,3}. 
\end{align*}
To handle $I^n_{1,1}$, we closely follow the derivation of the bound on $A_1$, see the derivation of bound \eqref{eq:bd_A_1} in the proof of Theorem \ref{theorem:improved_estimate_velocity}. We provide the details in the Appendix: Lemma \ref{lemma:bound_on_I11}. %\textit{For our reference, I include a proof for the bound below in the Appendix: Lemma  \ref{lemma:bound_on_I11}}.  
%
% \begin{align}
%     &|I^n_{1,1}|  \leq C  \|\bm{e}_h^{n-1}\|\| \delta_\tau \bm{\chi}^{n+1} \|( h^k \vert \delta_\tau \bm{u}^{n+1}\vert_{H^{k+1}(\Omega)}  + \| \delta_\tau \errn \|_{\DG}) \nonumber \\
% & + C h \| \delta_\tau \bm{\chi}^{n+1}\|( h^{k+1} \vert \delta_\tau \bm{u}^{n+1}\vert_{H^{k+1}(\Omega) }   
%  +  \| \delta_\tau \bm{e}_h^{n+1}\|) \nonumber \\ 
% & +C h\| \delta_\tau \bm{\chi}^{n+1}\| (\| \delta_\tau \errn\|_{\DG} + (1+ 1/\mu)^{1/2}
% \Vert \delta_\tau\bm{u}^{n+1}\Vert_{H^2(\Omega)} + \tau |\delta_\tau \phi_h^{n+1}|_{\DG})  \nonumber \\&+ C \|\delta_\tau \bm{U}_h^{n+1}\|_{\DG}(\|\bm{e}_h^{n-1}\| \Vert \delta_\tau\bm{u}^{n+1}\Vert_{H^2(\Omega)}  + \| \delta_\tau \bm{e}_h^{n+1}\| + \tau |\delta_\tau \phi_h^{n+1}|_{\DG} + h^{k+1}| \delta_\tau \bm{u}^{n+1}|_{H^{k+1}(\Omega)})\nonumber 
% \end{align}
%With Young's and triangle inequalities, we obtain
\begin{align}
& |I^n_{1,1}|  \leq \epsilon  \mu  \|\delta_\tau \bm{e}_h^{n+1} \|^2 
\nonumber \\ 
&+C\left(1+\frac{1}{\epsilon\mu}\right) (\|\bm{e}_h^{n-1} \|^2+h^2)( h^{2k} \vert \delta_\tau \bm{u}^{n+1}\vert_{H^{k+1}(\Omega)}^2  + \| \delta_\tau \errn \|^2_{\DG}) \nonumber \\ 
& + C\left(\frac{1}{\epsilon\mu} +1 \right)h^2 \| \delta_\tau \bm{e}_h^{n+1}\|^2  
+  C \left(1+\frac{1}{\epsilon\mu}\right)\|\delta_\tau \bm{U}_h^{n+1} \|^2_{\DG} \nonumber \\ 
&
+ C\tau^2 \left(1+\frac{h^2}{\epsilon\mu}\right) \vert \delta_\tau \phi_h^{n+1} \vert^2_{\DG} \nonumber \\ 
& + C \left(1+ \frac{1}{\epsilon \mu}\right) \left( h^2  \left(1+\frac{1}{\mu}\right)  + \| \bm{e}_h^{n-1}\|^2\right) \| \delta_\tau \bm{u}^{n+1}\|_{H^2(\Omega)}^2.
\label{eq:boundI11} 
\end{align}
For $I^n_{1,2}$, we have: 
\begin{align*}
I^n_{1,2} =  \sum_{E \in \mesh_h} \int_E ((\bm{u}_h^{n-1} - \bm{u}^{n-1})\cdot \nabla \delta_\tau \bm{U}^{n+1}) \cdot(\delta_\tau \bm{E}^{n+1}+ \delta_\tau \errn) \\  +  \frac{1}{2}b(\bm{e}_h^{n-1} + (\Pi_h \bm{u}^{n-1} - \bm{u}^{n-1}), \delta_\tau \bm{U}^{n+1}\cdot (\delta_\tau \bm{E}^{n+1}
+ \delta_\tau \errn)).
\end{align*}
% To bound the first term, we note that similar arguments to Lemma \ref{lemma:regularity_bounds_dual_problem} yield: 
% \begin{equation}
% \|\nabla \delta_{\tau} \bm{U}^{n+1}\|_{L^3(\Omega)} \leq C\| \delta_\tau \bm{e}_h^{n+1}\|. 
% \end{equation}
With \eqref{eq:regularity_delta},  \eqref{eq:boundEn}, Holder's inequality, triangle inequality, approximation property \eqref{eq:approximation_prop_1},  and \eqref{eq:discreter_poincare}, we bound the first term by: 
\begin{multline}
     \|\bm{u}_h^{n-1} - \bm{u}^{n-1}\| \|\nabla \delta_{\tau} \bm{U}^{n+1}\|_{L^3(\Omega)} \|\delta_\tau \bm{E}^{n+1}+ \delta_\tau \errn\|_{L^6(\Omega)} \\ \leq C (\|\bm{e}_h^{n-1}\| + h^2| \bm{u}^{n-1}|_{H^2(\Omega)}) \| \delta_\tau \bm{\chi}^{n+1}\|\|\delta_\tau \bm{E}^{n+1}+ \delta_\tau \errn\|_{\DG}
\\
\leq C(\|\bm{e}_h^{n-1}\| + h^2| \bm{u}^{n-1}|_{H^2(\Omega)}) \| \delta_\tau \bm{\chi}^{n+1}\|
( h^k \vert \delta_\tau \bm{u}^{n+1} \vert_{H^{k+1}(\Omega)} + \Vert \delta_\tau \errn\|_{\DG}). \nonumber 
\end{multline} 
To bound the second term in $I_{1,2}^n$, we refer to the derivation of bounds \eqref{eq:bounding_b_eh_pih_u} and \eqref{eq:first_estimate_nonlinear_form}.  We use \eqref{eq:regularity_delta},  \eqref{eq:boundEn} and  obtain:  
\begin{multline*}
|b(\bm{e}_h^{n-1}, \delta_\tau \bm{U}^{n+1}\cdot (\delta_\tau \bm{E}^{n+1}+\delta_\tau \errn))| %\leq |b(\bm{e}_h^{n-1}, \delta_\tau \bm{U}^{n+1}\cdot \delta_\tau \bm{E}^{n+1}) |  \\ & +  |b(\bm{e}_h^{n-1}, \delta_\tau \bm{U}^{n+1}\cdot \delta_\tau \errn)|
 \\ \leq C\|\bm{e}_h^{n-1}\| \| \delta_\tau \bm{\chi}^{n+1}\| (h^{k}|\delta_\tau \bm{u}^{n+1}|_{H^{k+1}(\Omega)} + \|\delta_\tau \errn\|_{\DG}).
\end{multline*} 
Bounds \eqref{eq:third_estimate_nonlinear_term} and \eqref{eq:regularity_delta} yield: 
\begin{equation}
 |b(\Pi_h \bm{u}^{n-1} - \bm{u}^{n-1}, \delta_\tau \bm{U}^{n+1} \cdot \delta_\tau \errn) | \leq Ch^{k+1}|\bm{u}^{n-1}|_{H^{k+1}(\Omega)}\| \delta_\tau \bm{\chi}^{n+1}\|\| \delta_\tau \errn \|_{\DG}.\nonumber
\end{equation}
Approximation properties, trace estimates \eqref{eq:trace_ineq_continuous}-\eqref{eq:trace_ineq_discrete},  and \eqref{eq:regularity_delta} yield:  
\begin{align*}
& |b(\Pi_h \bm{u}^{n-1} - \bm{u}^{n-1},  \delta_\tau \bm{U}^{n+1} \cdot \delta_\tau \bm{E}^{n+1} )| \\ &  
\leq C  \| \delta_\tau \bm{U}^{n+1}\|_{L^{\infty}(\Omega)}( h^{k} \vert \bm{u}^{n-1}\vert_{H^{k+1}(\Omega)}   \| \delta_\tau \bm{E}^{n+1}\| \nonumber  +  h^{k+2}|\bm{u}^{n-1}|_{H^{k+1}(\Omega)} |\delta_\tau \bm{u}
^{n+1} |_{H^2(\Omega)} )\nonumber \\ &  
   \leq C  h^{k+2} |\delta_\tau \bm{u}^{n+1}|_{H^{2}(\Omega)}| \bm{u}^{n-1}|_{H^{k+1}(\Omega)}\|\delta_\tau \bm{\chi}^{n+1}\|. \nonumber
 \end{align*}
 Hence, Young's inequality and the assumption that $\bm{u} \in L^{\infty}(0,T;H^{k+1}(\Omega)^d )$ yield  
 \begin{align}
&|I^n_{1,2}| \leq \epsilon  \mu \| \delta_\tau \bm{e}_h^{n+1} \|^2 \nonumber\\ & + C\left(\frac{1}{\epsilon\mu} +1 \right)(\| \bm{e}_h^{n-1}\|^2 + h^4)
( h^{2k} \vert \delta_\tau \bm{u}^{n+1}\vert_{H^{k+1}(\Omega)}^2 + \Vert  \delta_\tau \errn \|_{\DG}^2)
%&+ \frac{C}{\epsilon\mu}h^{2k}| \delta_\tau \bm{u}^{n+1} |_{H^{k+1}(\Omega)}^2 \|\bm{e}_h^{n-1}\|^2 
%+\frac{C}{\epsilon\mu} \|\bm{e}_h^{n-1}\|^2  \,  \|\delta_\tau \errn \|_{\DG}^2
\nonumber\\ 
 & +   C h^{2k+2} |\delta_\tau \bm{u}^{n+1} |_{H^{k+1}(\Omega)}^2.\label{eq:boundI12}
 \end{align}
To handle $I^n_{1,3}$, we have: 
\begin{align}
    I^n_{1,3} & =  \sum_{E\in\mesh_h} \int_E (\bm{u}^{n-1} \cdot \nabla ((\delta_\tau \bm{U}^{n+1} - \delta_\tau \bm{U}_h^{n+1}) +\delta_\tau \bm{U}_h^{n+1} )\cdot (\delta_\tau \bm{E}^{n+1} + \delta_\tau \errn ). \nonumber %\\ & \quad + \sum_{E\in\mesh_h} \int_E (\bm{u}^{n-1} \cdot \nabla  \delta_\tau \bm{U}_h^{n+1})\cdot (\delta_\tau \bm{E}^{n+1} + \delta_\tau \errn ).\nonumber 
\end{align}
Note that by the triangle inequality and \eqref{eq:update_velocity_delta}, we have: 
\begin{align}
\|\delta_\tau \errn\| \leq \| \delta_\tau \errn - \delta_\tau \bm{e}_h^{n+1}\| + \| \delta_\tau \bm{e}_h^{n+1}\| \leq C\tau |\delta_{\tau}\phi_h^{n+1}|_{\DG} + \| \delta_\tau \bm{e}_h^{n+1}\|. \label{eq:bounddeltataueh} 
\end{align} 
Hence, with the help of \eqref{eq:conv_delta_dual},  \eqref{eq:boundEn} and Holder's inequality, we obtain
\begin{align}
 & |I^n_{1,3} | \leq  C \| \bm{u}^{n-1}\|_{L^{\infty}(\Omega)}(h\|\delta_\tau \bm{\chi}^{n+1} \|+ \| \delta_\tau \bm{U}_h^{n+1}\|_{\DG}) \nonumber \\ & \times (\|\delta_\tau \bm{E}^{n+1}\| + \| \delta_\tau \bm{e}_h^{n+1}\|  +  \tau | \delta_\tau \phi_h^{n+1}|_{\DG}) \nonumber %\\ & + C \|\bm{u}^{n-1}\|_{L^{\infty}(\Omega)}\| \delta_\tau \bm{U}_h^{n+1}\|_{\DG}(\|\delta_\tau \bm{E}^{n+1}\| + \| \delta_\tau \bm{e}_h^{n+1}\| + \tau  | \delta_\tau \phi_h^{n+1}|_{\DG}),  \nonumber \\ 
\\ & \leq \epsilon  \mu \|\delta_\tau \bm{e}_h^{n+1} \|^2+ C\left( 1+\frac{h^{2}}{\epsilon \mu} \right) (h^{2k+2}|\delta_\tau \bm{u}^{n+1}|_{H^{k+1}(\Omega)}^2+ \tau^2 |\delta_\tau \phi_h^{n+1}|_{\DG}^2) \nonumber \\ & \quad  +\frac{C}{\epsilon \mu} h^2 \|\delta_\tau \bm{e}_h^{n+1}\|^2  + C\left(1+\frac{1}{\epsilon \mu}\right) \|\delta_\tau \bm{U}_h^{n+1}\|^2_{\DG}.  \label{eq:boundI13} 
\end{align}
With \eqref{eq:boundI2I3I4}, \eqref{eq:boundI11}, \eqref{eq:boundI12}, and \eqref{eq:boundI13}, 
we obtain a bound on $\beta_3^n$. 
 \begin{align}
 & |\beta_3^n|  \leq 4\epsilon  \mu  \|\delta_\tau \bm{e}_h^{n+1} \|^2 \nonumber \\ 
 & 
 +C\left(1+\frac{1}{\epsilon\mu}\right) \tau  (\|\bm{e}_h^{n-1} \|^2+h^2)( h^{2k} \vert \delta_\tau \bm{u}^{n+1}\vert_{H^{k+1}(\Omega)}^2  + \| \delta_\tau \errn \|^2_{\DG}) \nonumber \\ 
 & + C\left(1 + \frac{1}{\epsilon\mu} \right) \tau h^2 \| \delta_\tau \bm{e}_h^{n+1}\|^2  
 +  C \left(1+\frac{1}{\epsilon\mu}\right) \tau \|\delta_\tau \bm{U}_h^{n+1} \|^2_{\DG} \nonumber \\ 
 &  + C \left(1+\frac{h^2}{\epsilon\mu}\right) \tau^3 \vert \delta_\tau \phi_h^{n+1} \vert^2_{\DG}+ C \left( 1+\frac{h^{2}}{\epsilon \mu} \right)\tau h^{2k+ 2}|\delta_\tau \bm{u}^{n+1}|_{H^{k+1}(\Omega)}^2  \nonumber \\ 
& + C \left(1+ \frac{1}{\epsilon \mu}\right) \tau  \left( h^2  \left(1+\frac{1}{\mu}\right)  + \| \bm{e}_h^{n-1}\|^2\right) \| \delta_\tau \bm{u}^{n+1}\|_{H^2(\Omega)}^2 
\nonumber \\ 
%&  + C \left( 1+\frac{h^{2}}{\epsilon \mu} \right)\tau h^{2k+ 2}|\delta_\tau \bm{u}^{n+1}|_{H^{k+1}(\Omega)}^2 \nonumber \\ 
& +C \left(1 + \frac{1}{\epsilon\mu}\right)\tau( \|\delta_\tau \bm{e}_h^n\|^2 + 1)(\|\errn\|^2_{\DG} + h^{2k} ). 
\end{align}
We take the sum of $\beta_3^n$, choose $\epsilon = 1/18$, and use the results of Lemma \ref{lemma:stability_time_derivative}, Lemma \ref{lemma:first_err_estimate}, and the condition on $\tau$ \eqref{eq:cfl_cond_detla}. We obtain
\begin{align} 
\sum_{n=1}^m |\beta_3^n| &\leq \frac{2}{9} \tau \mu \sum_{n=1}^{m} \|\delta_\tau \bm{e}_h^{n+1} \|^2 \nonumber  + C\left(1 + \frac{1}{\mu}\right) \tau \sum_{n=1}^{m} \| \delta_\tau \bm{U}_h^{n+1}\|^2_{\DG} \\ & \quad  +(C_\mu K_\mu + C_{\gamma,\mu}K_\mu (1+ C_\mu) )(\tau + h^{2k}) .\label{eq:bound_beta_3}
\end{align}
Further, observe that the sum of the last two terms in \eqref{eq:error_delta_semi_final} yields  
\begin{align}
& \frac{1}{\tau}(\bm{G}^{m+1}, \delta_\tau \bm{U}_h^{m+1})-\frac{1}{\tau} (\bm{G}^1, \delta_\tau \bm{U}_h^{1})  \nonumber \\ &\leq 
 \frac{1}{8}   ( \| \delta_\tau \bm{U}_h^{m+1} \|^2_{\DG}+ \| \delta_\tau \bm{U}_h^{1} \|_{\DG}^2)  
+ C\tau \int_{t^m}^{t^{m+1}} \|\partial_{tt} \bm{u}\|^2 + C\tau \int_{t^0}^{t^{1}} \|\partial_{tt} \bm{u}\|^2.
\label{eq:boundGmG1}
\end{align}
We sum \eqref{eq:error_delta_semi_final} from $n = 1$ to $n=m$  and use Lemma~\ref{lemma:stability_time_derivative}, Lemma~\ref{lemma:first_err_estimate}, the bounds 
\eqref{eq:bound_beta1}, \eqref{eq:bound_beta2}, \eqref{eq:bound_beta_3}, \eqref{eq:boundGmG1}, the coercivity and continuity of $a_\mathcal{D}$, 
the condition \eqref{eq:cfl_cond_detla}.   We obtain
\begin{align}
 \frac{1}{8}  \|\delta_\tau \bm{U}_h^{m+1}\|_{\DG}^2 
+\frac{\mu}{2} \tau \sum_{n=1}^{m}\| \delta_\tau \bm{e}_h^{n+1}  \|^2  \leq  (C_\mu K_\mu + C_{\gamma,\mu}K_\mu (1+ C_\mu) )(\tau + h^{2k}) \nonumber \\ +C \|\delta_\tau \bm{U}_h^1 \|_{\DG}^2 +  C\left(1+\frac{1}{\mu}\right) \tau \sum_{n=1}^{m} \|\delta_\tau \bm{U}_h^{n+1}\|^2_{\DG}.  \nonumber 
\end{align}
To handle $\| \delta_\tau \bm{U}_h^1\|_{\DG}^2$,
we take $n=1$ in \eqref{eq:first_error_eq_dual} and use the same expressions and bounds as in the proof of Theorem \ref{theorem:improved_estimate_velocity} except for  bound  
\eqref{eq:bounding_time_err_dual} and the bound on $T_{2,1}^n$ \eqref{eq:bd_T_21}. Instead, we use the following alternative bounds which result from a different application of Young's inequality:
\begin{align*}
 |T_{2,1}^1| &\leq C \tau^2 \int_{t^{0}}^{t^{1}} \|\partial_t \bm{u}\|^2 + \frac{1}{16} \tau^{-1} \|\bm{U}_h^1\|^2_{\DG}, \\
 |(\tau (\partial_t \bm{u})^1 - (\bm{u}^1 - \bm{u}^{0}),  \bm{U}_h^1)| & \leq C\tau^3\int_{t^{0}}^{t^1}  \| \partial_{tt}  \bm{u}\|^2 +  \frac{1}{16}\|\bm{U}_h^1\|_{\DG}^2. 
\end{align*}
We use \eqref{eq:cfl_cond_detla}, \eqref{eq:initial_error_eh} and its derivation, and the observation that $\bm{U}_h^0 = \bm{0}$, see \eqref{eq:UH0_ZERO}. With assuming that $C\tau(1+1/\mu) \leq 1/16 $, we have: 
 \begin{align*}
 &\frac{1}{16}\|\bm{U}_h^1\|^2_{\DG}\\  &\leq C\left(\frac{1}{\mu}+1+\mu\right)\tau\left(h^{2k+2} + \tau^2|\phi_h^1|^2_{\DG} + \|\bm{e}_h^1 - \bm{e}_h^0 \|^2 +h^2 ( \|\tilde{\bm{e}}_h^1\|^2_{\DG} + \|\bm{e}_h^1\|^2 ) \right)  \\ & \quad + C\left(\frac{1}{\mu}+1\right)\tau \|\bm{U}_h^1\|_{\DG}^2 \leq  C\left(\frac{1}{\mu}+1+\mu\right)^3 \tau^3. 
 \end{align*}
 Hence $\|\delta_\tau \bm{U}_h^1\|_{\DG}^2 \leq CK_\mu \tau$. With $\tau$ small enough, say $C\tau(1+1/\mu)<1$, the final result is obtained by Gronwall's and triangle inequalities. 
\end{proof} 

    \section{Error estimate for the pressure }\label{sec:pressure_estimate}
With the results of the previous section, we show the following a priori estimate for the pressure, optimal in space and sub-optimal in time.
\begin{theorem} \label{theorem:pressure_error}
Under the same assumptions of Lemma  \ref{lemma:error_delta_velocity}, we have the following error estimate. 
There exists a constant $\hat{C}_{\gamma,\mu}$ independent of $h$ and $\tau$ but inversely dependent on $\gamma$ and $\mu$ such that, for $1 \leq m \leq N_T$:  
\begin{equation}
    \tau \sum_{n=1}^{m} \| p^n - p_h^n\|^2 \leq \hat{C}_{\gamma,\mu} (\tau + h^{2k}). 
\end{equation}  
Here, $\hat{C}_{\gamma,\mu} = C \mu^{-1} \tilde{C}_{\gamma,\mu} + C_\mu K_\mu$ where $C_\mu$ depends on $e^{\frac{1}{\mu}}$.
\end{theorem}
%{\blue question: from the proof, is (189) only hold for $n$ starting from $1$? in addition, it seems counter intuitive in case of $||p^0||$ can be controled? maybe I missing something..}
 \begin{proof}
Let $\tilde{\bm{X}}_h$ be a subspace of $\bm{X}_h$ satisfying 
\begin{equation}
\tilde{\bm{X}}_h = \{ \bm{\theta}_h \in \bm{X}_h:~ \forall e\in \Gamma_h \cup \partial \Omega,~ [\bm{\theta}_h ]\cdot \bm{n}_e = 0\}. 
\end{equation} 
There exists a  positive constant $\beta^{\ast}$,  independent of $h$, such that \cite{girault2005discontinuous,riviere2008discontinuous} 
\begin{equation}\label{eq:inf_sup_cond}
\inf_{q_h\in M_h} \sup_{\bm{v}_h \in \tilde{\bm{X}}_h} \frac{-b(\bm{v}_h,q_h)}{\|\bm{v}_h\|_{\DG} \| q_h\|} \geq \beta^{\ast}. 
\end{equation} 
%{\blue (maybe it is better to use another symbol $\beta^\ast$ since $\beta$ is already used in lemma 7)}
As a result of this inf-sup condition, for $n\geq 0$, there exists a unique $\bm{\zeta}^n_h \in \tilde{\bm{X}}_h$ with
\begin{equation}
b(\bm{\zeta}^n_h, p_h^n - \pi_h p^n ) = \|p_h^n - \pi_h p^n\|^2, \quad \|\bm{\zeta}^n_h \|_{\DG} \leq \frac{1}{\beta^{\ast}} \|p_h^n - \pi_h p^n\|.  \label{eq:results_inf_sup}
\end{equation}
We substitute \eqref{eq:sec_err_eq} in \eqref{eq:first_err_eq}. For $\bm{\theta}_h \in \bm{X}_h$, 
\begin{multline}
(\bm{e}_h^{n} - \bm{e}_h^{n-1}, \bm{\theta}_h ) + \tau a_{\mathcal{C}}(\bm{u}_h^{n-1}; \bm{u}_h^{n-1}, \err , \bm{\theta}_h) + \tau R_{\mathcal{C}}(\bm{\theta}_h)\\ + \tau \mu a_\mathcal{D}(\err, \bm{\theta}_h)  =  \tau b(\bm{\theta}_h, p_h^{n-1} + \phi_h^n - p^n) - \tau \mu a_\mathcal{D}(\Pi_h \bm{u}^n - \bm{u}^n ,\bm{\theta}_h) +  R_t(\bm{\theta}_h). \nonumber
\end{multline}
Using \eqref{eq:error_update_pressure}, we obtain for all $\bm{\theta}_h \in \bm{X}_h$,  
\begin{align}
&(\bm{e}_h^{n} - \bm{e}_h^{n-1}, \bm{\theta}_h )  + \tau  a_{\mathcal{C}}(\bm{u}_h^{n-1};  \bm{u}_h^{n-1},\err, \bm{\theta}_h) + \tau R_{\mathcal{C}}(\bm{\theta}_h) + \tau \mu a_\mathcal{D}(\err, \bm{\theta}_h) \nonumber \\ & =  \tau b(\bm{\theta}_h, p_h^{n} - \pi_h p^n)   - \tau b(\bm{\theta}_h, p^n - \pi_h p^n) - \tau \mu a_\mathcal{D}(\Pi_h \bm{u}^n - \bm{u}^n ,\bm{\theta}_h)\nonumber \\ & \quad +\tau \delta \mu b(\bm{\theta}_h, \nabla_h \cdot \err - R_h([\err]))  +  R_t(\bm{\theta}_h).  \nonumber 
\end{align}
Choosing $\bm{\theta}_h = \bm{\zeta}^n_h$ and using \eqref{eq:results_inf_sup} result in 
\begin{align} 
&\tau \|p_h^n - \pi_h p^n \|^2 =
(\bm{e}_h^{n} - \bm{e}_h^{n-1}, \bm{\zeta}_h^n ) + \tau  a_{\mathcal{C}}(\bm{u}_h^{n-1}; \bm{u}_h^{n-1}, \err , \bm{\zeta}_h^n) + \tau R_{\mathcal{C}}(\bm{\zeta}_h^n) \nonumber\\ & + \tau \mu a_\mathcal{D}(\err, \bm{\zeta}_h^n)  + \tau b(\bm{\zeta}_h^n, p^n - \pi_h p^n )+  \tau \mu a_\mathcal{D}(\Pi_h \bm{u}^n - \bm{u}^n ,\bm{\zeta}_h^n)\nonumber\\ &- \tau \delta \mu b(\bm{\zeta}_h^n, \nabla_h \cdot \err - R_h([\err])) - R_t(\bm{\zeta}_h^n).  \label{eq:pressure_err_eq1}
\end{align}
The first term is bounded using Cauchy-Schwarz's inequality, \eqref{eq:results_inf_sup} and Young's inequality. 
\begin{equation}
|(\bm{e}_h^{n} - \bm{e}_h^{n-1}, \bm{\zeta}_h^n ) | \leq  C   \tau \|\delta_\tau \bm{e}_h^n\|^2 + \frac{\tau}{16} \| p_h^n - \pi_h p^n\|^2.
\end{equation}
With Lemma \ref{lemma:bd_eta_1_eta_2_semi} bound \eqref{eq:apendix_lemma2_firstestimate} in the appendix, we obtain: 
\begin{align}
|a_{\mathcal{C}}(\bm{u}_h^{n-1}; \bm{u}_h^{n-1}, \err , \bm{\zeta}_h^n)| \leq C  \|\bm{u}_h^{n-1}\|_{L^3(\Omega)}\|\err\|_{\DG} \|\bm{\zeta}_h^n \|_{\DG}.
\end{align}
With Lemma \ref{lemma:first_err_estimate},  \eqref{eq:approximation_prop_1}, \eqref{eq:inverse_estimate}, the assumption that $\bm{u} \in L^{\infty}(0,T; H^{k+1}(\Omega)^d)$, and the CFL condition \eqref{eq:cfl_cond_detla}, we have
\begin{align*}
&\|\bm{u}_h^{n-1}\|^2_{L^3(\Omega)}  \leq  2 \|\bm{e}_h^{n-1}\|^2_{L^3(\Omega)} + 2\| \Pi_h \bm{u}^{n-1}\|^2_{L^3(\Omega)}\\ & \leq C_\mu \left(1+\frac{1}{\mu} + \mu\right) h^{-d/3} (\tau + h^{2k}) +C\|\bm{u}^{n-1}\|^2_{W^{1,3}(\Omega)}\leq C_\mu \left(1+\frac{1}{\mu} + \mu\right). 
\end{align*}
Hence, using Young's inequality  and \eqref{eq:results_inf_sup}, we have:
\begin{align}
     |a_{\mathcal{C}}(\bm{u}_h^{n-1}; \bm{u}_h^{n-1}, \err , \bm{\zeta}_h^n)| \leq C_\mu\left(1+\frac{1}{\mu} + \mu\right)\|\err\|^2_{\DG}  + \frac{1}{16}  \|p_h^n - \pi_h p^n \|^2.
\end{align}
For the remaining nonlinear terms, we split them as follows: 
\begin{align}
 R_{\mathcal{C}}(\bm{\zeta}_h^n) &=  a_\mathcal{C}(\bm{u}_h^{n-1}; \bm{u}_h^{n-1}, \Pi_h \bm{u}^n - \bm{u}^n,\bm{\zeta}_h^n) +  a_\mathcal{C}(\bm{u}_h^n; \bm{e}_h^{n-1},\bm{u}^n, \bm{\zeta}_h^n ) \nonumber \\ &\quad + a_\mathcal{C}(\bm{u}_h^n; \Pi_h \bm{u}^{n-1} - \bm{u}^{n-1},\bm{u}^n, \bm{\zeta}_h^n ) +  a_\mathcal{C}(\bm{u}^n;\bm{u}^{n-1}-\bm{u}^n,\bm{u}^n,\bm{\zeta}_h^n)\nonumber  =\sum_{i = 1}^4  L_i^n 
\end{align}
We further split $L_1$ in the following way. 
\begin{align*}
L^n_1  &=      \mathcal{C}(\bm{e}_h^{n-1}, \Pi_h \bm{u}^n - \bm{u}^n,\bm{\zeta}_h^n) +  \mathcal{C}(\Pi_h \bm{u}^{n-1}, \Pi_h \bm{u}^n - \bm{u}^n, \bm{\zeta}_h^n) \\
&  - \mathcal{U}(\bm{u}_h^{n-1}; \bm{e}_h^{n-1}, \Pi_h \bm{u}^n - \bm{u}^n,\bm{\zeta}_h^n) 
-  \mathcal{U}(\bm{u}_h^{n-1}; \Pi_h \bm{u}^{n-1}, \Pi_h \bm{u}^n - \bm{u}^n,\bm{\zeta}_h^n).   
\end{align*}
We apply Lemma \ref{prep:bounds_nonlinear_terms}: \eqref{eq:second_estimate_nonlinear_term_C1} for the first term, 
\eqref{eq:second_estimate_nonlinear_term_U2} for the second and fourth terms, and \eqref{eq:second_estimate_nonlinear_term} for the third term.
%With the assumption that $\bm{u} \in L^{\infty}(0,T; H^{k+1}(\Omega)^d)$, we have: 
\begin{equation*}
|L^n_1| \leq C (\| \bm{e}_h^{n-1}\|  + h^{k} \Vert \bm{u}^{n} \Vert_{H^{k+1}(\Omega)}) \| \bm{\zeta}_h^n\|_{\DG}. 
\end{equation*}
The term $L^n_2$ is bounded by \eqref{eq:first_estimate_nonlinear_form}, $L^n_3$ is bounded by  \eqref{eq:third_estimate_nonlinear_term},  and $L^n_4$ is bounded as follows:
\[
L_4^n = \int_\Omega (\bm{u}^{n-1}-\bm{u}^n) \cdot \nabla \bm{u}^n \cdot \bm{\zeta}_h^n  \leq  C\|\bm{u}^n -\bm{u}^{n-1}\||\bm{u}^n|_{W^{1,3}(\Omega)} \|\bm{\zeta}^n_h\|_{\DG}. 
\] 
With the assumption that $\bm{u} \in L^{\infty}(0,T;H^{k+1}(\Omega)^d)$ and \eqref{eq:results_inf_sup}, we obtain: 
\begin{align}
&|R_{\mathcal{C}}(\bm{\zeta}_h^n)| %\leq C (\| \bm{e}_h^{n-1}\| + h^{k} + \| \bm{u}^n - \bm{u}^{n-1}\|) \| \bm{\zeta}_h^n\|_{\DG} \nonumber \\ 
 \leq C  (\| \bm{e}_h^{n-1}\|^2  + h^{2k }  )+    C\tau  \int_{t^{n-1}}^{t^{n}} \| \partial_t \bm{u}\|^2 + \frac{1}{16} \| p_h^n - \pi_h p^n  \|^2.
\end{align}
Following similar arguments to bound (6.90) in \cite{inspaper1}, we have 
\[|R_t(\bm{\zeta}_h^n)| \leq \frac{1}{16} \tau \|p_h^n - \pi_h p^n\|^2 + C\tau^2 \int_{t^{n-1}}^{t^n} \|\partial_{tt} \bm{u}\|^2 + C h^{2k+2} \int_{t^{n-1}}^{t^n}\| \partial_t \bm{u}\|^2_{H^{k+1}(\Omega)}. \]
The remaining terms are handled via similar arguments as before. The details are skipped for brevity. We have:
\begin{align*}
 |a_\mathcal{D}(\err, \bm{\zeta}_h^n)|&\leq C\|\err\|_{\DG}\|\bm{\zeta}_h^n\|_{\DG}
\leq C\|\err\|_{\DG} \, \|p_h^n - \pi_h p^n\|,\\
 |b(\bm{\zeta}_h^n,  p^n - \pi_h p^n)| &\leq Ch^k\|\bm{\zeta}_h^n\|_{\DG} | p^n |_{H^k(\Omega)}
\leq Ch^k | p^n |_{H^k(\Omega)}\, \|p_h^n - \pi_h p^n\|,\\
  |a_\mathcal{D}(\Pi_h \bm{u}^n - \bm{u}^n , \bm{\zeta}_h^n) | &\leq C h^{k}| \bm{u}^n|_{H^{k+1}(\Omega)}\|\bm{\zeta}_h^n\|_{\DG}
\leq C h^{k}| \bm{u}^n|_{H^{k+1}(\Omega)} \, \|p_h^n - \pi_h p^n\|,\\
| b(\bm{\zeta}_h^n, \nabla_h \cdot \err - R_h([\err]))| &\leq 
C \|\bm{\zeta}_h^n\|_{\DG} \|\err \|_{\DG} \leq C \|\err \|_{\DG} \, \|p_h^n - \pi_h p^n\|.
%C\| \bm{\zeta}_h^n  \|_{\DG} \|\nabla_h \cdot \err - R_h([\err]) \| \\ & \quad \leq  C \|\bm{\zeta}_h^n\|_{\DG} \|\err \|_{\DG}.  %\leq C\| \err \|_{\DG}^2 + \ell   \| p_h^n - \pi_h p^n  \|^2, \\  
\end{align*}
 Making use of Young's inequality for the above bounds with \eqref{eq:results_inf_sup}, substituting the bounds in \eqref{eq:pressure_err_eq1},
and summing the resulting inequality from $n=1$ to $n=m$. We obtain: 
\begin{align*}
  \frac{\tau}{2}  \sum_{n=1}^m \|p_h^n - \pi_h p^n\|^2 
 \leq C_{\mu}\left(\sum_{i=-1}^2 \mu^{i}\right)\tau\sum_{n=1}^m  \|\err\|_{\DG}^2
 + C\tau^2\int_0^T (\|\partial_{tt} \bm{u}\|^2  + \|\partial_t \bm{u}\|^2)\\
  + C\tau \sum_{n=1}^m (\| \bm{e}_h^{n-1} \|^2 + h^{2k}+ \|\delta_\tau \bm{e}_h^n\|^2) 
   + Ch^{2k+2}\int_0^T \|\partial_{t} \bm{u}\|^2_{H^{k+1}(\Omega)}  + C(\mu^2 + 1)h^{2k}.
 \end{align*}
The result is concluded by invoking Lemma \ref{lemma:first_err_estimate}, Lemma \ref{lemma:error_delta_velocity}, the bound on the initial error \eqref{eq:initial_error_eh},  and the triangle inequality.  \end{proof}

   % \section{Conclusion }\label{sec:conclusion}
    \section{Appendix}
\begin{lemma}\label{lemma:appendix_delta_expression}
[Equality \eqref{eq:bform_delta_expression}]\end{lemma}
\begin{proof}
%From \eqref{eq:update_velocity_1} and \eqref{eq:def_b_lift_2}, we have
%\begin{align*}
%\delta_\tau \bm{e}_h^{n+1} - \delta_\tau \errn = -\tau \nabla_h \delta_\tau \phi_h^{n+1} + \tau  \bm{G}_h([\delta_\tau \phi_h^{n+1}]),  \quad n \geq 1. 
%\end{align*}
%Then, we obtain: 
%\begin{align*}
%& \|\delta_\tau \errn -  \delta_\tau \bm{e}_h^n \|^2  = (\delta_\tau \errn - \delta_\tau \bm{e}_h^n , \delta_\tau \errn - \delta_\tau \bm{e}_h^n) \\ &= (\delta_\tau \bm{e}_h^{n+1} - \delta_\tau \bm{e}_h^n +  \tau(\nabla_h \delta_\tau \phi_h^{n+1} -  \bm{G}_h([\delta_\tau \phi_h^{n+1}])),\delta_\tau \bm{e}_h^{n+1} - \delta_\tau \bm{e}_h^n +  \tau(\nabla_h \delta_\tau \phi_h^{n+1} -  \bm{G}_h([\delta_\tau \phi_h^{n+1}])) \\ 
%& = \|\delta_\tau \bm{e}_h^{n+1} - \delta_\tau \bm{e}_h^n\|^2 -2\tau(\delta_\tau \bm{e}_h^{n+1} + \delta_\tau \bm{e}_h^n,\nabla_h \delta_\tau \phi_h^{n+1} -  \bm{G}_h([\delta_\tau \phi_h^{n+1}]))\\&  \quad  + \tau^2 \|\nabla_h \delta_\tau \phi_h^{n+1} -  \bm{G}_h([\delta_\tau \phi_h^{n+1}])\|^2. 
%\end{align*}
%Using \eqref{eq:def_b_lift_2} and expanding the last term, we obtain: 
    %\begin{align*}
        %\|\delta_\tau \errn - & \delta_\tau \bm{e}_h^n \|^2 =  \|\delta_\tau \bm{e}_h^{n+1} - \delta_\tau \bm{e}_h^n \|^2 + \tau^2 \| \nabla_h \delta_\tau \phi_h^{n+1} \|^2 + \tau^2 \|\bm{G}_h ([\delta_\tau \phi_h^{n+1}]) \|^2\\ &\quad -2\tau b(\delta_\tau \bm{e}_h^{n+1}  -  \delta_\tau \bm{e}_h^n, \delta_\tau \phi_h^{n+1})  -2\tau^2(\nabla_h \delta_\tau \phi_h^{n+1}, \bm{G}_h([\delta_\tau \phi_h^{n+1}])). 
        %\end{align*}
%We now show \eqref{eq:bform_delta_expression}.  
For $n=1$, with \eqref{eq:property_err_delta}, we have
\begin{align*}
A &= -2b(\delta_\tau \bm{e}_h^2 - \delta_\tau \bm{e}_h^1, \delta_\tau \phi_h^2) = -2b(\delta_\tau \bm{e}_h^2, \delta_\tau \phi_h^2) + 2b( \delta_\tau \bm{e}_h^1, \delta_\tau \phi_h^2) \\
&= 2\tau \sum_{e\in \Gamma_h} \tilde{\sigma} h_e^{-1} \|\delta_{\tau} \phi_h^2\|_{L^2(e)}^2 - 2\tau \|\bm{G}_h([\delta_\tau \phi_h^2])\|^2 + 2b( \delta_\tau \bm{e}_h^1, \delta_\tau \phi_h^2).
\end{align*}
We can write the last term as follows. Since $\phi_h^0 = 0$, $\phi_h^1 =\tau \delta_\tau \phi_h^1$, we write 
\begin{align*}
&2b( \delta_\tau \bm{e}_h^1, \delta_\tau \phi_h^2) = \frac{2}{\tau} b(\bm{e}_h^1, \delta_\tau \phi_h^2) -  \frac{2}{\tau} b(\bm{e}_h^0, \delta_\tau \phi_h^2) \\ 
    &=  -2\tau \sum_{e\in \Gamma_h} \tilde{\sigma} h_e^{-1} \int_e [\delta_\tau \phi_h^1] [\delta_\tau \phi_h^2] + 2 \tau (\bm{G}_h([\delta_\tau \phi_h^1]), \bm{G}_h([\delta_\tau \phi_h^2])) -  \frac{2}{\tau} b(\bm{e}_h^0, \delta_\tau \phi_h^2).
\end{align*}
Thus, we have 
\begin{align*}
&A = \tau \sum_{e\in\Gamma_h}\tilde{\sigma} h_e^{-1} (\|[\delta_\tau \phi_h^2]\|_{L^2(e)}^2  -\|[\delta_\tau \phi_h^1]\|_{L^2(e)}^2 + \|[\delta_\tau \phi_h^2 - \delta_\tau \phi_h^1]\|_{L^2(e)} ^2 ) \\ & -  \tau (\|\bm{G}_h([\delta_\tau \phi_h^2])\|^2 -\|\bm{G}_h([\delta_\tau \phi_h^1])\|^2 + \|\bm{G}_h([\delta_\tau \phi_h^2 - \delta_\tau \phi_h^1])\|^2 ) -  \frac{2}{\tau} b(\bm{e}_h^0, \delta_\tau \phi_h^2) \\
& = \tau (\tilde{A}_1^1 - \tilde{A}_2^1) + \tau ( \sum_{e\in\Gamma_h} \tilde{\sigma} h_e^{-1}\|[\delta_\tau \phi_h^2 - \delta_\tau \phi_h^1]\|_{L^2(e)} ^2\\ &    - \|\bm{G}_h([\delta_\tau \phi_h^2 - \delta_\tau \phi_h^1])\|^2 ) -  \frac{2}{\tau} b(\bm{e}_h^0, \delta_\tau \phi_h^2).  
\end{align*} 
For $n \geq 2$, we use \eqref{eq:property_err_delta} for both terms. 
\begin{align*}
  -2b(\delta \bm{e}_h^{n+1} - \delta_\tau \bm{e}_h^n , \delta_\tau \phi_h^{n+1}) = 2 \tau \sum_{e\in \Gamma_h} \tilde{\sigma} h_e^{-1} \int_e [\delta_\tau \phi_h^{n+1}]([\delta_\tau \phi_h^{n+1}] -[\delta_\tau \phi_h^n]) \\ - 2\tau  (\bm{G}_h([\delta_\tau \phi_h^{n+1}], \bm{G}_h([\delta_\tau \phi_h^{n+1}] - \bm{G}_h([\delta_\tau \phi_h^{n}])).
\end{align*}
Using the fact that:
\[(f,f-g) = \frac{1}{2}\|f\|^2 - \frac12 \|g\|^2 + \frac12 \|f-g\|^2, \]
we obtain the result.
\end{proof}
\begin{lemma} \label{lemma:bound_I1} 
    Let $\bm{u}_h,\bm{\theta}_h,  \bm{w}_h \in \bm{X}_h$ and $\bm{v} \in \bm{X}$. There exists a constant $C$ independent of $h$ such that the following bounds hold for $I =\mathcal{U}(\bm{u}_h; \bm{u}_h, \bm{v},  \bm{w}_h) - \mathcal{U}(\bm{\theta}_h; \bm{u}_h, \bm{v}, \bm{w}_h)$. 
    \begin{align}
        |I| &\leq C \| \bm{u}_h - \bm{\theta}_h \|_{L^3(\Omega)}\left(\sum_{e\in \Gamma_h \cup \partial \Omega}\sigma h_e^{-1} \|[\bm{v}]\|_{L^2(e)}^2\right)^{1/2}\|\bm{w}_h\|_{\DG},  \label{eq:lemma_I_appendix_result1}\\ 
        |I| &\leq C \| \bm{u}_h - \bm{\theta}_h \|\left(\sum_{e\in\Gamma_h \cup \partial \Omega} \sigma h_e^{-1} \|[\bm{v}]\|_{L^{2}(e)}^2 \right)^{1/2}\|\bm{w}_h \|_{L^{\infty}(\Omega)} .\label{eq:lemma_I_appendix_result2}
    \end{align}
\end{lemma}
\begin{proof}
    The proof of this Lemma closely follows the proof of Proposition 4.10 in \cite{girault2009dg}. The main idea of the proof is to write the expression of $\mathcal{U}$ in terms of the faces' contributions \cite{girault2009dg}. Take $\bm{z}_h \in \bm{X}_h$. Fix a face $e \in \Gamma_h$ adjacent to the elements $E_e^1$ and $E_e^2$ with $\bm{n}_e = \bm{n}_{E_e^1}$.  The contribution of $e$ to the expression $\mathcal{U}(\bm{z}_h; \bm{u}_h, \bm{v}, \bm{w}_h)$ is given by: 
\begin{equation}
\int_e \{ \bm{u}_h \} \cdot  \bm{n}_e [\bm{v}] \cdot  \bm{w}_h^{\bm{z}_h}, 
\end{equation}
where for $e \in \Gamma_h$,  $\bm{w}_h^{\bm{z}_h}$ is defined as follows. 
\begin{equation}
     \bm{w}_h^{\bm{z}_h} \vert_e = \begin{cases}
        \bm{w}_h \vert_{E_e^1},  & \{ \bm{z}_h \}  \cdot \bm{n}_e < 0, \\
        \bm{0},  &  \{\bm{z}_h\} \cdot \bm{n}_e = 0 , \\ 
        \bm{w}_h \vert_{E_e^2},  & \{ \bm{z}_h \}  \cdot \bm{n}_e  >0. \\ 
    \end{cases}
\end{equation}
If $\bm{e} \in \partial \Omega \cap E_e$, then $\bm{n}_e = \bm{n}_{\partial \Omega}$. Then, the contribution of $e$ is given by:
\begin{equation}
    \int_e \bm{u}_h \cdot \bm{n}_e \bm{v} \cdot \bm{w}_h^{\bm{z}_h}, 
\end{equation}
where for $e \in \partial \Omega$, $\bm{w}_h^{\bm{z}_h}$ is defined as follows.
\begin{equation}
    \bm{w}_h^{\bm{z}_h} \vert_e = \begin{cases}
        \bm{w}_h \vert_{E_e},  & \{ \bm{z}_h \}  \cdot \bm{n}_e < 0, \\
        \bm{0},  &  \mathrm{otherwise}.  \\ 
    \end{cases}
\end{equation}
Hence, we have the following expression for $I = \mathcal{U}(\bm{u}_h; \bm{u}_h, \bm{v}, \bm{w}_h) - \mathcal{U}(\bm{\theta}_h; \bm{u}_h, \bm{v}, \bm{w}_h):$ 
\begin{equation}
    I = \sum_{e \in \Gamma_h \cup \partial \Omega} \int_e \{\bm{u}_h\} \cdot \bm{n}_e [\bm{v}] (\bm{w}_h^{\bm{u}_h } - \bm{w}_h^{\bm{\theta}_h}). 
\end{equation}
We now split the domain of integration as follows \cite{girault2009dg}. 
\begin{equation}
    \Gamma_h \cup \partial \Omega = \mathcal{F}_1 \cup \mathcal{F}_2 \cup \mathcal{F}_3, 
\end{equation}
where \begin{align*}
    \mathcal{F}_1 & = \{e: \{\bm{u}_h\}\cdot \bm{n}_e \neq 0 \, \, \mathrm{ and } \, \, \{ \bm{\theta}_h \} \cdot \bm{n}_e \neq 0 \,\,  a.e \,\, \mathrm{ on } \,\, e \},\\ 
    \mathcal{F}_2 & = \{e: \{\bm{u}_h\}\cdot \bm{n}_e \neq 0 \, \, \mathrm{ and } \, \, \{ \bm{\theta}_h \} \cdot \bm{n}_e =  0 \,\,  a.e \,\, \mathrm{ on } \,\, e \}, \\ 
    \mathcal{F}_3 & = (\Gamma_h \cup \partial \Omega ) \backslash (\mathcal{F}_1 \cup \mathcal{F}_2).  
\end{align*}
The contribution of $\mathcal{F}_3$ to $I$ is zero. Consider $\mathcal{F}_1$ and split $e$ into $e_1$ where $\{\bm{u}_h\}\cdot \bm{n}_e$ and $\{\bm{\theta}_h\}\cdot \bm{n}_e$ have the same sign and into $e_2$ where they have opposite signs. On $e_1$, $\bm{w}_h^{ \bm{u}_h} - \bm{w}_h^{\bm{\theta}_h} = \bm{0}$. On $e_2$, $\bm{w}_h^{\bm{u}_h} - \bm{w}_h^{\bm{\theta}_h} = [\bm{w}_h]$ up to the sign. Note that, due to the opposite signs on $e_2$, we have 
\begin{equation}
    |\{ \bm{u}_h\} \cdot \bm{n}_e | \leq |\{\bm{\theta}_h\} \cdot \bm{n}_e  - \{ \bm{u}_h\} \cdot \bm{n}_e |.
\end{equation}
 We apply Holder's inequality, trace estimate \eqref{eq:trace_ineq_discrete}, and \eqref{eq:discreter_poincare}. We obtain:
\begin{align}
&\sum_{e \in \mathcal{F}_1} \left\vert \int_e  \{\bm{u}_h\} \cdot \bm{n}_e [\bm{v}] (\bm{w}_h^{ \bm{u}_h} - \bm{w}_h^{\bm{\theta}_h}) \right\vert \nonumber\\ &  \leq  \sum_{e\in \mathcal{F}_1} \|\{ \bm{\theta}_h\} \cdot \bm{n}_e - \{\bm{u}_h\} \cdot \bm{n}_e\|_{L^{3}(e)}\|[\bm{v}]\|_{L^2(e)}\| [\bm{w}_h] \|_{L^{6}(e)}  \nonumber\\ & 
\leq  \sum_{e\in \mathcal{F}_1} \sum_{i,j=1}^2 Ch^{-1/3}\|\bm{\theta}_h - \bm{u}_h\|_{L^{3}(E^i_e )} \|[\bm{v}]\|_{L^2(e)}h^{-1/6}\| \bm{w}_h \|_{L^{6}(E_e^j)}  \nonumber \\ &
 %\leq C \|\bm{\theta}_h - \bm{u}_h \|_{L^3(\Omega)} \left(\sum_{e\in\Gamma_h \cup \partial \Omega} \sigma h_e^{-1} \|[\bm{v}]\|_{L^{2}(e)}^2 \right)^{1/2} \| \bm{w}_h \|_{L^{6}(\Omega)}  \nonumber \\ &
 \leq C \|\bm{\theta}_h - \bm{u}_h \|_{L^3(\Omega)} \left(\sum_{e\in\Gamma_h \cup \partial \Omega} \sigma h_e^{-1} \|[\bm{v}]\|_{L^{2}(e)}^2 \right)^{1/2}\| \bm{w}_h \|_{\DG}. \label{eq:appendix_lemmaI_1}  
\end{align}
An alternative bound reads: 
\begin{align}
&\sum_{e\in\mathcal{F}_1} \left\vert \int_e  \{\bm{u}_h\} \cdot \bm{n}_e [\bm{v}] (\bm{w}_h^{\bm{u}_h} - \bm{w}_h^{\bm{\theta}_h}) \right\vert \nonumber \\ & \leq C \|\bm{w}_h \|_{L^{\infty}(\Omega)} \|\bm{\theta}_h - \bm{u}_h\|\left(\sum_{e\in\Gamma_h \cup \partial \Omega} \sigma h_e^{-1} \|[\bm{v}]\|_{L^{2}(e)}^2 \right)^{1/2}.\label{eq:appendix_lemmaI_2}
\end{align}
The contribution of $\mathcal{F}_2$ to $I$ is bounded similarly by: 
\begin{multline}
    \sum_{e \in \mathcal{F}_2} \left\vert \int_e \{\bm{u}_h\} \cdot \bm{n}_e [\bm{v}] \bm{w}_h^{ \bm{u}_h} \right \vert \\  \leq C  \|\bm{\theta}_h - \bm{u}_h \|_{L^3(\Omega)} \left(\sum_{e\in\Gamma_h \cup \partial \Omega} \sigma h_e^{-1} \|[\bm{v}]\|_{L^{2}(e)}^2 \right)^{1/2} \| \bm{w}_h \|_{\DG}. \label{eq:appendix_lemmaI_3}
\end{multline}
Alternatively, we have 
\begin{multline}
    \sum_{e \in \mathcal{F}_2} \left\vert \int_e \{\bm{u}_h\} \cdot \bm{n}_e [\bm{v}] \bm{w}_h^{\bm{u}_h} \right \vert \\  \leq  \|\bm{\theta}_h - \bm{u}_h \| \left(\sum_{e\in\Gamma_h \cup \partial \Omega} \sigma h_e^{-1} \|[\bm{v}]\|_{L^{2}(e)}^2 \right)^{1/2}\| \bm{w}_h \|_{L^{\infty}(\Omega)}. \label{eq:appendix_lemmaI_4}
\end{multline}
Bounds \eqref{eq:appendix_lemmaI_1} and \eqref{eq:appendix_lemmaI_3} yield  \eqref{eq:lemma_I_appendix_result1}. Bounds \eqref{eq:appendix_lemmaI_2} and \eqref{eq:appendix_lemmaI_4} yield \eqref{eq:lemma_I_appendix_result2}. 
\end{proof}
\begin{lemma}
\label{lemma:bd_eta_1_eta_2_semi}
Let $\bm{u}_h,\bm{\theta}_h, \bm{v}_h$ and $\bm{w}_h \in \bm{X}_h$. There exists a constant $C$ independent of $h$ such that:
\begin{align}
|\mathcal{C}(\bm{u}_h, \bm{v}_h, \bm{w}_h )| + |\mathcal{U}(\bm{\theta}_h; \bm{u}_h, \bm{v}_h, \bm{w}_h ) | & \leq C \| \bm{u}_h \|_{L^3(\Omega)}\|\bm{w}_h\|_{\DG} \|\bm{v}_h \|_{\DG}.  \label{eq:apendix_lemma2_firstestimate}
\end{align}
In addition, let $\bm{v} = \bm{z} + \bm{z}_h$ where $\bm{z} \in \bm{X}$ and $\bm{z}_h \in \bm{X}_h$. We have, 
\begin{align}
&|\mathcal{C}(\bm{u}_h, \bm{v}, \bm{w}_h )|+| \mathcal{U}(\bm{\theta}_h; \bm{u}_h, \bm{v}, \bm{w}_h ) | \leq C\|\bm{u}_h\|\left(\| \nabla_h \bm{w}_h\|_{L^3(\Omega)} + \| \bm{w}_h\|_{L^{\infty}(\Omega)}\right) \|\bm{v}\|_{\DG}\nonumber \\ &  + C \|\bm{u}_h \|_{L^3(\Omega)}\left(\sum_{e\in\Gamma_h} \sigma h_e^{-1} \| [\bm{w}_h]\|_{L^2(e)}^2\right)^{1/2}\|\bm{z}_h\|_{\DG} \nonumber \\ &+ C\|\bm{u}_h\| \|\bm{w}_h\|_{L^{\infty}(\Omega)}(h^{-1} \| \bm{z}\| + \|\nabla_h \bm{z}\|). \label{eq:appendix_lemma3_secondestimate}
\end{align}
\end{lemma}
\begin{proof}
We have
\begin{align*}
\mathcal{C}(\bm{u}_h, \bm{v}_h, \bm{w}_h ) = \sum_{E \in \mesh_h} \int_E (\bm{u}_h \cdot \nabla \bm{v}_h)\cdot \bm{w}_h + \frac{1}{2} b(\bm{u}_h, \bm{v}_h \cdot \bm{w}_h).
\end{align*}
The first term is bounded by Cauchy-Schwarz inequality and \eqref{eq:discreter_poincare}. We have
\begin{align}
    \sum_{E \in \mesh_h} \left| \int_E (\bm{u}_h \cdot \nabla \bm{v}_h)\cdot \bm{w}_h \right|  \leq C \|\bm{u}_h\|_{L^3(\Omega)} \| \nabla \bm{v}_h \| \|\bm{w}_h\|_{\DG}. \label{eq:appendix_lemma2_1}
\end{align}
For the second term, we use \eqref{eq:equiv_form_b} and the following identity: 
\begin{equation}
[\bm{v}_h \cdot \bm{w}_h] = [\bm{v}_h ] \cdot \{ \bm{w}_h \} +  [\bm{w}_h ] \cdot \{ \bm{v}_h \}, \quad \forall e \in \Gamma_h 
\end{equation}
We obtain 
\begin{align*}
\frac{1}{2} b(\bm{u}_h, \bm{v}_h \cdot \bm{w}_h )& = -\frac{1}{2} \sum_{E\in \mesh_h} \int_{E} \bm{u}_h \cdot ((\nabla \bm{v}_h)^T \bm{w}_h  + (\nabla \bm{w}_h)^T \bm{v}_h)  \\ & \quad + \frac{1}{2}
\sum_{e\in\Gamma_h} \int_e \{\bm{u}_h\} \cdot \bm{n}_e ([\bm{v}_h] \cdot \{\bm{w}_h\} + [\bm{w}_h] \cdot \{\bm{v}_h\}).
\end{align*}
With Holder's inequality, trace estimate \eqref{eq:trace_ineq_discrete} and \eqref{eq:discreter_poincare}, we have 
\begin{align}
\frac{1}{2}\left| b(\bm{u}_h, \bm{v}_h \cdot \bm{w}_h)\right|   %& \leq C \| \bm{u}_h\|_{L^3(\Omega)} (\| \nabla_h \bm{v}_h \| \|\bm{w}_h\|_{L^{6}(\Omega)} + \| \nabla_h \bm{w}_h \| \|\bm{v}_h\|_{L^{6}(\Omega)}) \nonumber \\ 
%& \quad +  C \| \bm{u}_h\|_{L^3(\Omega)} (\| \bm{v}_h \|_{\DG} \|\bm{w}_h\|_{L^{6}(\Omega)} + \| \bm{w}_h \|_{\DG} \|\bm{v}_h\|_{L^{6}(\Omega)}), \nonumber \\ 
 \leq C  \| \bm{u}_h\|_{L^3(\Omega)} \| \bm{v}_h \|_{\DG}  \|\bm{w}_h\|_{\DG}. \label{eq:appendix_lemma2_2} 
\end{align}
The upwind term is bounded via similar arguments: 
\begin{align}
|\mathcal{U}(\bm{\theta}_h; \bm{u}_h, \bm{v}_h, \bm{w}_h )| \leq  C \| \bm{u}_h\|_{L^3(\Omega)} \| \bm{v}_h \|_{\DG} \|\bm{w}_h\|_{\DG}.  \label{eq:appendix_lemma2_3}
\end{align}
Bounds \eqref{eq:appendix_lemma2_1}, \eqref{eq:appendix_lemma2_2} and \eqref{eq:appendix_lemma2_3} yield bound \eqref{eq:apendix_lemma2_firstestimate}. 
Let  $\bm{v}_h$ be replaced by $\bm{v} = \bm{z} + \bm{z}_h$ where $\bm{z} \in \bm{X}$ and $\bm{z}_h \in \bm{X}_h$. We replace bound \eqref{eq:appendix_lemma2_1} by: 
\begin{align}
\sum_{E \in \mesh_h} \left\vert \int_E (\bm{u}_h \cdot \nabla \bm{v} ) \cdot \bm{w}_h \right\vert \leq \| \bm{u}_h \|\| \nabla_h \bm{v} \| \|\bm{w}_h\|_{L^\infty(\Omega)}.\label{eq:appendix_lemma2_4}
\end{align}
We also write:  
\begin{align*}
    \frac{1}{2} b(\bm{u}_h, \bm{v} \cdot \bm{w}_h )& = -\frac{1}{2} \sum_{E\in \mesh_h} \int_{E} \bm{u}_h \cdot ((\nabla \bm{v})^T \bm{w}_h  + (\nabla \bm{w}_h)^T \bm{v})  \\ & \quad + \frac{1}{2}
    \sum_{e\in\Gamma_h} \int_e \{\bm{u}_h\} \cdot \bm{n}_e ([\bm{v}] \cdot \{\bm{w}_h\} + [\bm{w}_h] \cdot \{\bm{z}\} + [\bm{w}_h] \cdot \{\bm{z}_h\}).
    \end{align*}
With H\"{o}lder's inequality and trace estimate \eqref{eq:trace_ineq_discrete}, we obtain 
\begin{align}
&\frac{1}{2}|b(\bm{u}_h, \bm{v} \cdot \bm{w}_h)| \leq C \|\bm{u}_h\|\| \nabla_h \bm{v}\| \| \bm{w}_h \|_{L^{\infty}(\Omega)} + C \| \bm{u}_h \| \|\nabla_h \bm{w}_h\|_{L^3(\Omega)} \| \bm{v}\|_{L^6(\Omega)}  \nonumber \\ & \quad + C \| \bm{w}_h \|_{L^{\infty}(\Omega)}\|\bm{u}_h \| \left(\sum_{e\in\Gamma_h} \sigma h_e^{-1}\| [\bm{v}]\|_{L^2(e)}^2\right)^{1/2} \nonumber \\ & \quad + C \| \bm{w}_h \|_{L^{\infty}(\Omega)}\|\bm{u}_h\| (h^{-1} \| \bm{z}\| + \| \nabla_h \bm{z}\|)  \nonumber \\ & \quad + C \|\bm{u}_h \|_{L^3(\Omega)} \left(\sum_{e\in\Gamma_h} \sigma h_e^{-1} \| [\bm{w}_h] \|^2_{L^2(e)}\right)^{1/2} \|\bm{z}_h\|_{L^6(\Omega)}.  \label{eq:appendix_lemma2_5}
\end{align}
An alternative bound for the upwind term reads: 
\begin{align}
    |\mathcal{U}(\bm{\theta}_h; \bm{u}_h, \bm{v}, \bm{w}_h )| \leq C \| \bm{w}_h \|_{L^\infty(\Omega)} \| \bm{u}_h \|\| \bm{v}\|_{\DG}.\label{eq:appendix_lemma2_6}
\end{align}
Bounds \eqref{eq:appendix_lemma2_4}, \eqref{eq:appendix_lemma2_5}, \eqref{eq:appendix_lemma2_6}, and \eqref{eq:discreter_poincare} yield estimate \eqref{eq:appendix_lemma3_secondestimate}. 
\end{proof}
\begin{lemma}\label{lemma:bound_hat_Rt_stability_proof}
    [Bound \eqref{eq:bound_hat_Rt_stability_proof}]
\end{lemma}
\begin{proof}
 For $n\geq 0$, recall that  $\bm{E}^n \in \bm{X}$: $\bm{E}^n = \Pi_h \bm{u}^n - \bm{u}^n$. We write: 
\begin{align*}
\hat{R}_t(\delta_\tau \errn) & = \frac{1}{\tau}\left(\tau(\partial_t \bm{u})^{n+1} - \tau (\partial_t \bm{u})^n - \bm{u}^{n+1} + 2\bm{u}^n - \bm{u}^{n-1}, \delta_\tau \errn \right) \\ & \quad  -\frac{1}{\tau}\left(\bm{E}^{n+1} - 2\bm{E}^n + \bm{E}^{n-1}, \delta_\tau \errn\right). 
\end{align*}
With Cauchy-Schwarz's and Young's inequality, we have 
\begin{align*}
& |\hat{R}_t(\delta_\tau \errn)|  \leq \ell \tau \mu \| \delta_\tau \errn\|_{\DG}^2 + \frac{C}{\mu}\tau^{-3}\|\bm{u}^n - \bm{u}^{n+1}+ \tau (\partial_t \bm{u})^{n+1}\|^2 \\
& \quad + \frac{C}{\mu}\tau^{-3} \|\bm{u}^n - \bm{u}^{n-1} - \tau (\partial_t \bm{u})^n\|^2  + \frac{C}{\mu} \tau^{-3} \|\bm{E}^{n+1} - \bm{E}^n \|^2 + \frac{C}{\mu}\tau^{-3}\|\bm{E}^n - \bm{E}^{n-1}\|^2. 
\end{align*}
% With Taylor expansions, we have 
% \begin{align*}
% \bm{u}^n & = \bm{u}^{n+1} - \tau (\partial_t \bm{u})^{n+1} + \int_{t^n}^{t^{n+1}}(t-t^n) \partial_{tt} \bm{u} ,\\  
% \bm{u}^{n-1} & = \bm{u}^{n} -  \tau (\partial_t \bm{u})^{n} + \int_{t^{n-1}}^{t^n} (t -t^{n-1}) \partial_{tt} \bm{u}.
% \end{align*}
With Taylor expansions and Cauchy-Schwarz's inequality, we  have  
\begin{align*}
    \|\bm{u}^n - \bm{u}^{n+1}+ \tau (\partial_t \bm{u})^{n+1}\|^2 &\leq \tau^3 \int_{t^n}^{t^{n+1}} \|\partial_{tt} \bm{u}\|^2,  \\ 
    \|\bm{u}^n - \bm{u}^{n-1} - \tau (\partial_t \bm{u})^n\|^2 &\leq \tau^3 \int_{t^{n-1}}^{t^{n}} \|\partial_{tt} \bm{u}\|^2. 
\end{align*}
Similarly, with a Taylor expansion and 
 approximation properties \eqref{eq:approximation_prop_1} - \eqref{eq:approximation_prop_2}, for $n \geq 1$:
\begin{align*}
\| \bm{E}^n - \bm{E}^{n-1} \|^2 \leq \tau \int_{t^{n-1}}^{t^n} \| \partial_t \bm{E}\|^2 \leq C \tau h^4 \int_{t^{n-1}}^{t^n} |\partial_t \bm{u} |^2_{H^2(\Omega)}. 
\end{align*}
With the CFL condition \eqref{eq:cfl_cond_detla}, we know that $h^4 \leq c_2^2 \tau^2$. This implies that for $n \geq 1$: 
\begin{align*}
    \| \bm{E}^n - \bm{E}^{n-1} \|^2 \leq C \tau^3 \int_{t^{n-1}}^{t^n} |\partial_t \bm{u} |^2_{H^2(\Omega)}. 
    \end{align*}
Thus, by using these bounds, we obtain 
\begin{align}
    \hat{R}_t(\delta_\tau\errn) & \leq \ell \tau \mu \| \delta_\tau \errn\|_{\DG}^2 +  \frac{C}{\mu} \int_{t^{n-1}}^{t^{n+1}} \| \partial_{tt} \bm{u }\|^2 + \frac{C}{\mu} \int_{t^{n-1}}^{t^{n+1}}|\partial_t \bm{u}|_{H^{2}(\Omega)}^2. 
\end{align}
\end{proof}
\begin{lemma}\label{lemma:bound_on_I11}[Bound \eqref{eq:boundI11}]
\end{lemma}
\begin{proof}
The expression for $I^n_{1,1}$ reads: 
\begin{align}
&I^n_{1,1}  =  \sum_{E \in \mesh_h} \int_E (\bm{u}_h^{n-1} \cdot \nabla(\delta_\tau \bm{U}_h^{n+1} - \delta_\tau \bm{U}^{n+1})) \cdot (\delta_{\tau} \bm{E}^{n+1} + \delta_\tau \errn) \nonumber \\ & + \frac{1}{2} b(\bm{u}_h^{n-1}, (\delta_\tau \bm{U}_h^{n+1} - \delta_\tau \bm{U}^{n+1})\cdot (\delta_\tau \bm{E}^{n+1} +  \delta_\tau \errn)) \nonumber \\ 
&  + \sum_{E \in \mesh_h} \int_{\partial E_{-}^{\bm{u}_h^{n-1}}\backslash \partial \Omega} |\{ \bm{u}_h^{n-1} \} \cdot  \bm{n}_E | ((\delta_\tau \bm{U}_h^{n+1})^{\mathrm{int}} - (\delta_\tau \bm{U}_h^{n+1})^{\mathrm{ext}})\cdot (\delta_\tau \bm{E}^{n+1} + \delta_\tau \errn)^{\mathrm{ext}} \nonumber \\ 
&- \frac{1}{2} \sum_{e \in \partial \Omega} \int_{e} (| \bm{u}_h^{n-1} \cdot \bm{n}_e| -  \bm{u}_h^{n-1} \cdot \bm{n}_e) (\delta_\tau \bm{U}_h^{n+1})\cdot (\delta_\tau \bm{E}^{n+1} + \delta_\tau \errn) \nonumber 
 = \sum_{i=1}^4 S_i^n. % + S_2 +S_3 +S_4. 
\end{align}
We handle each term separately. For $S^n_1$, we have 
\begin{align*}
    S^n_1 =   \sum_{E \in \mesh_h} \int_E ((\bm{e}_h^{n-1} + \Pi_h \bm{u}^{n-1}) \cdot  \nabla  (\delta_\tau \bm{U}_h^{n+1} - \delta_\tau \bm{U}^{n+1})) \cdot (\delta_{\tau} \bm{E}^{n+1} + \delta_\tau \errn) . 
\end{align*}
We use Holder's inequality, \eqref{eq:discreter_poincare},  \eqref{eq:regularity_delta}, \eqref{eq:conv_delta_dual},  \eqref{eq:l3_bd_nabla_delta_Uh}, and \eqref{eq:stability_linf_pih}. 
\begin{align}
|S^n_1| & \leq  \|\bm{e}_h^{n-1}\|\|\nabla_h (\delta_\tau \bm{U}_h^{n+1} - \delta_\tau \bm{U}^{n+1})\|_{L^3(\Omega)}\|\delta_\tau \bm{E}^{n+1} + \delta_\tau \errn\|_{L^6(\Omega)} \nonumber \\ & \quad  +  \|\Pi_h \bm{u}^{n-1}\|_{L^{\infty}(\Omega)} \|\nabla_h  (\delta_\tau \bm{U}_h^{n+1} - \delta_\tau \bm{U}^{n+1})\| \| \delta_\tau \bm{E}^{n+1}+ \delta_\tau \errn\| \nonumber \\ & \leq C \|\bm{e}_h^{n-1}\|\|\delta_\tau \bm{\chi}^{n+1}\| (\|\delta_\tau \bm{E}^{n+1}\Vert_{\DG} + \Vert\delta_\tau \errn \|_{\DG}) \nonumber \\ & \quad + C  \vertiii{\bm{u}^{n-1}}h\|\delta_\tau \bm{\chi}^{n+1}\| (\|\delta_\tau \bm{E}^{n+1}\Vert  +\Vert \delta_\tau \errn \|). 
\end{align}
From \eqref{eq:update_velocity_delta} and \eqref{eq:boundformb}, note that 
\begin{align}
\|\delta_\tau \errn\| \leq \| \delta_\tau \bm{e}_h^{n+1} \| +C \tau |\delta_\tau \phi_h^{n+1}|_{\DG}. \label{eq:relating_errn_eh_deltaphi}
\end{align}
Hence,  with \eqref{eq:boundEn}, we obtain 
\begin{align}
& |S^n_1|\leq C \|\bm{e}_h^{n-1}\|\|\delta_\tau \bm{\chi}^{n+1}\| ( h^{k} \vert \delta_\tau \bm{u}^{n+1}\vert_{H^{k+1}(\Omega)}   
+ \Vert \delta_\tau \errn \|_{\DG}) \nonumber \\ &  + C  \vertiii{\bm{u}^{n-1}}h\|\delta_\tau \bm{\chi}^{n+1}\| (  h^{k+1} \vert \delta_\tau \bm{u}^{n+1}\vert_{H^{k+1}(\Omega)}  +\tau  | \delta_\tau \phi_h^{n+1} |_{\DG} + \|\delta_\tau \bm{e}_h^{n+1} \|). 
\end{align}
For $S^n_2$, we have 
\begin{align}
S^n_2 &=  \frac{1}{2} b(\bm{u}_h^{n-1}, (\delta_\tau \bm{U}_h^{n+1} - \delta_\tau \bm{U}^{n+1})\cdot \delta_\tau \bm{E}^{n+1} ) \nonumber \\ 
& + \frac{1}{2} b(\bm{u}_h^{n-1}, (\delta_\tau \bm{U}_h^{n+1} - \delta_\tau \bm{U}^{n+1})\cdot  \delta_\tau \errn)   = S^n_{2,1} + S^n_{2,2}. 
\end{align}
The term $S^n_{2,1}$ reads: 
\begin{align}
S^n_{2,1} & = \frac{1}{2} \sum_{E \in \mesh_h} \int_E (\nabla \cdot \bm{u}_h^{n-1}) (\delta_\tau \bm{U}_h^{n+1} - \delta_\tau \bm{U}^{n+1})\cdot \delta_\tau \bm{E}^{n+1} \nonumber \\ & \quad - \frac{1}{2}\sum_{e \in \Gamma_h \cup \partial \Omega} \int_e \{(\delta_\tau \bm{U}_h^{n+1} - \delta_\tau \bm{U}^{n+1})\cdot \delta_\tau \bm{E}^{n+1}\} [ \bm{u}_h^{n-1}] \cdot \bm{n}_e 
\end{align}
We use Holder's inequality, trace inequalities \eqref{eq:trace_ineq_continuous}-\eqref{eq:trace_ineq_discrete}, and inverse estimates \eqref{eq:inverse_estimate_dG}. We obtain: 
\begin{align}
&|S^n_{2,1}|  \leq C h^{-1}\|\bm{u}_h^{n-1}\| \| \delta_\tau \bm{U}_h^{n+1} - \delta_\tau \bm{U}^{n+1}\| \|\delta_\tau \bm{E}^{n+1}\|_{L^\infty(\Omega)} \nonumber \\ 
& + C\|\delta_\tau \bm{E}^{n+1} \|_{L^\infty(\Omega)}(h^{-1} \| \delta_\tau \bm{U}_h^{n+1} - \delta_\tau \bm{U}^{n+1}\| + \|  \nabla_h(\delta_\tau \bm{U}_h^{n+1} - \delta_\tau \bm{U}^{n+1})\|)\|\bm{u}_h^{n-1}\|.   \nonumber
\end{align}
Similar arguments to \eqref{eq:error_dg_aux} yield: 
\begin{equation}
    \|\delta_\tau \bm{U}_h^{n+1} - \delta_\tau \bm{U}^{n+1}\| \leq Ch^2\|\delta_\tau \bm{\chi}^{n+1}\| \nonumber
\end{equation}
With \eqref{eq:conv_delta_dual}, the above bound, and the stability result in Lemma \ref{lemma:first_err_estimate}, we obtain 
\begin{multline}
|S^n_{2,1}|\leq C (1+1/\mu)^{1/2} h \|\delta_\tau \bm{\chi}^{n+1}\|\|\delta_\tau \bm{E}^{n+1} \|_{L^{\infty}(\Omega)} \\ \leq C (1+1/\mu)^{1/2} h \|\delta_\tau \bm{\chi}^{n+1}\| \|\delta_\tau \bm{u}^{n+1} \|_{H^2(\Omega)}. 
\end{multline}
The expression for $S_{2,2}$ is: 
\begin{align}
S^n_{2,2} &= \frac{1}{2} \sum_{E \in \mesh_h} \int_E (\nabla \cdot \bm{u}_h^{n-1}) (\delta_\tau \bm{U}_h^{n+1} - \delta_\tau \bm{U}^{n+1})\cdot \delta_\tau \errn \nonumber \\ & \quad  - \frac{1}{2}\sum_{e \in \Gamma_h \cup \partial \Omega} \int_e \{(\delta_\tau \bm{U}_h^{n+1} - \delta_\tau \bm{U}^{n+1})\cdot \delta_\tau \errn \} [ \bm{u}_h^{n-1}] \cdot \bm{n}_e. 
 \end{align}
With Holder's, trace estimates \eqref{eq:trace_ineq_continuous}-\eqref{eq:trace_ineq_discrete}, and inverse inequalities \eqref{eq:inverse_estimate}, we obtain 
\begin{align}
    &|S^n_{2,2}| \leq C(h^{-1} \| \bm{e}_h^{n-1}\|+\|\nabla_h \Pi_h \bm{u}^{n-1}\|) \| \delta_{\tau} \bm{U}_h^{n+1} - \delta_\tau \bm{U}^{n+1}\|_{L^3(\Omega)} \| \delta_\tau \errn\|_{L^6(\Omega)} \nonumber \\ &   + C(h^{-d/6}\|\bm{e}_h^{n-1}\| + \|\Pi_h \bm{u}^{n-1}\|_{L^3(\Omega)})(h^{-1}\|\delta_{\tau} \bm{U}_h^{n+1} - \delta_\tau \bm{U}^{n+1}\|\nonumber \\ & \quad   + \| \nabla_h (\delta_{\tau} \bm{U}_h^{n+1} - \delta_\tau \bm{U}^{n+1})\|) \| \delta_\tau \errn\|_{L^6(\Omega)} \nonumber 
\end{align}
%With the triangle inequality, approximation property, \eqref{eq:regularity_delta}, \eqref{eq:H2_delta_bound}, and \eqref{eq:conv_delta_dual}, we have 
With Poincare's inequality, we have 
\[
\| \delta_{\tau} \bm{U}_h^{n+1} - \delta_\tau \bm{U}^{n+1}\|_{L^3(\Omega)} 
\leq C \| \delta_{\tau} \bm{U}_h^{n+1} - \delta_\tau \bm{U}^{n+1}\|_{\DG}
\leq C h \| \delta_\tau \bm{\chi}^{n+1}\|  
\]
%\begin{align}
%\| \delta_{\tau} \bm{U}_h^{n+1} - \delta_\tau \bm{U}^{n+1}\|_{L^3(\Omega)} & \leq Ch^{-d/6} \|\delta_{\tau} \bm{U}_h^{n+1} - \Pi_h \delta_\tau \bm{U}^{n+1} \| + \| \Pi_h \delta_{\tau} \bm{U}^{n+1} - \delta_\tau \bm{U}^{n+1}\|_{L^3(\Omega)} \nonumber \\ 
%& \leq  Ch^{-d/6}h^2(| \delta_\tau \bm{U}^{n+1}|_{H^2(\Omega)} + \| \delta_\tau \bm{e}_h^{n+1}\|) + Ch |\delta_{\tau } \bm{U}^{n+1} |_{W^{1,3}(\Omega)} \nonumber \\
%& \leq C h \| \delta_\tau \bm{e}_h^{n+1}\| \nonumber 
%\end{align}
Hence, with \eqref{eq:discreter_poincare} and \eqref{eq:conv_delta_dual}, we have 
\begin{align}
| S^n_{2,2} | & \leq C \|\bm{e}_h^{n-1}\|\| \delta_\tau \bm{\chi}^{n+1}\| \|\delta_\tau \errn \|_{\DG} \nonumber\\  & \quad + C (|\bm{u}^{n-1}|_{H^1(\Omega)} + \Vert \bm{u}^{n-1} \Vert_{W^{1,3}(\Omega)}) h \|\delta_\tau \bm{\chi}^{n+1}\|\|\delta_\tau \errn\|_{\DG}. 
\end{align}
With the assumption that $\bm{u} \in L^{\infty}(0,T;H^{k+1}(\Omega)^d)$, we obtain a bound for $S^n_2$: 
\begin{multline}
    |S^n_2| \leq C (\|\bm{e}_h^{n-1}\| +  h )\|\delta_\tau \bm{\chi}^{n+1}\|\|\delta_\tau \errn \|_{\DG} \\ + C\left(1+1/\mu\right)^{1/2} h \| \delta_{\tau} \bm{\chi}^{n+1}\|  \|\delta_\tau \bm{u}^{n+1}\|_{H^2(\Omega)} .
\end{multline}
For $S^n_3$, we use the definition \eqref{eq:inflow_boundary} and write $\bm{u}_h^{n-1} = \bm{e}_h^{n-1} + \Pi_h \bm{u}^{n-1}$. We also use that $[\delta_\tau \bm{U}^{n+1}] = \bm{0}$ a.e on $e \in \Gamma_h \cup \partial \Omega$.   
% \begin{align}
% & S_3   =   -  \sum_{E \in \mesh_h} \int_{\partial E_{-}^{\bm{u}_h^{n-1}}\backslash \partial \Omega} \{ \bm{e}_h^{n-1}  \} \cdot \bm{n}_{E}  ((\delta_\tau \bm{U}_h^{n+1} -\delta_\tau \bm{U}^{n+1})^{\mathrm{int}} - (\delta_\tau \bm{U}_h^{n+1} - \delta_\tau \bm{U}^{n+1})^{\mathrm{ext}})\cdot (\delta_\tau \bm{E}^{n+1} + \delta_\tau \errn)^{\mathrm{ext}} \nonumber \\ & \quad  -   \sum_{E \in \mesh_h} \int_{\partial E_{-}^{\bm{u}_h^{n-1}}\backslash \partial \Omega} \{ \Pi_h \bm{u}^{n-1} \} \cdot \bm{n}_E ((\delta_\tau \bm{U}_h^{n+1})^{\mathrm{int}} - (\delta_\tau \bm{U}_h^{n+1})^{\mathrm{ext}})\cdot (\delta_\tau \bm{E}^{n+1} + \delta_\tau \errn)^{\mathrm{ext}}  
% \end{align}
Hence, we bound $S^n_{3}$ as follows. 
 \begin{align}
& |S^n_{3}| \leq C \|\bm{e}_h^{n-1}\|_{L^3(\Omega)} \Vert \delta_\tau \bm{U}_h^{n+1} - \delta_\tau \bm{U}^{n+1} \Vert_{\DG} \|\delta_{\tau } \errn\|_{L^6(\Omega)}   \nonumber \\ &  + C \|\delta_\tau \bm{E}^{n+1}\|_{L^{\infty}(\Omega)} \|\bm{e}_h^{n-1}\|\|\delta_\tau \bm{U}_h^{n+1}\|_{\DG}   \nonumber \\ & + C\|\Pi_h \bm{u}^{n-1} \|_{L^\infty(\Omega)} \| \delta_\tau \bm{U}_h^{n+1} \|_{\DG} (\| \delta_\tau \bm{E}^{n+1}\| + h\|\nabla_h (\delta_\tau \bm{E}^{n+1})\|+ \| \delta_\tau \errn\|). %  + 
%C \| \Pi_h \bm{u}^{n-1}\|_{L^{\infty}(\Omega)} \|\delta_\tau \bm{U}_h^{n+1}\|_{\DG} \| \delta_\tau \errn\|. 
\end{align}
With \eqref{eq:boundEn}, \eqref{eq:boundEngrad}, \eqref{eq:relating_errn_eh_deltaphi}, \eqref{eq:inverse_estimate},  \eqref{eq:discreter_poincare}, \eqref{eq:stability_linf_pih}, \eqref{eq:approximation_prop_1}-\eqref{eq:approximation_prop_2},  and \eqref{eq:conv_delta_dual}, we obtain  
%\blue 
%For the second to last term, we apply approximation properties and the definition of $E^{n+1}$,  we have: 
%\begin{align*}
 %  \| \delta_\tau \bm{E}^{n+1}\| + h\|\nabla_h (\delta_\tau \bm{E}^{n+1})\| =     \| \Pi_h (\delta_\tau \bm{u}^{n+1}) - \delta_\tau \bm{u}^{n+1}\| \\ + h\|\nabla_h (\Pi_h (\delta_\tau \bm{u}^{n+1}) - \delta_\tau \bm{u}^{n+1})\|  \leq Ch^{k+1}|\delta_\tau \bm{u}^{n+1}|_{H^{k+1}(\Omega)} 
%\end{align*}
%\Bk
\begin{align}
&|S^n_3|  \leq C \|\bm{e}_h^{n-1} \| \| \delta_\tau \bm{\chi}^{n+1} \| \|\delta_\tau \errn\|_{\DG} +C \Vert \delta_\tau \bm{u}^{n+1} \Vert_{H^2(\Omega)} \|\bm{e}_h^{n-1}\|\|\delta_\tau \bm{U}_h^{n+1}\|_{\DG} \nonumber \\ &  + C h^{k+1} \vertiii{\bm{u}^{n-1}} \|\delta_{\tau } \bm{U}_h^{n+1} \|_{\DG} |\delta_\tau \bm{u}^{n+1} |_{H^{k+1}(\Omega)}\nonumber \\ & + C \vertiii{\bm{u}^{n-1}} \| \delta_\tau \bm{U}_h^{n+1}\|_{\DG}( \| \delta_\tau \bm{e}_h^{n+1} \| + \tau |\delta_\tau \phi_h^{n+1}|_{\DG}).
\end{align}
% With \eqref{eq:relating_errn_eh_deltaphi}, an inverse estimate \eqref{eq:inverse_estimate},  \eqref{eq:discreter_poincare}, \eqref{eq:stability_linf_pih}, \eqref{eq:approximation_prop_1}-\eqref{eq:approximation_prop_2},  and \eqref{eq:conv_delta_dual} % \Rd and another bound on $L^2$\Bk, we obtain  
% %\blue 
% For the second to last term, we apply approximation properties and the definition of $E^{n+1}$,  we have: 
% \begin{align*}
%    \| \delta_\tau \bm{E}^{n+1}\| + h\|\nabla_h (\delta_\tau \bm{E}^{n+1})\| =     \| \Pi_h (\delta_\tau \bm{u}^{n+1}) - \delta_\tau \bm{u}^{n+1}\| \\ + h\|\nabla_h (\Pi_h (\delta_\tau \bm{u}^{n+1}) - \delta_\tau \bm{u}^{n+1})\|  \leq Ch^{k+1}|\delta_\tau \bm{u}^{n+1}|_{H^{k+1}(\Omega)} 
% \end{align*}
% \Bk
% \begin{align}
% &|S_3|  \leq C \|\bm{e}_h^{n-1} \| \| \delta_\tau \bm{e}_h^{n+1} \| \|\delta_\tau \errn\|_{\DG} +C \|\delta_\tau \bm{E}^{n+1}\|_{L^{\infty}(\Omega)}\|\bm{e}_h^{n-1}\|\|\delta_\tau \bm{U}_h^{n+1}\|_{\DG} \nonumber \\ &  + C h^{k+1} \vertiii{\bm{u}^{n-1}} \|\delta_{\tau } \bm{U}_h^{n+1} \|_{\DG} |\delta_\tau \bm{u}^{n+1} |_{H^{k+1}(\Omega)} + C \vertiii{\bm{u}^{n-1}} \| \delta_\tau \bm{U}_h^{n+1}\|_{\DG}( \| \delta_\tau \bm{e}_h^{n+1} \| + \tau |\delta_\tau \phi_h^{n+1}|_{\DG}).
% \end{align}
The term $S^n_4$ is treated in a similar way to $S^n_3$ and the above bound also holds for $S^n_4$. We now combine the above bounds, use the assumption that $\bm{u} \in L^{\infty}(0,T; H^{k+1}(\Omega))$  and obtain: 
\begin{align}
& |I^n_{1,1}| \leq C  \|\bm{e}_h^{n-1}\|\| \delta_\tau \bm{\chi}^{n+1} \|( h^k \vert \delta_\tau \bm{u}^{n+1}\vert_{H^{k+1}(\Omega)}  + \| \delta_\tau \errn \|_{\DG}) \nonumber \\
& + C h \| \delta_\tau \bm{\chi}^{n+1}\|( h^{k+1} \vert \delta_\tau \bm{u}^{n+1}\vert_{H^{k+1}(\Omega)}   
 +  \| \delta_\tau \bm{e}_h^{n+1}\|) \nonumber \\ 
& +C h\| \delta_\tau \bm{\chi}^{n+1}\| (\| \delta_\tau \errn\|_{\DG} + (1+ 1/\mu)^{1/2}
\Vert \delta_\tau\bm{u}^{n+1}\Vert_{H^2(\Omega)} + \tau |\delta_\tau \phi_h^{n+1}|_{\DG})  \nonumber \\&+ C \|\delta_\tau \bm{U}_h^{n+1}\|_{\DG}(\|\bm{e}_h^{n-1}\| \Vert \delta_\tau\bm{u}^{n+1}\Vert_{H^2(\Omega)}  + \| \delta_\tau \bm{e}_h^{n+1}\| \nonumber \\ & \quad  + \tau |\delta_\tau \phi_h^{n+1}|_{\DG} + h^{k+1}| \delta_\tau \bm{u}^{n+1}|_{H^{k+1}(\Omega)}). \nonumber
    \end{align}
Bound \eqref{eq:boundI11} is obtained by using \eqref{eq:bounddeltatauchi} and applying Young's inequality in the above bound . 
\end{proof}
    \bibliographystyle{spmpsci}      % mathematics and physical sciences
    \bibliography{references}   % name your BibTeX data base

    \end{document}